\documentclass[final]{siamltex}
\usepackage{graphics,graphicx,epsf,subfigure,epstopdf}
\usepackage{amsmath,amsxtra,amsfonts,amscd,amssymb,bm}
\usepackage{algorithm}
\usepackage{algpseudocode}
\usepackage[top=2.5cm,bottom=2.5cm,right=2.5cm,left=2.5cm]{geometry}
\usepackage[mathscr]{eucal}
\usepackage{color}
\usepackage{tabularx}
\usepackage{makecell, booktabs, caption}
\usepackage{longtable}
\numberwithin{equation}{section}
\graphicspath{ {./newpic/} }
\usepackage{multirow}

\def\argmin{{\rm argmin}}

\def\eqnok#1{(\ref{#1})}

\newcommand{\tsum}{\textstyle\sum}
\newcommand{\bbe}{\mathbb{E}}

\newcommand{\bbr}{\Bbb{R}}
\newcommand{\nn}{\nonumber}

\def\vgap{\vspace*{.1in}}

\newcommand{\beq}{\begin{equation}}
\newcommand{\eeq}{\end{equation}}
\newcommand{\beqa}{\begin{eqnarray}}
\newcommand{\eeqa}{\end{eqnarray}}
\newcommand{\beqas}{\begin{eqnarray*}}
\newcommand{\eeqas}{\end{eqnarray*}}
\def\eqnok#1{(\ref{#1})}

\title{ Conditional Gradient Methods for Convex Optimization\\
with General Affine and Nonlinear Constraints
\thanks{
This research was partially supported by the ONR grant N00014-20-1-2089 and NSF grant CCF 1909298.}}
\author{
Guanghui Lan
	\thanks{H. Milton Stewart School of Industrial and Systems Engineering, Georgia Institute of Technology, Atlanta, GA, 30332 .
		(email: {\tt george.lan@isye.gatech.edu}).}
	\and
H. Edwin Romeijn
	\thanks{H. Milton Stewart School of Industrial and Systems Engineering, Georgia Institute of Technology, Atlanta, GA, 30332 .
		(email: {\tt edwin.romeijn@isye.gatech.edu}).}
	\and
Zhiqiang Zhou
    \thanks{H. Milton Stewart School of Industrial and Systems
    Engineering, Georgia Institute of Technology, Atlanta, GA, 30332.
    (email: {\tt zzhoubrian@gatech.edu}). }
}

\date{\today}

\begin{document}
\maketitle
\begin{abstract}
Conditional gradient methods have attracted much attention in both machine learning and optimization communities recently. 
These simple methods can guarantee the generation of sparse solutions. In addition, without the computation of full gradients, 
they can handle huge-scale problems sometimes even with an exponentially increasing number of decision variables.  
This paper aims to significantly expand the application areas of these methods
by presenting new conditional gradient methods for solving convex optimization problems with general affine and nonlinear constraints. 
More specifically, we first present a new constraint extrapolated condition gradient (CoexCG) method
that can achieve an ${\cal O}(1/\epsilon^2)$ iteration complexity for both smooth and structured nonsmooth 
function constrained convex optimization. We further develop novel variants of CoexCG, namely
constraint extrapolated and dual regularized conditional gradient (CoexDurCG) methods, that can 
achieve similar iteration complexity to CoexCG but allow adaptive selection for algorithmic parameters.
We illustrate the effectiveness of these methods for solving an important class of radiation therapy treatment planning problems arising from healthcare industry.
To the best of our knowledge, all the algorithmic schemes and
their complexity results are new in the area of projection-free methods.
\end{abstract}

\section{Introduction}

In this paper, we focus on the development of conditional gradient type methods
for solving the following convex optimization problem:
\begin{equation}\label{general}
\begin{aligned}
\min\quad & f(x) \\
\text{ s.t. }\quad & g(x) := Ax -b = 0,    \\
&  h_i(x) \leq 0, \ \  i =1,\ldots, d, \\
&  x\in X.
\end{aligned}
\end{equation}
Here $X \subseteq \bbr^n$ is a compact convex set, $f: X \to \bbr$ and $h_i: X \to \bbr$, $i = 1, \ldots, d$, are proper lower semicontinuous 
convex functions, $A: \bbr^n \to \bbr^m$ denotes a linear mapping, and $b$
is a given vector in $\bbr^m$. 
We assume that $X$ is relatively simple in the sense that one can minimize a linear function over $X$ easily.
Throughout this paper we assume that an optimal solution $x^*$ of problem \eqnok{general} exists.
For notational convenience, we often denote $h(x) \equiv (h_1(x); \ldots, h_d(x))$.

The conditional gradient method, initially developed by 
Frank and Wolfe in 1956~\cite{FrankWolfe56-1}, 
is one of the earliest first-order methods
for convex optimization. It has been widely used
for solving problems with relatively simple convex sets, i.e.,
when the constraints $g(x)=0$ and $h_i(x) \le 0$ do not appear in problem \eqnok{general}.
 Each iteration of this method
 computes the gradient of $f$ at the current search point $x_k$, and then
 solves the subproblem $\min_{x \in X} \langle \nabla f(x_k), x \rangle$
 to update the solution. In comparison with most other first-order methods,
 it does not require the projection over $X$, which in many cases
 could be computationally more expensive than to minimize a linear function over $X$
(e.g.. when $X$ is a spectrahedron given by $X:=\{ X \succeq 0: {\rm Tr}(X) =1\}$).
 These simple methods can also guarantee the generation of sparse solutions, e.g., when $X$ is a simplex or spectrahedron.
 In addition, without the computation of full gradients, they can handle huge-scale problems sometimes 
even with an exponentially increasing number of decision variables.  
 
 Much recent research effort has been devoted to the complexity analysis
of conditional gradient methods over simple convex set $X$. It is well-known that if $f$ is a smooth convex function, then
this algorithm can find an $\epsilon$-solution (i.e., a point $\bar x \in X$ s.t. $f(\bar x) - f^* \le \epsilon$)
in at most ${\cal O} (1/ \epsilon)$ iterations (see \cite{Jaggi13,Jaggi10,lan2013complexity,Freund13,HarJudNem12-1}). In fact, such a complexity result
has been established for the conditional gradient method under a stronger termination criterion called Wolfe Gap, based on
the first-order optimality condition~\cite{Jaggi13,Jaggi10,lan2013complexity,Freund13,HarJudNem12-1}. 
As shown in~\cite{Jaggi13,lan2013complexity,guzman2015lower}, this ${\cal O}(1/\epsilon)$ iteration complexity bound is tight
for smooth convex optimization. In addition,  if $f$ is a nonsmooth function with a saddle point structure,
one can not achieve an iteration complexity better than ${\cal O}(1/\epsilon^2)$~\cite{lan2013complexity}, in terms of the number of
times to solve the linear optimization subproblem. One possible way to improve the complexity bounds
is to use the conditional gradient sliding methods developed in~\cite{lan2014conditional} 
to reduce the number of gradient evaluations. 
Many other variants of conditional gradient methods have also been proposed
in the literature (see, e.g.,\cite{AhiTodd13-1,Bach12-1,BeckTeb04-1,Clarkson10,Freund13,Hazan08,HarJudNem12-1,Jaggi13,Jaggi10,LussTeb13-1,ShGosh11,ShenKim12-1, jiang2019structured, braun2017lazifying, gonccalves2017newton})
and Chapter 7 of \cite{LanBook2020} for an overview of these methods).
 
It should be noted, however, that none of the existing 
conditional gradient methods can be used to efficiently solve the
more general function constrained convex optimization problem in \eqnok{general}. With these function constraints ($g(x) =0$ and
$h_i(x) \le 0$), linear optimization over the feasible region of problem \eqnok{general} could become much more difficult.
As an example, if $X$ is the aforementioned spectrahedron and $h$ does not exist, 
the linear optimization problem over the feasible region $\{x \succeq 0: g(x) = 0, {\rm Tr}(X) =1\}$ becomes
a general semidefinite programming problem. Adding nonlinear function constraints $h_i(x) \le 0$ usually makes
the subproblem even harder. In fact, our study has been directly motivated by a
convex optimization problem with nonlinear function constraints arising
from radiation therapy treatment planning (see \cite{drzymala1991dose, mayles2007handbook, romeijn2005column, goitein2007radiation, men2009gpu, men2010ultrafast, men2007exact} and Section~\ref{sec_num} for more details).
The objective function of this problem, representing the quality of the treatment plan,
is smooth and convex. Besides a simplex constraint, it consists of
two types of nonlinear function constraints, namely the group sparsity constraint to reduce
radiation exposure for the patients, and the risk averse constraints to
avoid overdose (resp., underdose) to healthy (resp., tumor) structures.
This problem is highly challenging because the dimension of the decision variables
 can increase exponentially with respect to the size of data, which prevents the computation of
 full gradients as required by most existing optimization methods dealing with function constraints.
 
 %causes extremely high computational cost to compute the first-order information when implement most other first-order methods. Moreover, from the point of view in clinical purpose, the solution is expected to be sparse in order to minimizing the duration of the radiation treatment. Therefore, our goal is to develop a new type of the CG algorithm to solve the class of problem \eqref{general}.
%
%
%with function constraints as 
%However, we are interested in the problem \eqref{general} in this paper. This type of problem, especially with certain sparsity requirement on the solution, has many applications in a lot of areas, including finance, machine learning and clinical therapy etc, e.g., matrix completion, low rank approximation. In particular, if  $h(x)$ does not exist, and $X = \{x\succcurlyeq 0, \tsum \lambda(x) = 1\}$, it is a classical semidefinite programming problem. 
%
%And our direct motivation problem is a radiation therapy treatment plan problem, called IMRT problem. 
This paper aims to fill in the aforementioned gap in the literature by presenting 
a new class of conditional gradient methods for solving problem~\eqnok{general}.
Our main contributions are briefly summarized as follows. 
Firstly, inspired by the constraint-extrapolation (ConEx) method for function constrained
convex optimization in \cite{BoobDengLan19-1}, we develop a novel constraint-extrapolated conditional gradient (CoexCG) method
for solving problem \eqnok{general}. While both methods are single-loop 
primal-dual type methods for solving convex optimization problems with function constraints,
CoexCG only requires us to minimize a linear function, rather than to perform projection,
over $X$. %, and thus it can avoid the computation of full gradients. 
%In CoexCG we also need to carefully
%design the linear approximation of the objective and constraint functions, in order to guarantee
%the convergence of this method. 
In the basic setting when both $f$ and $h_i$ 
are smooth convex functions with Lipschitz continuous gradients,
we show that the total number of iterations performed by CoexCG before finding a $\epsilon$-solution of problem \eqref{general},
 i.e., a point $\bar x \in X$ s.t. $f(\bar x) - f(x^*) \leq \epsilon$ and $\|g(\bar x)\|_2 + \|[h(\bar x)]_+\|_2 \leq \epsilon$, can be bounded by ${\cal O}(1/\epsilon^2)$. 
 Here $[\cdot]_+ := \max\{\cdot, 0\}$.

Secondly, we consider more general function constrained optimization problems where either the objective
function $f$ or some constraint functions $h_i$ are possibly nondifferentiable, but contains
certain saddle point structure. We extend the CoexCG method for solving these problems in
combination with the well-known Nesterov's smoothing scheme~\cite{Nest05-1}. In general, even equipped with such smoothing technique, 
nonsmooth optimization is more difficult than smooth optimization, and its associated
iteration complexity is worse than that for smooth ones by orders of magnitude. However, we show that a similar ${\cal O}(1/\epsilon^2)$
complexity bound can be achieved by CoexCG for solving these nonsmooth
function constrained optimization problems. This seemly surprising result can be attributed to an inherent
acceleration scheme in CoexCG that can reduce the impact of the Lipschitz constants
induced by the smoothing scheme.

Thirdly, one possible shortcoming of CoexCG exists in that it requires
the total number of iterations $N$ fixed a priori before we run the algorithm
in order to achieve the best rate of convergence.
Therefore it is inconvenient to implement this algorithm
when such an iteration limit is not available.
In order to address this issue, we propose a constraint-extrapolated and dual-regularized conditional gradient (CoexDurCG)
method by adding a diminishing regularization term for the dual updates.
This modification allows us to design a novel adaptive stepsize policy which does not
require $N$ given in advance. Moreover, we show that the complexity of CoexDurCG
is still in the same order of magnitude as CoexCG with a slightly larger constant factor.
We also extend CoexDurCG for solving the aforementioned structured nonsmooth problems,
and demonstrate that  it is not necessary to explicitly define the smooth approximation
problem. We note that this technique of adding a diminishing
regularization term can be applied  for solving problems
with either unbounded primal feasible region (e.g., stochastic subgradient descent~\cite{NJLS09-1} and stochastic accelerated
gradient descent~\cite{Lan10-3}), or unbounded dual feasible region (e.g., ConEx \cite{BoobDengLan19-1}),
for which one often requires the number of iterations fixed in advance.

Finally, we apply the developed algorithms for solving
the radiation therapy treatment planning problem
on both randomly generated instances and a real data set.
We show that CoexDurCG performs comparably to CoexCG in terms 
of solution quality and computation time. 
%and may slightly outperform the latter on the initial phase of these algorithms
%for large-scale problems. 
We demonstrate that the incorporation of function constraints
 helps us  not only to
find feasible treatment plans satisfying clinical criteria, but also generate
alternative treatment plans that can possibly
reduce radiation exposure time for the patients. 

To the best of our knowledge, all the algorithmic schemes as well as
their complexity results are new in the area of projection-free methods for convex optimization.

This paper is organized as follows. Section~\ref{sec_coexCG} 
is devoted to the CoexCG method. We first present the CoexCG method
for smooth function constrained convex optimization in Subsection~\ref{sec_smooth_basic}
and extend it for solving structured nonsmooth function constrained convex optimization in Subsection~\ref{sec_nonsmooth_basic}.
We then discuss the CoexDurCG method in Section~\ref{sec_adaptive}, including
its basic version for smooth function constrained convex optimization in Subsection~\ref{sec_ada_smooth} and
its extended version for directly solving structured nonsmooth function constrained convex optimization problems in Subsection~\ref{sec_ada_nonsmooth}.
We apply these methods for radiation therapy treatment planning in Section~\ref{sec_num}, and conclude
the paper with a brief summary in Section~\ref{sec_remark}.
% and showed that when $f$ and $h$ are smooth or even some structured nonsmooth functions, the number of iterations required by the CoexCG to find a $\epsilon$-solution of problem \eqref{general}, i.e., the number of iterations to find a point $\bar x \in X$, s.t. $f(\bar x) - f(x^*) \leq \epsilon$ and $\|g(\bar x)\|_2 + \|[h(\bar x)]_+\|_2 \leq \epsilon$, is bounded by ${\cal O}(1/\epsilon^2)$. 
%
%Secondly, we developed the constraint-extrapolated and dual-regularized conditional gradient (CoexDurCG) method, which does not require the priori information of the total number of iterations $N$, and also showed that the number of iterations required by the CoexDurCG to find a $\epsilon$-solution of problem \eqref{general} is bounded by ${\cal O}(1/\epsilon^2)$ when $f$ and $h$ are smooth or even some structured nonsmooth functions. 
%
%Finally, the promising numerical results have been provided to demonstrate the performance in terms of objective function values and constraints violation when implementing CoexCG and CoexDurCG method to the IMRT problem.

\section{Constraint-extrapolated conditional gradient method} \label{sec_coexCG}
In this section, we present a basic version of the constraint-extrapolated conditional gradient method
for solving convex optimization problem~\eqref{general}.
Subsection~\ref{sec_smooth_basic} focuses on the
case when $f$ and $h_i$ are smooth convex functions,
while subsection~\ref{sec_nonsmooth_basic}
extends our discussion to the situation where
$f$ and $h_i$ are not necessarily differentiable.

\subsection{Smooth functions} \label{sec_smooth_basic}
Throughout this subsection, we assume that $f$ and $h_i$ are differential and their gradients are
Lipschitz continuous s.t. 
\begin{align}
\|\nabla f(x_1) - \nabla f(x_2)\|_* &\le L_f \| x_1 - x_2 \|,  \ \forall x_1, x_2 \in X, \label{eq:def_smooth_f}\\
\|\nabla h_i(x_1) - \nabla h_i(x_2) \|_* &\le L_{h,i} \|x_1 - x_2 \|, \forall x_1, x_2 \in X, i = 1, \ldots, d. \label{eq:def_smooth_h}
\end{align}
Here $\|\cdot\|$ denotes an arbitrary norm which is not necessarily associated with the inner product $\langle \cdot, \cdot \rangle$
($\|\cdot\|_*$ is the conjugate norm of $\|\cdot\|$). For notational convenience, we denote
\[
L_h = (L_{h,1}; \ldots; L_{h,d}) \ \ \mbox{and} \ \ \bar L_h = \| L_h\|_2.
\]

We need to use the Lipschitz continuity of the constraint function $h_i$ when developing
conditional gradient methods for function constrained problems.
Clearly, under the boundedness assumption of $X$, the constraint functions $h_i$ are Lipschitz continuous with constant $M_{h,i}$, i.e.,
\begin{equation} \label{eq:def_Lip_h}
\| \nabla h_i(x) \|_* \leq M_{h,i}, \ \forall x \in X. 
\end{equation}
In particular, letting $x^*$ be an optimal solution of problem \eqref{general},
we have
$
M_{h,i} \le \nabla f(x^*) + L_{h,i} D_X,
$
where $D_X$ denotes the diameter of $X$ given by
\begin{equation}\label{eq:def_D_X}
D_X := \max_{x_1, x_2 \in X} \|x_1 - x_2\|.
\end{equation}
Note that a different way to bound on $M_{h,i}$ will be discussed for certain structured nonsmooth problems in Subsection~\ref{sec_nonsmooth_basic}.
For the sake of notational convenience, we also denote
\beq \label{eq:def_Lip_h_combined}
\bar M_h = \sqrt{\tsum_{i=1}^d M_{h,i}^2}.
\eeq

%{\bf Assumption 1:} all $f(\cdot)$ and $h_i(\cdot)$ are smooth convex functions with parameters $L_f$ and $L_{h,i}$, respectively. \\
%And $h(\cdot)$ are Lipschitz continuous functions with parameters $M_h$, i.e.,
%\begin{align}
%&\| \nabla h(x) \|_* \leq M_{h,i} .
%\end{align}

Since we can only perform linear optimization over the feasible region $X$,
one natural way to solve problem \eqnok{general} is to 
consider its saddle point reformulation
\begin{equation}\label{minmax}
\min_{x\in X}\max_{y \in \bbr^m, z \in \bbr^d_+} f(x)+\langle g(x), y \rangle + \langle h(x),z \rangle.
\end{equation}
Throughout the paper, we assume that the standard Slater condition holds for problem \eqnok{general}
so that a pair of optimal dual solutions $(y^*, z^*)$ of problem~\eqnok{minmax} exists.

In \cite{Nest05-1}, Nesterov proposed a novel smoothing scheme
to solve a general bilinear saddle point problem when the term $\langle h(x), z\rangle$ does not exist in \eqnok{minmax}. 
More specifically, he suggested to apply an accelerated gradient method
to solve a smooth approximation for this bilinear saddle point problem. Using this idea,
in \cite{lan2013complexity} (see also Chapter 7 of \cite{LanBook2020}), Lan
presented a smoothing conditional gradient method by
appling the conditional gradient algorithm for a properly smoothed version of the objective function of \eqnok{minmax}.
%
%The proposed method, however, intends to solve a more complex saddle point problem in that (a) it has a nonlinear coupling term, i.e., $\langle h(x), z \rangle$;
%(b) the feasible set for the dual problems are unbounded. 
%As a consequence, existing acceleration schemes, such as Nesterov's accelerated gradient method
%and his smoothing scheme, can not be directly applied anymore for this problem.
%This necessitates the development of the CoexCG and ConEx methods.
However, this scheme is not applicable
for our setting due to the following reasons. 
Firstly, the smoothing conditional gradient method only solves
bilinear saddle point problems with linear coupling terms given by $\langle g(x), y \rangle$ and cannot deal with
the nonlinear coupling term $\langle h(x), z \rangle$. Secondly,
even for the bilinear saddle point problems, the smoothing conditional gradient method in \cite{LanBook2020,lan2013complexity}
requires the feasible set of $y$ to be bounded, which does not hold for problem~\eqnok{minmax}.

Our development has been inspired the constraint extrapolation (ConEx) method
recently introduced by Boob, Deng and Lan~\cite{BoobDengLan19-1}
for solving problem~\eqnok{minmax}. ConEx is an accelerated primal-dual type method which 
updates both the primal variable $x$ and dual variables $(y, z)$ in each iteration.
In comparison with some previously developed accelerated primal-dual methods for
solving saddle point problems with nonlinear coupling terms~\cite{Nem05-1,aybat2018primal},
one distinctive feature of ConEx is that it defines the 
acceleration (or momentum) step by extrapolating  the linear approximation
of the nonlinear function $h$. As a consequence, it can deal with unbounded
feasible regions for the dual variable $z$ (or $y$) and 
thus solve the function (or affine) constrained convex optimization problems.
However, each iteration of the ConEx method requires the projection onto the feasible region $X$, and hence 
is not applicable to our problem setting.

In order to address the above issues for solving 
problem~\eqnok{general} (or \eqnok{minmax}), we
present a novel constraint-extrapolated conditional gradient (CoexCG) method, which
incorporates some basic ideas of the ConEx method into the conditional gradient method.
As shown in Algorithm~\ref{algCoexCG}, the CoexCG method first performs  in \eqnok{eq:step_ex1} an
extrapolation step for the affine constraint $g$. Then in \eqnok{eq:step_ex2} it performs
 an extrapolation step based on the linear approximation of the constraint function $h$
given by
\begin{align}
l_{h_i}(\bar x, x) := h_i(\bar x) + \langle \nabla h_i(\bar x), x - \bar x\rangle, \label{eq:def_l_h_i} \\
l_h(\bar x, x) := (l_{h_1}(\bar x, x); \ldots, l_{h_d}(\bar x, x)). \label{eq:def_l_h}
\end{align}
Utilizing the extrapolated constraint values $\tilde g_k$ and $\tilde h_k$,
it then updates the dual variables $q_k$ and $r_k$ associated with the affine
constraint $g(x) = 0$ and the nonlinear constraints $h(x) \le 0$ in 
\eqnok{eq:step_dual1} and \eqnok{eq:step_dual2}, respectively.
With these updated dual variables and linear approximation
$l_f(x_{k-1}, x)$ and $l_h(x_{k-1}, x)$, it solves a linear optimization problem over $X$
to update the primal variable $p_k \in X$ in \eqnok{eq:step_lo}.
Finally, the output solution $x_k$ is computed as a convex combination of $x_{k-1}$
and $p_k$ in \eqnok{eq:step_out_x}.

\begin{algorithm}[H]
\caption{{\bf Co}nstraint-{\bf ex}trapolated {\bf C}onditional {\bf G}radient (CoexCG)}\label{algCoexCG}
\begin{algorithmic} %[1]
\State Let the initial points $p_0 = p_{-1} \in X$, $x_0 = x_{-1} = x_{-2} \in X$,
$q_0 \in \bbr^m$ and $r_0 \in \bbr^d_+$ be given. Also let
the stepsize parameters $\lambda_k \ge 0$, $\tau_k \ge 0$ and $\alpha_k \in [0,1]$ be given.
\For{$k = 1$ \textbf{ to } $N$}
\begin{align}
\tilde g_k &= g(p_{k-1})+\lambda_k  [g(p_{k-1})-g(p_{k-2})], \label{eq:step_ex1}\\
\tilde h_k &= l_h(x_{k-2},p_{k-1})+\lambda_k [l_h(x_{k-2},p_{k-1})-l_h(x_{k-3},p_{k-2})], \label{eq:step_ex2}\\
q_k &= \argmin_{y \in \bbr^m} \{ \langle -\tilde g_k, y\rangle + \tfrac{\tau_k}{2}\|y-q_{k-1}\|_2^2 \}, \label{eq:step_dual1}\\
r_k &= \argmin_{ z \in \bbr^d_+} \{ \langle -\tilde h_k, z\rangle + \tfrac{\tau_k}{2}\|z-r_{k-1}\|_2^2 \}, \label{eq:step_dual2}\\
p_k &= \argmin_{x\in X} \{ l_f(x_{k-1}, x) + \langle g(x), q_k \rangle +  \langle l_h(x_{k-1},x),  r_k\rangle \}, \label{eq:step_lo}\\
x_k &= (1-\alpha_k)x_{k-1}+ \alpha_k p_k. \label{eq:step_out_x}
\end{align}
\EndFor
\end{algorithmic}
\end{algorithm}

%{\color{blue} 
Similar to the game interpretation developed in \cite{lan2015optimal,LanBook2020} for Nesterov's accelerated gradient method~\cite{Nest83-1},
the CoexCG method can be viewed as an iterative game performed by the primal and dual players to achieve an equilibrium of 
\eqnok{minmax}. The extrapolation steps in \eqnok{eq:step_ex1}-\eqnok{eq:step_ex2} are used to predict
the possible action (or its consequences) of the primal player in each iteration. Based on the prediction $(\tilde g_k, \tilde h_k)$,
the dual player updates the decision $q_k$ (resp., $r_k$) in order to maximize the profit $\langle \tilde g_k, y\rangle$ (resp., 
 $\langle \tilde h_k, z\rangle$), but not to move too far away from the previous decision $q_{k-1}$ (resp., $r_{k-1}$)
 by using the regularization $\tfrac{\tau_k}{2}\|y-q_{k-1}\|_2^2$ (resp., $\tfrac{\tau_k}{2}\|z-r_{k-1}\|_2^2$).
 After observing the dual player's decisions $(q_k, r_k)$,
 the primal player first determines $p_k$ in a greedy manner by minimizing the cost $l_f(x_{k-1}, x) + \langle g(x), q_k \rangle +  \langle l_h(x_{k-1},x),  r_k\rangle$,
 and then takes a correction step in \eqnok{eq:step_out_x} so that its decision $x_k$ is not dramatically different from the previous decision $x_{k-1}$.
 Similar to Nesterov's method as interpreted in \cite{lan2015optimal,LanBook2020},
 CoexCG employs an intelligent dual player who predicts the other player's decision before taking actions.
 However, the primal updates in \eqnok{eq:step_lo} and \eqnok{eq:step_out_x} for CoexCG are different from
 those in \cite{lan2015optimal,LanBook2020} since no projection is allowed, even though the spirit of
 not moving too far away from the previous decision $x_{k-1}$ remains the same.
An interesting observation to us is that, due to the lack of the projection for the primal player, the incorporation of the extrapolation (or prediction) steps of 
the dual player appears to be important to guarantee the convergence of the algorithm (see the discussion after Proposition~\ref{Prop_CoexCG} for more details).
%}

It is interesting to build some connections between the CoexCG method and the
ConEx method in \cite{BoobDengLan19-1}.
In particular, by replacing the relations in \eqnok{eq:step_lo} and \eqnok{eq:step_out_x}
with 
\[
p_k = \argmin_{x\in X} \{ l_f(p_{k-1}, x) + \langle g(x), q_k \rangle +  \langle l_h(p_{k-1},x),  r_k\rangle + \tfrac{\eta_k}{2} \|x - p_{k-1}\|_2^2 \},
\]
then we essentially obtain the ConEx method. % possibly in a slightly more general form (by allowing the existence of the affine constraint $g(x) = 0$).
Comparing these relations, we observe that the CoexCG method differs from the ConEx method in the following few aspects.
Firstly, $p_t$ in CoexCG is computed by solving a linear optimization problem, while
the one in the ConEx method is computed by using a projection. 
%Moreover, in the CoexCG we do not need to select the algorithmic parameter $\eta_t$ as in the ConEx method. 
The use of linear optimization enables the CoexCG method to generate
sparse solutions in feasible sets $X$ with a huge large number of
extreme points (see Section~\ref{sec_num}).
Secondly, the linear approximation models $l_f$ and $l_h$ in the ConEx method is built on the search point $p_{k-1}$,
while the one in the CoexCG method is built on $x_{k-1}$, or equivalently, the convex combination of
all previous search points $p_i$, $i=1, \ldots, k-1$. 
%{\color{blue}
Using $l_f(x_{k-1}, x)$ and $l_h(x_{k-1}, x)$ in CoexCG  instead of $l_f(p_{k-1}, x)$ and $l_h(p_{k-1}, x)$ 
as in ConEx also seems to be critical to guarantee the convergence of the CoexCG algorithm. %}

We need to add a few more remarks about the CoexCG method.
Firstly, by \eqnok{eq:step_dual1} and \eqnok{eq:step_dual2}, we can define
$q_k$ and $r_k$ equivalently as
\begin{align*}
q_k = q_{k-1} + \tfrac{1}{\tau_k} \tilde g_k \ \ \mbox{and} \ \
r_k = \max\{r_{k-1} + \tfrac{1}{\tau_k} \tilde h_k, 0\}.
\end{align*}
It is also worth noting that we can generalize the CoexCG method to
deal with conic inequality constraint $h(x) \in {\cal K}$, by simply replacing the constraint
 $z \in \bbr^d_+$ in \eqnok{eq:step_dual2} with $z \in -{\cal K}^*$. Here
 ${\cal K} \subset \bbr^l$ is a given closed convex cone and 
${\cal K}^*$ denotes its the dual cone.
 
Secondly, in addition to the primal output solution $x_k$ in \eqnok{eq:step_out_x},
we can also define the dual output solutions $y_k$ and $z_k$ as
\begin{align}
y_{k} &= (1-\alpha_k)y_{k-1}+ \alpha_k q_k, \label{eq:step_out_y}\\
z_{k} &= (1-\alpha_k)z_{k-1}+ \alpha_k r_k. \label{eq:step_out_z}
\end{align}
Different from $x_k$, these dual variables $y_k$ and $z_k$ do not participate
in the updating of any other search points. However,
both of them will be used intensively in the convergence analysis
of the CoexCG method.

Thirdly, even though we do not need to select the parameter $\eta_k$ when defining $p_k$ as in the ConEx method,
we do need to specify the stepsize parameter $\tau_k$ to update the dual variables
$q_k$ and $r_k$. We also need to determine the parameters $\lambda_k$ 
and $\alpha_k$, respectively, to define the extrapolation steps and
the output solution $x_k$. We will discuss the selection of these algorithmic
parameters after establishing some general convergence properties
of the CoexCG method.

\vgap

Our goal in the remaining part of this subsection is to establish
the convergence of the CoexCG method.
Let $x_k, y_k$, and $z_k$ be defined in \eqnok{eq:step_out_x}, \eqnok{eq:step_out_y},
and \eqnok{eq:step_out_z}. Throughout this section,
we denote $w_k \equiv (x_k,y_k,z_k)$ and $w \equiv (x,y,z)$, and define the gap function $Q(w_k,w)$ as
\beq \label{eq:def_gap_function}
Q(w_k,w) := f(x_k)-f(x) +\langle g(x_k),y \rangle -\langle g(x),y_k \rangle +\langle h(x_k),z \rangle -\langle h(x),z_k \rangle.
\eeq

We start by stating some well-known technical results that have been used in
the convergence analysis of many first-order methods.
The first result, often referred to ``three-point lemma" (see, e.g., Lemma 3.1 of \cite{LanBook2020}), 
characterizes the optimality conditions of
\eqnok{eq:step_dual1} and \eqnok{eq:step_dual2}.

\begin{lemma} \label{lemma_projection}
Let $q_{k}$ and $r_k$ be defined in \eqnok{eq:step_dual1} and \eqnok{eq:step_dual2},
respectively. Then,
\begin{align}
 \langle -\tilde g_k, q_{k} - y \rangle + \tfrac{\tau_k}{2} \|q_{k} - q_{k-1}\|_2^2 
 \le \tfrac{\tau_k}{2} \|y - q_{k-1}\|_2^2 - \tfrac{\tau_k}{2} \|y - q_{k}\|_2^2, \forall y \in \bbr^m, \label{opt_cond_dual1}\\
  \langle -\tilde h_k, r_{k} - z \rangle + \tfrac{\tau_k}{2} \|r_{k} - r_{k-1}\|_2^2 
 \le \tfrac{\tau_k}{2} \|z - r_{k-1}\|_2^2 - \tfrac{\tau_k}{2} \|z - r_{k}\|_2^2, \forall z \in \bbr^d_+.  \label{opt_cond_dual2}
\end{align}
\end{lemma}

The following result helps us to take telescoping sums (see Lemma 3.17 of \cite{LanBook2020}).
\begin{lemma}\label{Lemma 2}
Let $\alpha_k\in(0,1], k = 0,1,2,\ldots$, be given and denote 
\begin{equation}\label{Def_Gamma}
\Gamma_k = \left\{
                     \begin{array}{ll}
                       1, & \hbox{if $k=1$;} \\
                       (1-\alpha_k)\Gamma_{k-1}, & \hbox{if $k>1$.}
                     \end{array}
                   \right.
\end{equation}
If $\{\Delta_k\}$ satisfies
$\Delta_{k+1}\leq (1-\alpha_k)\Delta_k + B_k, \forall k \geq 1,$
then we have
$\tfrac{\Delta_{k+1}}{\Gamma_k}\leq (1-\alpha_1)\Delta_1 + \tsum_{i=1}^k\tfrac{B_i}{\Gamma_i}.$
\end{lemma}

We now establish an important recursion of the CoexCG method.
\begin{proposition}\label{Prop_CoexCG}
For any $k>1$, we have
\begin{align*}
Q(w_k,w) &\leq  (1-\alpha_k) Q(w_{k-1},w)+ \tfrac{(L_f+z^T L_h)\alpha_k^2 D_X^2}{2} +\tfrac{\alpha_k\lambda_k^2 (9 \bar M_h^2+\|A\|^2) D_X^2}{2 \tau_k}\\
&\quad + \alpha_k [ \langle A(p_k-p_{k-1}), y- q_k\rangle - \lambda_k\langle A (p_{k-1}-p_{k-2}), y -q_{k-1}\rangle]\\
& \quad + \alpha_k[\langle l_h(x_{k-1},p_k)-l_h(x_{k-2},p_{k-1}), z- r_k\rangle - \lambda_k \langle l_h(x_{k-2},p_{k-1})-l_h(x_{k-3},p_{k-2}), z- r_{k-1}\rangle] \\
&\quad  + \tfrac{\alpha_k\tau_k}{2}[\|y-q_{k-1}\|_2^2 -\|y-q_k\|_2^2 + \|z-r_{k-1}\|_2^2 -\|z-r_k\|_2^2], \ \forall w \in X \times \bbr^m \times \bbr^d_+,
\end{align*}
where $D_X$ is defined in \eqnok{eq:def_D_X}.
\end{proposition}

\begin{proof}
It follows 
from the smoothness of $f$ and $h$ (e.g., Lemma 3.2 of \cite{LanBook2020}) 
and the definition of $x_k$ in \eqnok{eq:step_out_x}
that
\begin{align*}
f(x_k)&\leq  l_f(x_{k-1}, x_k) + \tfrac{L_f}{2}\|x_k-x_{k-1}\|^2 \\
 &=  (1-\alpha_k) l_f(x_{k-1}, x_{k-1}) + \alpha_k l_f(x_{k-1}, p_k) + \tfrac{L_f\alpha_k^2}{2}\|p_k-x_{k-1}\|^2 \\
 & =  (1-\alpha_k) f(x_{k-1}) +\alpha_kl_f(x_{k-1}, p_k)+ \tfrac{L_f\alpha_k^2}{2}\|p_k-x_{k-1}\|^2 . \\
 h_i(x_k) &\leq  (1-\alpha_k)h_i(x_{k-1}) + \alpha_k l_{h_i}(x_{k-1},p_k) + \tfrac{L_{h,i} \alpha_k^2}{2}\|p_k - x_{k-1}\|^2.
\end{align*}
Using the above two relations in the definition of $Q(w_k,w)$ in \eqnok{eq:def_gap_function}, we have for any $w \equiv (x, y, z) \in X \times \bbr^m \times \bbr^d_+$,
\begin{align*}
Q(w_k,w) &= f(x_k)-f(x) +\langle g(x_k),y \rangle -\langle g(x),y_k \rangle +\langle h(x_k),z \rangle -\langle h(x),z_k \rangle \\
&\le  (1-\alpha_k)  f(x_{k-1}) + \alpha_k l_f(x_{k-1}, p_k)
-f(x) +\langle g(x_k),y \rangle -\langle g(x),y_k \rangle \\
&\quad +\langle (1-\alpha_k)h(x_{k-1}) + \alpha_k l_{h}(x_{k-1},p_k)),z \rangle -\langle h(x),z_k \rangle\\
& \quad+ \tfrac{(L_f + z^T L_h) \alpha_k^2}{2} \|p_k - x_{k-1}\|^2 \\
&= (1-\alpha_k) Q(w_{k-1},w) + \tfrac{(L_f + z^T L_h) \alpha_k^2}{2} \|p_k - x_{k-1}\|^2 \\
&\quad +\alpha_k[ l_f(x_{k-1}, p_k) - f(x) +  \langle g(p_k) ,y\rangle - \langle g(x) ,q_k\rangle 
+ \langle l_h(x_{k-1},p_k) ,z\rangle - \langle h(x) ,r_{k}\rangle ].
\end{align*}
Moreover, by the definition of $x_k$ in \eqnok{eq:step_out_x} and the convexity of $f$ and $h_i$, we have
\begin{align*}
&l_f(x_{k-1}, p_k) + \langle g(p_k), q_k\rangle + \langle l_h(x_{k-1},p_k), r_k\rangle\\
&\le l_f(x_{k-1}, x) + \langle g(x), q_k\rangle + \langle l_h(x_{k-1},x), r_k\rangle\\
&\le f(x) + \langle g(x), q_k\rangle + \langle h(x), r_k\rangle, \ \forall x \in X.
\end{align*}
Combining the above two relations, we obtain
\begin{align}
Q(w_k,w) &\le  (1-\alpha_k) Q(w_{k-1},w) + \tfrac{(L_f + z^T L_h) \alpha_k^2}{2} \|p_k - x_{k-1}\|^2 \nn \\
&\quad + \alpha_k [\langle g(p_k), y - q_k\rangle + \langle l_h(x_{k-1},p_k) ,z - r_k \rangle ] \nn \\
&\le 1-\alpha_k) Q(w_{k-1},w) + \tfrac{(L_f + z^T L_h) \alpha_k^2 D_X^2}{2} \nn \\
&\quad + \alpha_k [\langle g(p_k), y - q_k\rangle + \langle l_h(x_{k-1},p_k) ,z - r_k \rangle ], \ \forall w \in X \times \bbr^m \times \bbr^d_+. \label{eq:recursion_on_gap_function}
\end{align}
Multiplying both sides of \eqnok{opt_cond_dual1} and \eqnok{opt_cond_dual2} by $\alpha_k$ and
summing them up with the above inequality, we
have 
\begin{align}
Q(w_k,w) & \leq  (1-\alpha_k) Q(w_{k-1},w)+ \tfrac{(L_f+z^TL_h)\alpha_k^2 D_X^2}{2} \nn \\
&\quad +\alpha_k \langle g(p_k)- \tilde g_k), y- q_k\rangle +\alpha_k \langle l_h(x_{k-1},p_k)- \tilde h_k, z - r_k\rangle\nn \\
&\quad + \tfrac{\alpha_k\tau_k}{2}[\|y-q_{k-1}\|_2^2 -\|y-q_k\|_2^2-\|q_k-q_{k-1}\|_2^2]  \nn \\
 &\quad + \tfrac{\alpha_k\tau_k}{2}[\|z-r_{k-1}\|_2^2 -\|z-r_k\|_2^2-\|r_k-r_{k-1}\|_2^2], \ \forall w \in X \times \bbr^m \times \bbr^d_+. \label{eq:Coex_temp1}
\end{align}
Now observe that by the definition of $\tilde g_k$ in \eqnok{eq:step_ex1} and
the fact that $g(x) = Ax - b$, we have
\begin{align}
&\langle g(p_k)- \tilde g_k), y- q_k\rangle -\tfrac{\tau_k}{2} \|q_k-q_{k-1}\|_2^2 \nn \\
 &= \langle A[(p_k - p_{k-1})-\lambda_k (p_{k-1}-p_{k-2})], y - q_k \rangle-\tfrac{\tau_k}{2} \|q_k-q_{k-1}\|_2^2 \nn \\
& = \langle A(p_k-p_{k-1}), y -q_k\rangle - \lambda_k\langle A (p_{k-1}-p_{k-2}), y - q_{k-1}\rangle\nn \\
& \quad + \lambda_k\langle A (p_{k-1}-p_{k-2}), q_k- q_{k-1}\rangle-\tfrac{\tau_k}{2} \|q_k-q_{k-1}\|_2^2 \nn \\
&\leq  \langle A(p_k-p_{k-1}), y- q_k\rangle - \lambda_k\langle A (p_{k-1}-p_{k-2}), y - q_{k-1}\rangle \nn\\
& \quad +\tfrac{\lambda_k^2}{2\tau_k}\|A\|^2\|p_k-p_{k-1}\|_2^2 \nn \\
&\le \langle A(p_k-p_{k-1}), y- q_k\rangle - \lambda_k\langle A (p_{k-1}-p_{k-2}), y - q_{k-1}\rangle + \tfrac{\lambda_k^2}{2\tau_k}\|A\|^2 D_X^2, \label{eq:Coex_temp2}
\end{align}
where the first inequality follows from Young's inequality and the last one follows from the definition of $D_X$ in \eqnok{eq:def_D_X}.
In addition, by the definition of $\tilde h_k$ in \eqnok{eq:step_ex2}, we have
\begin{align}
&\langle l_h(x_{k-1},p_k)- \tilde h_k, z - r_k\rangle  -\tfrac{\tau_k}{2} \|r_k-r_{k-1}\|_2^2 \nn \\
& \leq  \langle l_h(x_{k-1},p_k)- l_h(x_{k-2},p_{k-1}), z - r_k\rangle -\lambda_k\langle l_h(x_{k-2},p_{k-1})- l_h(x_{k-3},p_{k-2}), z - r_{k-1}\rangle \nn \\
& \quad + \lambda_k\langle l_h(x_{k-2},p_{k-1})- l_h(x_{k-3},p_{k-2}), r_k - r_{k-1}\rangle - \tfrac{\tau_k}{2} \|r_k-r_{k-1}\|_2^2 \nn\\
& \leq \langle l_h(x_{k-1},p_k)- l_h(x_{k-2},p_{k-1}), z - r_k\rangle \nn \\
& \quad -\lambda_k\langle l_h(x_{k-2},p_{k-1})- l_h(x_{k-3},p_{k-2}), z - r_{k-1}\rangle +\tfrac{9\lambda_k^2 \bar M_h^2 D_X^2}{2\tau_k}, \label{eq:Coex_temp3}
\end{align}
where the last inequality follows from
\begin{align}
&\lambda_k\langle l_h(x_{k-2},p_{k-1})- l_h(x_{k-3},p_{k-2}), r_k - r_{k-1}\rangle - \tfrac{\tau_k}{2} \|r_k-r_{k-1}\|_2^2 \nn \\
&\le \tfrac{\lambda_k^2}{2\tau_k} \tsum_{i=1}^d [ l_{h_i}(x_{k-2},p_{k-1})- l_{h_i}(x_{k-3},p_{k-2})]^2 \nn \\
&= \tfrac{\lambda_k^2}{2\tau_k}\tsum_{i=1}^d [ h_i(x_{k-2})-h_i(x_{k-3}) + \langle \nabla h_i(x_{k-2}),p_{k-1}-x_{k-2}\rangle + \langle \nabla h_i(x_{k-3}),p_{k-2}-x_{k-3}\rangle]^2 \nn \\
&\le \tfrac{9\lambda_k^2D_X^2}{2\tau_k}\tsum_{i=1}^dM_{h,i}^2 
= \tfrac{9\lambda_k^2 \bar M_h^2 D_X^2}{2 \tau_k}. \label{eq:Coex_temp4}
\end{align}
The result then follows by plugging relations \eqnok{eq:Coex_temp2} and \eqnok{eq:Coex_temp3} into \eqnok{eq:Coex_temp1}.
\end{proof}

\vgap

We add some comments about the importance of the extrapolation steps
in the proposed CoexCG method. Without these steps (i.e., $\lambda_k = 0$ in \eqnok{eq:step_ex1} and \eqnok{eq:step_ex2}),
probably we can not  even guarantee the convergence of the CoexCG algorithm.
As we can see from the proof of Proposition~\ref{Prop_CoexCG},
if $\lambda_k = 0$, it is not clear how to
take care of the inner product terms $\langle A(p_k - p_{k-1}, y - q_k\rangle$
and $\langle  l_h(x_{k-1},p_k)-l_h(x_{k-2},p_{k-1}), z- r_k\rangle$.
The error caused by these terms may accumulate.

\vgap

We are now ready to establish the main convergence properties for the CoexCG method.

\begin{theorem} \label{The_Coex_main}
Let $\Gamma_k$ be defined in \eqref{Def_Gamma} and assume that
the algorithmic parameters $\alpha_k, \tau_k$ and $\lambda_k$ in the CoexCG method satisfy
\begin{equation}\label{cond1}
\alpha_1 = 1, \ \tfrac{\lambda_k\alpha_k}{\Gamma_k} = \tfrac{\alpha_{k-1}}{\Gamma_{k-1}} \text{ and } \tfrac{\alpha_k\tau_k}{\Gamma_k} \leq \tfrac{\alpha_{k-1}\tau_{k-1}}{\Gamma_{k-1}},
\forall k \ge 2.
\end{equation}
Then we have
\begin{equation}\label{thm1_gap_result}
\begin{aligned}
Q(w_N,w) &\leq  \Gamma_N\tsum_{k=1}^N\left[\tfrac{(L_f+z^T L_h)\alpha_k^2D_X^2}{2\Gamma_k} +\tfrac{\alpha_k\lambda_k^2(9 \bar M_h^2+\|A\|^2)D_X^2}{2\tau_k\Gamma_k}\right] \\
&\quad +\tfrac{\alpha_N(9\bar M_h^2+\|A\|^2) D_X^2}{2\tau_N} +\tfrac{\tau_1\Gamma_N}{2}\|y-q_0\|_2^2 +\tfrac{\tau_1\Gamma_N}{2}\|z-r_0\|_2^2,
\ \forall w \in X \times \bbr^m \times \bbr^d_+,
\end{aligned}
\end{equation}
where $D_X$ is defined in \eqnok{eq:def_D_X}. As a consequence, we have
\begin{align}
f(x_N) - f(x^*) &\le  \Gamma_N\tsum_{k=1}^N\left[\tfrac{L_f \alpha_k^2D_X^2}{2\Gamma_k} +\tfrac{\alpha_k\lambda_k^2(9 \bar M_h^2+\|A\|^2)D_X^2}{2\tau_k\Gamma_k}\right] \nn \\
&\quad + \tfrac{\alpha_N(9\bar M_h^2+\|A\|^2) D_X^2}{2\tau_N} +\tfrac{\tau_1\Gamma_N}{2}(\|q_0\|_2^2 + \|r_0\|_2^2) \label{thm1_obj_result}
\end{align}
and
\begin{align}
\|g(x_N)\|_2 + \|[h(x_N)]_+\|_2 &\le   \Gamma_N\tsum_{k=1}^N\left[\tfrac{[L_f+(\|z^*\|_2+1) \bar L_h]\alpha_k^2D_X^2}{2\Gamma_k} +\tfrac{\alpha_k\lambda_k^2(9 \bar M_h^2+\|A\|^2)D_X^2}{2\tau_k\Gamma_k}\right]
+\tfrac{\alpha_N(9\bar M_h^2+\|A\|^2) D_X^2}{2\tau_N} \nn\\
&\quad + \tau_1\Gamma_N [(\|y^*\|_2+1)^2 + \|q_0\|_2^2 + (\|z^*\|_2+1)^2 + \|r_0\|_2^2],
\label{thm1_feas_result}
\end{align}
where $(x^*, y^*, z^*)$ denotes a triple of optimal solutions for problem~\eqnok{minmax}.
\end{theorem}

\begin{proof}
It follows from Lemma~\ref{Lemma 2} and Proposition~\ref{Prop_CoexCG}  that
\begin{align*}
\tfrac{Q(w_N,w)}{\Gamma_N} \leq&  (1-\alpha_1) Q(w_0,w)+ \tsum_{k=1}^N[\tfrac{(L_f+z^T L_h)\alpha_k^2 D_X^2 }{2\Gamma_k} +\tfrac{\alpha_k\lambda_k^2(9 \bar M_h^2 + \|A\|^2)D_X^2}{2\tau_k\Gamma_k}]\\
& + \tsum_{k=1}^N \tfrac{\alpha_k}{\Gamma_k} [ \langle A(p_k-p_{k-1}), y- q_k\rangle - \lambda_k \langle A (p_{k-1}-p_{k-2}), y - q_{k-1}\rangle]\\
& + \tsum_{k=1}^N\tfrac{\alpha_k}{\Gamma_k}[\langle l_h(x_{k-1},p_k)-l_h(x_{k-2},p_{k-1}), z - r_k\rangle  \\
& \quad \quad  - \lambda_k \langle l_h(x_{k-2},p_{k-1})-l_h(x_{k-3},p_{k-2}), z - r_{k-1}\rangle ]\\
& + \tsum_{k=1}^N\tfrac{\alpha_k\tau_k}{2\Gamma_k}[\|y-q_{k-1}\|_2^2 -\|y-q_k\|_2^2 + \|z-r_{k-1}\|_2^2 -\|z-r_k\|_2^2],
\end{align*}
which, in view of \eqref{cond1}, then implies that
\begin{align*}
Q(w_N,w) &\leq  \Gamma_N \tsum_{k=1}^N[\tfrac{(L_f+z^T L_h)\alpha_k^2 D_X^2 }{2\Gamma_k} +\tfrac{\alpha_k\lambda_k^2(9 \bar M_h^2 + \|A\|^2)D_X^2}{2\tau_k\Gamma_k}]\\
&\quad +\alpha_N \langle A(p_N-p_{N-1}), y-q_N\rangle -\tfrac{\alpha_N\tau_N}{2}\|y-q_N\|_2^2 \\
&\quad + \alpha_N\langle l_h(x_{N-1},p_N)-l_h(x_{N-2},p_{N-1}),z-r_N\rangle -\tfrac{\alpha_N\tau_N}{2}\|z-r_N\|_2^2 \\
&\quad  + \tfrac{\alpha_1\tau_1\Gamma_N}{2}[\|y-q_{0}\|_2^2 + \|z-r_{0}\|_2^2 ]\\
& \leq   \Gamma_N \tsum_{k=1}^N[\tfrac{(L_f+z^T L_h)\alpha_k^2 D_X^2 }{2\Gamma_k} +\tfrac{\alpha_k\lambda_k^2(9 \bar M_h^2 + \|A\|^2)D_X^2}{2\tau_k\Gamma_k}]\\
&\quad +\tfrac{\alpha_N}{2\tau_N} \|A\|^2\|p_N-p_{N-1}\|_2^2 +\tfrac{9 \bar M_h^2 \alpha_N D_X^2}{2\tau_N}\\
& \quad  + \tfrac{\alpha_1\tau_1\Gamma_N}{2}[\|y-q_{0}\|_2^2+ \|z-r_{0}\|_2^2],
\end{align*}
where the last relation follows from Young's inequality and 
a result similar to \eqnok{eq:Coex_temp4}. The result in \eqnok{thm1_gap_result} then immediately follows from the above inequality.

Note that by the definition of $Q(w_k,w)$ in \eqnok{eq:def_gap_function},
and the facts that $g(x^*) = 0$ and $h(x^*) \le 0$,
we have $f(x_N) - f(x^*) \le Q(w_N, (x^*,0,0))$.
Using this observation and fixing $x = x^*, y=0, z = 0$ in \eqnok{thm1_gap_result},
we obtain \eqnok{thm1_obj_result}.
Now let us denote
\begin{align}
\hat y_N &:= (\|y^*\|_2+1) \tfrac{g(x_N)}{\|g(x_N)\|_2}, \label{eq:def_hat_y}\\
\hat z_N &:= (\|z^{*}\|_2+1)\tfrac{[h(x_N)]_+}{\|[h(x_N)]_+\|_2}, \label{eq:def_hat_z}\\
\hat w^*_N &:= (x^*, \hat y_N,\hat z_N). \label{eq:def_hat_w}
\end{align}
Note that by the optimality condition of \eqnok{minmax}, we have
\begin{align*}
 0\leq Q(w_N, w^*) &=   f(x_N) -f(x^*) +\langle  g(x_N), y^{*}\rangle +\langle h(x_N), z^{*}\rangle  \\
& \leq   f(x_N)-f(x^*) + \|g(x_N)\|_2 \cdot \|y^{*}\|_2   + \|[h(x_N)]_+\|_2  \cdot \|z^{*}\|_2.
\end{align*}
In addition, using the fact that
 $g(x^*) = 0$ and $\langle h(x^*), \hat z_N\rangle \le 0$, 
we have
\begin{align*}
Q(w_N,\hat w^*_N) 
&\ge   f(x_N) -f(x^*) +\langle  g(x_N), \hat y_N \rangle + \langle  h(x_N),  \hat z_N \rangle  \\
&=   f(x_N)-f(x^*) +  \|g(x_N)\|_2 (\|y^{*}\|_2+1)  +  \|[h(x_N)]_+\|_2 (\|z^*\|_2+1).
\end{align*}
Combining the previous two observations, we conclude that
\beq \label{bnd_infeasibility}
\|g(x_N)\|_2 + \|[h(x_N)]_+\|_2  \leq Q(w_N,\hat w^*_N).
\eeq
The previous conclusion, together with \eqnok{thm1_gap_result} and the facts that
\begin{align}
\|\hat y_N -q_0\|_2^2  &\le 2 [\|\hat y_N\|_2^2 +  \|q_0\|_2^2] = 2 [(\|y^*\|_2+1)^2 + \|q_0\|_2^2], \label{eq:bound_hat_y}\\
\|\hat z_N -r_0\|_2^2 &\le 2[\hat z_N\|_2^2 +  \|r_0\|_2^2] = 2 [(\|z^*\|_2+1)^2 + \|r_0\|_2^2], \label{eq:bound_hat_z}\\
\hat z_N^T L_h &\le \|\hat z_N^T\|_2 \|L_h\|_2 = (\|z^*\|_2+1) \bar L_h, \label{eq:bound_hat_inner}
\end{align}
then imply that
\begin{align*}
&\|g(x_N)\|_2 + \|[h(x_N)]_+\|_2 \\
&\le  \Gamma_N\tsum_{k=1}^N\left[\tfrac{(L_f+\hat z_N^T L_h)\alpha_k^2D_X^2}{2\Gamma_k} +\tfrac{\alpha_k\lambda_k^2(9 \bar M_h^2+\|A\|^2)D_X^2}{2\tau_k\Gamma_k}\right] \\
&\quad + \tfrac{\alpha_N(9\bar M_h^2+\|A\|^2) D_X^2}{2\tau_N} +\tfrac{\tau_1\Gamma_N}{2}\|\hat y_N -q_0\|_2^2 +\tfrac{\tau_1\Gamma_N}{2}\|\hat z_N -r_0\|_2^2 \\
&\le  \Gamma_N\tsum_{k=1}^N\left[\tfrac{[L_f+(\|z^*\|_2+1) \bar L_h]\alpha_k^2D_X^2}{2\Gamma_k} +\tfrac{\alpha_k\lambda_k^2(9 \bar M_h^2+\|A\|^2)D_X^2}{2\tau_k\Gamma_k}\right] \nn\\
&\quad +\tfrac{\alpha_N(9\bar M_h^2+\|A\|^2) D_X^2}{2\tau_N} + \tau_1\Gamma_N [(\|y^*\|_2+1)^2 + \|q_0\|_2^2 + (\|z^*\|_2+1)^2 + \|r_0\|_2^2].
\end{align*}
\end{proof}

\vgap

Below we provide a specific selection of the algorithmic parameters $\alpha_k$, $\lambda_k$ and $\tau_k$
and establish the associated rate of convergence for the CoexCG method.

\begin{corollary}\label{Cor_Coex_main}
If the number of iterations $N$ is fixed a priori, and
\begin{equation}\label{para_Coex_main}
\alpha_k = \tfrac{2}{k+1}, \lambda_k = \tfrac{k-1}{k},\ \tau_k = \tfrac{N^{3/2}}{k}D_X \sqrt{9\|M_h\|^2+\|A\|^2}, k = 1, \ldots, N,
\end{equation}
then we have
\begin{align}
Q(w_N,w) &\leq \tfrac{2(L_f+z^T L_h)D_X^2}{N+1} + \tfrac{D_X \sqrt{9\bar M_h^2+\|A\|^2}}{\sqrt{N}}  \left( \|y-q_0\|_2^2 + \|z-r_0\|_2^2 + 1 \right),\nn \\
& \quad \quad \ \forall w \in X \times \bbr^m \times \bbr^d_+,  \label{cor_gap_result}\\
f(x_N) - f(x^*) &\le \tfrac{2L_f D_X^2}{N+1} +  \tfrac{ D_X \sqrt{9\bar M_h^2+\|A\|^2}}{\sqrt{N}}  ( \|(q_0;r_0)\|_2^2 +1 ), \label{cor_obj_result} \\
\|g(x_N)\|_2 + \|[h(x_N)]_+\|_2 &\le \tfrac{2[L_f+(\|z^*\|_2+1) \bar L_h] D_X^2}{N+1} \nn \\
&\quad +  \tfrac{ 2D_X \sqrt{9\bar M_h^2+\|A\|^2}}{\sqrt{N}}  [2\|(y^*;z^*)\|_2^2 + \|(q_0;r_0)\|_2^2  +5 ]. \label{cor_feas_result}
\end{align}
\end{corollary}

\begin{proof}
By \eqref{Def_Gamma} and  the definition of $\alpha_k$ in \eqnok{para_Coex_main},
we have 
$\Gamma_k = 2/[k(k+1)]$ and $\alpha_k/\Gamma_k = k$.
We can easily see from these identities and \eqnok{para_Coex_main} that 
the conditions in \eqref{cond1} hold. It is also easy to verify that
\begin{align*}
\tsum_{k=1}^N \tfrac{\alpha_k^2}{\Gamma_k} &= 2 \tsum_{k=1}^N\tfrac{k}{k+1} \le 2N,\\
\tsum_{k=1}^N \tfrac{\alpha_k \lambda_k^2}{\tau_k \Gamma_k} &= \tfrac{\tsum_{k=1}^N (k-1)^2}{2 N^{3/2} D_X \sqrt{9\|M_h\|^2+\|A\|^2}}  \le  \tfrac{N^{3/2}}{6  D_X \sqrt{9\|M_h\|^2+\|A\|^2}}.
\end{align*}
Using these relations in \eqnok{thm1_gap_result}, \eqnok{thm1_obj_result} and \eqnok{thm1_feas_result}, we conclude that
\begin{align*}
Q(w_N,w) &\leq \tfrac{2(L_f+z^T L_h) D_X^2}{N+1}+   \tfrac{\sqrt{N} D_X \sqrt{9 \bar M_h^2+\|A\|^2}}{6(N+1) } \\
& \quad + \tfrac{  D_X \sqrt{9\bar M_h^2+\|A\|^2}}{(N+1) \sqrt{N}} +\tfrac{\sqrt{N}D_X \sqrt{9\bar M_h^2+\|A\|^2}}{N+1} (\|y-q_0\|_2^2 + \|z-r_0\|_2^2)  \\
&= \tfrac{2(L_f+z^T L_h) D_X^2}{N+1} +  \left[ \tfrac{\sqrt{N} }{6(N+1) } + \tfrac{1}{(N+1) \sqrt{N}} + \tfrac{\sqrt{N}}{N+1}(\|y-q_0\|_2^2 + \|z-r_0\|_2^2)   \right]D_X \sqrt{9\bar M_h^2+\|A\|^2}\\
%&\le  \tfrac{(L_f+z^T L_h) D_X^2}{N+1} +  \left( \tfrac{1}{2} + \|y-q_0\|_2^2 + \|z-r_0\|_2^2  + \tfrac{1}{6} \right)\tfrac{D_X \sqrt{3\bar M_h^2+\|A\|^2}}{\sqrt{N}}\\
&\le  \tfrac{2(L_f+z^T L_h) D_X^2}{N+1} + \tfrac{D_X \sqrt{9\bar M_h^2+\|A\|^2}}{\sqrt{N}}  \left( \|y-q_0\|_2^2 + \|z-r_0\|_2^2 + 1 \right),\\
f(x_N) - f(x^*) %&\le 
%\tfrac{  D_X \sqrt{3\bar M_h^2+\|A\|^2}}{(N+1) \sqrt{N}}  +\tfrac{\sqrt{N}D_X \sqrt{3\bar M_h^2+\|A\|^2}}{N+1}(\|q_0\|_2^2 + \|r_0\|_2^2) 
 %+ \tfrac{L_f D_X^2}{N+1}+   \tfrac{\sqrt{N} D_X \sqrt{3 \bar M_h^2+\|A\|^2}}{6(N+1) } \\
%&=  \tfrac{L_f D_X^2}{N+1} + \left[ \tfrac{1}{(N+1) \sqrt{N}} + \tfrac{\sqrt{N}}{N+1}(\|q_0\|_2^2 + \|r_0\|_2^2)  + \tfrac{\sqrt{N} }{6(N+1) } \right]D_X \sqrt{3\bar M_h^2+\|A\|^2}\\
%&\le  \tfrac{L_f D_X^2}{N+1} +  \left( \tfrac{1}{2} + \|q_0\|_2^2 + \|r_0\|_2^2  + \tfrac{1}{6} \right)\tfrac{D_X \sqrt{3\bar M_h^2+\|A\|^2}}{\sqrt{N}}\\
&\le \tfrac{2L_f D_X^2}{N+1} +  \tfrac{ D_X \sqrt{9\bar M_h^2+\|A\|^2}}{\sqrt{N}}  ( \|q_0\|_2^2 + \|r_0\|_2^2 +1),
\end{align*}
and
\begin{align*}
&\|g(x_N)\|_2 + \|[h(x_N)]_+\|_2 \\
&\le   \tfrac{2[L_f+(\|z^*\|_2+1) \bar L_h] D_X^2}{N+1} +  \tfrac{\sqrt{N} D_X \sqrt{9 \bar M_h^2+\|A\|^2}}{6(N+1) } +\tfrac{ D_X \sqrt{9\bar M_h^2+\|A\|^2}}{(N+1) \sqrt{N}} \nn\\
&\quad + \tfrac{2\sqrt{N}D_X \sqrt{9\bar M_h^2+\|A\|^2}}{N+1} [(\|y^*\|_2+1)^2 + \|q_0\|_2^2 + (\|z^*\|_2+1)^2 + \|r_0\|_2^2] \nn\\
&\le \tfrac{2[L_f+(\|z^*\|_2+1) \bar L_h] D_X^2}{N+1} +  \tfrac{ D_X \sqrt{9\bar M_h^2+\|A\|^2}}{\sqrt{N}}  [1 +2(\|y^*\|_2+1)^2 +2 \|q_0\|_2^2 +2 (\|z^*\|_2+1)^2 +2 \|r_0\|_2^2 ] \\
&\le \tfrac{2[L_f+(\|z^*\|_2+1) \bar L_h] D_X^2}{N+1} +  \tfrac{ 2D_X \sqrt{9\bar M_h^2+\|A\|^2}}{\sqrt{N}}  [2(\|y^*\|_2^2+\|z^*\|_2^2) + \|q_0\|_2^2 + \|r_0\|_2^2 +5].
\end{align*}
\end{proof}

A few remarks about the results obtained in Theorem~\ref{The_Coex_main} and Corollary~\ref{Cor_Coex_main} are in place.
Firstly, in view of \eqnok{cor_gap_result}, the gap function $Q(w_N, w)$
converges to $0$ with the rate of convergence given by ${\cal O}(1/\sqrt{N})$. This bound
has been shown to be not improvable in  \cite{lan2013complexity} 
for general saddle point problems in terms of the number of calls
to linear optimization oracles (see also Chapter 7 of \cite{LanBook2020}),
even though such a lower complexity bound cannot be directly applied to our setting since we are dealing
with a specific saddle point problem with unbounded dual variables.
Secondly, in view of \eqnok{cor_obj_result} and \eqnok{cor_feas_result}, the number of iterations required
by the CoexCG method to find a $\epsilon$-solution of problem~\eqnok{general}, i.e., a point
$\bar x \in X$ s.t. $f(\bar x) - f(x^*) \le \epsilon$ and $\|g(\bar x)\|_2 + \|[h(\bar x)]_+\|_2 \le \epsilon$,
is bounded by ${\cal O} (1/\epsilon^2)$. 
Thirdly, it is interesting to observe that in both \eqnok{cor_obj_result} and \eqnok{cor_feas_result},
the Lipschitz constants $L_f$ and $\bar L_h$ do not impact too much the rate of
convergence of the CoexCG method, since both of them appear only in the non-dominant terms.
%{\color{blue}We note that without the extrapolation steps (i.e., $\lambda_k = 0$ in \eqnok{eq:step_ex1}
%and \eqnok{eq:step_ex2}), these terms would be $O(1/\sqrt{N})$
%and the size of Lipschtiz constant $L_f$ and $L_h$ would impact the rate of convergence of the algorithm.}
We will explore further this property of the CoexCG method in order to solve problems
with certain nonsmooth objective and constraint functions. 
Finally, it is worth noting that in the parameter setting \eqnok{para_Coex_main},
we need to fix the total number of iterations $N$ in advance. This is not desirable for 
the implementation of the CoexCG method, especially for the situation when
one has finished the scheduled $N$ iterations, but then realizes
that a more accurate solution is needed. In this case, one has to completely restart the CoexCG method
with a different parameter setting that depends on the modified iteration limit.
%However, without using this assumption it appears that we cannot
%achieve the best rate of convergence for the basic CoexCG method in Algorithm~\ref{algCoexCG}.
We will discuss how to address this issue in Section~\ref{sec_adaptive}.

\subsection{Structured nonsmooth functions}  \label{sec_nonsmooth_basic}
In this subsection, we still consider problem \eqnok{general}, but the objective function $f$ and constraint functions $h_i$
are not necessarily differentiable. More specifically, we assume that $f(\cdot)$ and $h_i(\cdot)$ are given in the following form:
\begin{equation}\label{Nonsmooth}
\begin{aligned}
& f(x) = \max_{q\in Q}\{\langle Bx,q\rangle - \hat f(q)\},\\
& h_i(x) = \max_{s\in S_i}\{\langle C_ix,s\rangle - \hat h_i(s)\}, i=1,\ldots, d,
\end{aligned}
\end{equation}
where $Q \subseteq \bbr^{m_0}$ and $S \subseteq \bbr^{m_i}$ are closed convex sets, and $\hat f$ and $\hat h_i$ are simple convex functions.
Many nonsmooth functions can be represented in this form (see \cite{Nest05-1}).
In this paper, we assume that $\hat f$ and $\hat h_i$ are possibly strongly convex w.r.t. the given norms in the respective spaces, i.e.. 
\begin{align}
\hat f (q_1) - \hat f(q_2) - \langle \hat f'(q_2), q_1 - q_2\rangle \ge \tfrac{\mu_0 }{2} \|q_1 - q_2\|^2, \forall q_1, q_2 \in Q\\
\hat h_i (s_1) - \hat h_i(s_2) - \langle \hat h'_i(s_2), s_1 - s_2\rangle \ge \tfrac{\mu_i }{2} \|s_1 - s_2\|^2, \forall s_1, s_2 \in S_i, i = 1, \ldots, d,
\end{align}
for some $\mu_i \ge 0$. 
If $\mu_0 > 0$ (resp., $\mu_i > 0$), then $f$ (resp., $h_i$) must be 
differentiable with Lipschitz continuous gradients. Therefore, our nonsmooth formulation in~\eqnok{Nonsmooth} allows
either the objective and/or some constraint functions to be smooth. 

Our goal in this subsection is to generalize the CoexCG method to solve these structured nonsmooth convex optimization problems.
In fact, we show that
 that the number of CoexCG iterations 
required to solve these problems is in the same order of magnitude as
if $f$ and $h_i$'s are smooth convex functions.

Since $f$ and $h_i$ are possibly not differentiable, we cannot directly apply the CoexCG algorithm
to solve problem~\eqnok{general}. However, as pointed out by Nesterov~\cite{Nest05-1},
these nonsmooth functions can be closely approximated by smooth convex ones.  
Let us first consider the objective function $f$.
Assume that $u: Q \rightarrow \bbr$ is a given strongly convex function with modulus $1$ w.r.t. a given norm $\|\cdot\|$ in $\bbr^{m_0}$, i.e.,
\[
u(q_1) \geq u(q_2) + \langle u'(q_2), q_1-q_2\rangle + \tfrac{1}{2} \|q_1-q_2\|^2, \forall q_1,q_2 \in Q.
\]
Let us denote $c_u := \argmin_{q \in Q} u(y)$, $U(q) := u(q) - u(c_u) -\langle \nabla u(c_u), q-c_u\rangle$ and
\begin{equation}\label{eq:def_DU}
D_U := [\max_{q \in Q} U(y)]^{1/2},
\end{equation}
and define
\beq \label{eq:def_smoothapp_f}
f_{\eta_0}(x) := \max_{q\in Q}\{\langle Bx,q\rangle - \hat f(q)-\eta_0 U(q)]\}
\eeq
for some $\eta_0 \ge 0$. 
Then, we can show that $f_{\eta_0}$ is differentiable and its gradients satisfy (see \cite{Nest05-1})
\beq \label{eq:smoothapp_f_Lip}
\| \nabla f_{\eta_0}(x_1) - \nabla f_{\eta_0}(x_2)\|_* \le  L_{f,\eta} \|x_1 - x_2\|, \ \forall x_1, x_2 \in X \ \
\mbox{with} \ \ L_{f,\eta} := \tfrac{\|B\|^2}{\mu_0 + \eta_0}.
\eeq
In addition, we have
\beq \label{eq:app_close_f}
 f_{\eta_0}(x) \leq f(x) \leq  f_{\eta_0}(x)  +\eta_0 D_U^2, \ \forall x \in X.
\eeq
In our algorithmic scheme, we will set $\eta_0 = 0$ whenever $\hat f$ is strongly convex, i.e., $\mu_0 > 0$.

Similarly, let us assume that $v_i: S_i \rightarrow \bbr$ are strongly convex with modulus $1$ w.r.t. a given norm $\|\cdot\|$ in $\bbr^{m_i}$, $i = 1, \ldots,d$.
Also let us denote $c_{v_i} := \argmin_{s \in S_i} v_i(s)$, $V_i(s) := v_i(s) - v_i(c_{v_i} ) -\langle \nabla v_i(c_{v_i} ), s-c_{v_i} \rangle$ and
\beq \label{eq:def_DV}
D_{V_i} := [\max_{s \in S_i} V_i(s)]^{1/2},
\eeq
and define 
\beq \label{eq:def_smoothapp_h}
h_{i, \eta_{i}}(x) = \max_{s\in S_i}\{\langle C_ix,s\rangle - \hat h_i(s)-\eta_{i} V_i(s)\}
\eeq
for some $\eta_i \ge 0$.
We can show that for all $i = 1, \ldots, d$,
\begin{align}
\| \nabla h_{i,\eta_i}(x_1) - \nabla h_{i, \eta_i}(x_2)\|_* \le  \tfrac{\| C_i\|^2}{\mu_i + \eta_i} \|x_1 - x_2\|, \ \forall x_1, x_2 \in X,\\
h_{i, \eta_i}(x) \leq h_i(x) \leq  h_{i, \eta_i}(x)  +\eta_i D_{V_i}^2, \ \forall x \in X. \label{eq:app_close_h}
\end{align}
In our algorithmic scheme, we will set $\eta_i = 0$ whenever $\hat h_i$ is strongly convex, i.e., $\mu_i > 0$.
For notational convenience, we denote
\beq \label{eq:def_app_h_lip_grad}
h_\eta(x) := (h_{1, \eta_1}(x); \ldots; h_{d, \eta_d}(x)), \
 L_{h,\eta} := (\tfrac{\| C_1\|^2}{\mu_{\hat h_1} + \eta_1}; \ldots; \tfrac{\| C_d\|^2}{\mu_{\hat h_d}+\eta_d} ) \ \ \mbox{and} \ \
 \bar L_{h, \eta} := \| L_{h,\eta}\|_2.
\eeq

Different from the objective function, we need to show that
the gradient of the $h_{i,\eta_i}$ is bounded. Note that the boundedness of the gradients for smooth constraint functions (with $\mu_i > 0$ and hence $\eta_i = 0$)
follows from the boundedness of $X$ (see Section~\ref{sec_smooth_basic}). For those nonsmooth
constraint functions $h_i$ (with $\mu_i = 0$), we need to 
assume that $S_i$'s are compact. 
For a given $x \in X$, let $s^*(x)$ be
the optimal solution of \eqnok{eq:def_smoothapp_h}. Then
\begin{align}
\|\nabla h_{i, \eta_{i}}(x)\|_* &= \|C_i^T \cdot s^*(x)\|_* \le \|C_i\| \|s^*(x)\| \nn \\
&\le \|C_i\| (\|c_{v_i}\| + \|s^*(x) - c_{v_i}\|) \nn \\
&\le  \|C_i\| (\|c_{v_i}\| + \sqrt{2}D_{V_i} ) =: M_{C_i, V_i}, i = 1, \ldots, d. \label{eq:def_Lips_MC}
\end{align}
For notational convenience, we also denote
\beq \label{eq:def_app_h_lip}
\bar M_{C,V} := \sqrt{\tsum_{i=1}^d M_{C_i, V_i}^2}.
\eeq
Observe that the Lipschitz constants  $M_{C_i, V_i}$ defined in \eqnok{eq:def_Lips_MC} 
do not depend on the smoothing parameters $\eta_i$, $i = 1, \ldots, d$.
This fact will be important for us to derive the complexity bound of the CoexCG method for solving convex optimization problems
with nonsmooth function constraints.

Instead of solving the original problem~\eqnok{general}, we suggest to apply the CoexCG method to
the smooth approximation problem
\begin{equation}\label{general2}
\begin{aligned}
\min\quad & f_{\eta_0}(x) \\
\text{ s.t. }\quad & g(x) = 0,\\
&  h_{i, \eta_i}(x) \leq 0, \forall i =1,\ldots, d,\\
&  x\in X.
\end{aligned}
\end{equation}
More specifically, we replace the linear approximation functions $l_h$ and $l_f$
used in \eqnok{eq:step_ex2} and \eqnok{eq:step_lo} by
$l_{h_{i, \eta_i}}$ and $l_{f_{\eta_o}}$, respectively.
However, we will establish the convergence of this method in terms of the solution of
the original problem in \eqnok{general} rather than the approximation problem in \eqnok{general2}. 
Our convergence analysis below exploits
the smoothness of $f_{\eta_0}$  (resp., $h_{i, \eta_i}$), the closeness between $f$ and $f_{\eta_0}$
(resp., $h_i$ and  $h_{i, \eta_i}$), and also importantly, the fact that
$h_{i, \eta_i}(x)$ underestimates $h_i(x)$ for all $x\in X$.

\vgap

\begin{theorem} \label{the_nonsmooth}
Consider the CoexCG method applied to the smooth approximation problem~\eqnok{general2}.
Assume that the number of iterations $N$ is fixed a priori, and that the parameters $\{\alpha_k\}, \{\tau_k\}$ and $\{\lambda_k\}$ are set to \eqref{para_Coex_main}
with $\bar M_h$ replaced by $\bar M_{C,V}$ in \eqnok{eq:def_app_h_lip}. 
Then we have
\begin{align}
f(x_N) - f(x^*) &\leq \tfrac{2L_{f,\eta} D_X^2}{N+1} + \tfrac{D_X \sqrt{9\bar M_{C,V}^2+\|A\|^2}}{\sqrt{N}}  \left( \|q_0\|_2^2 + \|r_0\|_2^2 + 1 \right) +\eta_0 D_U^2, \label{eq:nonsmooth_obj_result} \\ 
\|[h(x_N)]_+\| +\|Ax_N\| &\leq  \tfrac{2[L_{f,\eta}+(\|z^*\|_2+1) \bar L_{h,\eta}]) D_X^2}{N+1} + \tfrac{2D_X \sqrt{9\bar M_{C,V}^2+\|A\|^2}}{\sqrt{N}}  \left( 2 \|(y^*;z^*)\|_2^2 + \|(q_0;r_0)\|_2^2 + 5 \right) \nn \\
&\quad + \eta_0 D_U^2 + (\|z^*\|_2+1)(\tsum_{i=1}^d (\eta_i D_{V_i}^2)^2)^{1/2}, \label{eq:nonsmooth_feas_result}
\end{align}
where $(x^*, y^*, z^*)$ is a triple of optimal solutions for problem~\eqnok{minmax}, $L_{f,\eta}$ and $\bar L_{h, \eta}$ are defined in \eqnok{eq:smoothapp_f_Lip}
and \eqnok{eq:def_app_h_lip_grad}, respectively, and $D_X$, $D_U$ and $D_{V_i}$ are defined in \eqnok{eq:def_D_X}, \eqnok{eq:def_DU} and \eqnok{eq:def_DV},
respectively.
\end{theorem}

\begin{proof}
Denote $Q_\eta (w_N,w) := f_{\eta_0}(x_N) -f_{\eta_0}(x) +\langle g(x_N), y\rangle-\langle g(x), y_N\rangle + \langle h_{\eta}(x_N), z\rangle - \langle h_{\eta}(x), z_N\rangle $. 
In view of Corollary~\ref{Cor_Coex_main}, we have 
\begin{align}
Q_\eta(w_N,w)&\le \tfrac{(L_{f,\eta}+z^T L_{h,\eta}) D_X^2}{N+1} + \tfrac{D_X \sqrt{9\bar M_{C,V}^2+\|A\|^2}}{\sqrt{N}}  \left( \|y-q_0\|_2^2 + \|z-r_0\|_2^2 + 1 \right) \label{eq:bnd_Q_eta}
\end{align}
for any $w \in X \times \bbr^m \times \bbr^d_+$.
Using the relations in \eqnok{eq:app_close_f} and \eqnok{eq:app_close_h}, and the fact that $z, z_N \in \bbr^d_+$,
we can see that
\begin{align}
Q(w_N,w) &\leq Q_\eta(w_N,w) + \eta_0 D_U^2 + \tsum_{i=1}^d (\eta_i z_i D_{V_i}^2) \nn\\ 
&\le Q_\eta(w_N,w) + \eta_0 D_U^2 + \|z\|_2 (\tsum_{i=1}^d (\eta_i D_{V_i}^2)^2)^{1/2}, \
\forall w \in X \times \bbr^m \times \bbr^d_+. \label{eq:bnd_Q_eta_close}
\end{align}
By letting $x = x^*$, $ y = 0$ and $z = 0$, we have
\begin{align*}
f(x_N) - f(x^*)\leq Q(w_N,z) \leq  Q_\eta(z_N,z) + \eta_0 D_U^2,
\end{align*}
which, in view of \eqnok{eq:bnd_Q_eta}, then implies \eqnok{eq:nonsmooth_obj_result}.
Now let $\hat w^*_N$ be defined in \eqnok{eq:def_hat_w}. By \eqnok{bnd_infeasibility}, \eqnok{eq:bnd_Q_eta} and \eqnok{eq:bnd_Q_eta_close},
we have
\begin{align*}
&\|g(x_N)\|_2 + \|[h(x_N)]_+\|_2  \leq Q(w_N,\hat w^*_N) \\
&\le  Q_\eta(w_N,\hat w^*_N) + \eta_0 D_U^2 + \|\hat z_N\|_2 (\tsum_{i=1}^d (\eta_i D_{V_i}^2)^2)^{1/2}\\
&\le \tfrac{(L_{f,\eta}+\hat z_N^T L_{h,\eta}) D_X^2}{N+1} + \tfrac{D_X \sqrt{9\bar M_{C,V}^2+\|A\|^2}}{\sqrt{N}}  \left( \|\hat y_N-q_0\|_2^2 + \|\hat z_N-r_0\|_2^2 + 1 \right) \\
& \quad + \eta_0 D_U^2 + \|\hat z_N\|_2 (\tsum_{i=1}^d (\eta_i D_{V_i}^2)^2)^{1/2}\\
&\le  \tfrac{[L_{f,\eta}+(\|z^*\|_2+1) \bar L_{h,\eta}]) D_X^2}{N+1} + \tfrac{2D_X \sqrt{9\bar M_{C,V}^2+\|A\|^2}}{\sqrt{N}}  \left( 2 \|(y^*;z^*)\|_2^2 + \|(q_0;r_0)\|_2^2 + 5 \right) \\
&\quad + \eta_0 D_U^2 + (\|z^*\|_2+1)(\tsum_{i=1}^d (\eta_i D_{V_i}^2)^2)^{1/2},
\end{align*}
where the last inequality follows from the bounds in \eqnok{eq:bound_hat_y} and \eqnok{eq:bound_hat_z}, and the facts that $ \|\hat z_N\|_2 \le \|z^*\|_2+1$ and
$
\hat z_N^T L_{h,\eta} \le \|\hat z_N^T\|_2 \|L_{h,\eta}\|_2 = (\|z^*\|_2+1) \bar L_{h,\eta}.
$
\end{proof}

\vgap

We now specify the selection of the smoothing parameters $\eta_i$, $i =0, \ldots, d$.
We consider only the most challenging case when 
the objective and all constraint functions are nonsmooth and 
establish the rate of convergence of the aforementioned CoexCG method for nonsmooth convex optimization.

\begin{corollary} \label{cor_nonsmooth}
Suppose that the smoothing parameters in problem \eqnok{general2} are set to
\begin{equation}\label{eq:smooth_eta}
\eta_0 = \tfrac{\|B\| D_X}{D_U\sqrt{N}} \ \ \mbox{and} \ \ \eta_{i} =\tfrac{\|C_i\| D_X}{D_{V_i} \sqrt{N}}, i = 1,\ldots, d.
\end{equation}
Then under the same premise of Theorem~\ref{the_nonsmooth}, we have
\begin{align}
f(x_N) - f(x^*) &\leq\tfrac{3D_X D_U \|B\|}{\sqrt{N}} + \tfrac{D_X \sqrt{9\bar M_{C,V}^2+\|A\|^2}}{\sqrt{N}}  \left( \|q_0\|_2^2 + \|r_0\|_2^2 + 1 \right),\\ 
\|[h(x_N)]_+\| +\|Ax_N\| &\leq  \tfrac{3 D_X D_U \|B\|}{\sqrt{N}} + \tfrac{2(\|z^*\|_2+1) D_X \sqrt{\tsum_{i=1}^d (D_{V_i} \|C_i\|)^2}}{\sqrt{N}} \nn \\
&\quad + \tfrac{2D_X \sqrt{9\bar M_{C,V}^2+\|A\|^2}}{\sqrt{N}}  \left( 2 \|(y^*;z^*)\|_2^2 + \|(q_0;r_0)\|_2^2 + 5 \right).
\end{align}
\end{corollary}

\begin{proof}
It follows from \eqnok{eq:smoothapp_f_Lip}, \eqnok{eq:def_app_h_lip_grad} and \eqnok{eq:smooth_eta} that
$L_{f,\eta} = \tfrac{\|B\|^2}{\eta_0} = \tfrac{D_U  \|B\| \sqrt{N}} {D_X}$ and that
\begin{align*}
\bar L_{h,\eta} &= \sqrt{\tsum_{i=1}^d \left(\tfrac{\|C_i\|^2}{\eta_i}\right)^2} 
=  \sqrt{\tsum_{i=1}^d \left(\tfrac{D_{V_i} \|C_i\| \sqrt{N }}{D_X}\right)^2}
= \tfrac{\sqrt{N} \sqrt{\tsum_{i=1}^d (D_{V_i} \|C_i\|)^2}}{D_X}.
\end{align*}
Also notice that $\eta_0 D_U^2 = \tfrac{D_X D_U \|B\|}{\sqrt{N}}$ and that
\begin{align*}
(\tsum_{i=1}^d (\eta_i D_{V_i}^2)^2)^{1/2}
&= \left(\tsum_{i=1}^d \tfrac{\|C_i\|^2 D_X^2 D_{V_i}^2 }{N} \right)^{1/2}
= \tfrac{D_X \sqrt{\tsum_{i=1}^d (D_{V_i} \|C_i\|)^2}}{\sqrt{N}}.
\end{align*}
Using these identities and the assumptions in \eqnok{eq:nonsmooth_obj_result}
and \eqnok{eq:nonsmooth_feas_result},
we have
\begin{align*}
f(x_N) - f(x^*) &\leq \tfrac{2 D_X D_U \|B\|}{\sqrt{N+1}} 
+ \tfrac{D_X \sqrt{9\bar M_{C,V}^2+\|A\|^2}}{\sqrt{N}}  \left( \|q_0\|_2^2 + \|r_0\|_2^2 + 1 \right) + \tfrac{D_X D_U\|B\|}{\sqrt{N}}\\ 
&\le \tfrac{3D_X D_U \|B\|}{\sqrt{N}} + \tfrac{D_X \sqrt{9\bar M_{C,V}^2+\|A\|^2}}{\sqrt{N}}  \left( \|q_0\|_2^2 + \|r_0\|_2^2 + 1 \right),\\
\|[h(x_N)]_+\| +\|Ax_N\| &\leq  \tfrac{2D_X D_U \|B\|}{\sqrt{N+1}} + \tfrac{(\|z^*\|_2+1) D_X \sqrt{\tsum_{i=1}^d (D_{V_i} \|C_i\|)^2}}{\sqrt{N+1}} \\
&\quad + \tfrac{2D_X \sqrt{9\bar M_{C,V}^2+\|A\|^2}}{\sqrt{N}}  \left( 2 \|(y^*;z^*)\|_2^2 + \|(q_0;r_0)\|_2^2 + 5 \right)  \\
&\quad + \tfrac{D_X D_U\|B\|}{\sqrt{N}}+ \tfrac{(\|z^*\|_2+1) D_X \sqrt{\tsum_{i=1}^d (D_{V_i} \|C_i\|)^2}}{\sqrt{N}} \\
&\le \tfrac{3 D_X D_U \|B\|}{\sqrt{N}} + \tfrac{2(\|z^*\|_2+1) D_X \sqrt{\tsum_{i=1}^d (D_{V_i} \|C_i\|)^2}}{\sqrt{N}} \\
&\quad + \tfrac{2D_X \sqrt{9\bar M_{C,V}^2+\|A\|^2}}{\sqrt{N}}  \left( 2 \|(y^*;z^*)\|_2^2 + \|(q_0;r_0)\|_2^2 + 5 \right).
\end{align*}
\end{proof}

We add a few remarks about the results obtained in Theorem~\ref{the_nonsmooth} and Corollary~\ref{cor_nonsmooth}.
Firstly, in view of Corollary~\ref{cor_nonsmooth}, even if $f$ and $h_i$ are nonsmooth functions,
the number of CoexCG iterations required to find an $\epsilon$-solution of problem~\eqnok{general}
is still bounded by ${\cal O}(1/\epsilon^2)$. Therefore, by utilizing the structural information of $f$ and $h_i$,
the CoexCG can solve this type of nonsmooth problem efficiently as if they are smooth functions.
Secondly, if either the objective function or some constraint functions are smooth,
we can set the corresponding smoothing parameter to be zero and obtain slightly 
improved complexity bounds than those in Corollary~\ref{cor_nonsmooth}.
Thirdly, similar to the CoexCG method applied for solving problem~\eqnok{general} with
smooth objective and constraint functions, we need to fix the number of iterations $N$
in advance when specifying the algorithmic parameters and smoothing parameters.
We will address this issue in the next section.

\section{Constraint-extrapolated and dual-regularized conditional gradient method} \label{sec_adaptive}
One critical shortcoming associated with the basic version of the CoexCG method is that
we need to fix the number of iterations $N$ a priori. Our goal
in this section is to develop a variant of CoexCG which does not have this requirement.
We consider the case when $f$ and $h_i$ are
 smooth and structured nonsmooth functions, respectively, in Subsections~\ref{sec_ada_smooth} 
 and \ref{sec_ada_nonsmooth}.

\subsection{Smooth functions} \label{sec_ada_smooth}
In order to remove the assumption of fixing $N$ a priori, 
we suggest to modify
the dual projection steps \eqnok{eq:step_dual1} and \eqnok{eq:step_dual2}
in the CoexCG method. More specifically,
we add an additional regularization term with diminishing
weights into these steps. 
This variant of CoexCG is formally described in 
Algorithm~\ref{algAdaCoexCG}.

\begin{algorithm}[H]
\caption{{\bf Co}nstraint-{\bf ex}trapolated and {\bf Du}al-{\bf r}egularized Conditional Gradient (CoexDurCG)}\label{algAdaCoexCG}
\begin{algorithmic} 
\State The algorithm is the same as CoexCG except that \eqnok{eq:step_dual1} and \eqnok{eq:step_dual2}
are replaced by
\begin{align}
q_k &= \argmin_{y \in \bbr^m} \{ \langle -\tilde g_k, y\rangle + \tfrac{\tau_k}{2}\|y-q_{k-1}\|_2^2 + \tfrac{\gamma_k}{2}\|y-q_{0}\|_2^2\}, \label{eq:step_dual1_ada}\\
r_k &= \argmin_{ z \in \bbr^d_+} \{ \langle -\tilde h_k, z\rangle + \tfrac{\tau_k}{2}\|z-r_{k-1}\|_2^2 +  \tfrac{\gamma_k}{2}\| z-r_{0}\|_2^2 \}, \label{eq:step_dual2_ada}
\end{align}
for some $\gamma_k \ge 0$.
\end{algorithmic}
\end{algorithm}

Clearly, we can write $q_k$ and $r_k$ in \eqnok{eq:step_dual1_ada} and \eqnok{eq:step_dual2_ada}
equivalently as
\begin{align*}
q_k = \tfrac{1}{\tau_k + \gamma_k} (\tau_k q_{k-1} + \gamma_k q_0 + \tilde g_k) \ \ \mbox{and} \ \
r_k = \max \left\{ \tfrac{1}{\tau_k + \gamma_k} (\tau_k r_{k-1} + \gamma_k r_0 + \tilde h_k), 0\right\}.
\end{align*}
Similar to the CoexCG method, it is also possible to generalize CoexDurCG for solving problems with
conic inequality constraints.
The following result, whose proof can be found in Lemma 3.5 of \cite{LanBook2020}, characterizes
the optimality conditions for \eqnok{eq:step_dual1_ada} and \eqnok{eq:step_dual2_ada}.

\begin{lemma}
Let $q_{k}$ and $r_k$ be defined in \eqnok{eq:step_dual1_ada} and \eqnok{eq:step_dual2_ada},
respectively. Then,
\begin{align}
& \langle -\tilde g_k, q_{k} - y \rangle + \tfrac{\tau_k}{2} \|q_{k} - q_{k-1}\|_2^2 + \tfrac{\gamma_k}{2} \|q_k - q_0\|_2^2 \nn \\
&\quad \le \tfrac{\tau_k}{2} \|y - q_{k-1}\|_2^2 - \tfrac{\tau_k+\gamma_k}{2} \|y - q_{k}\|_2^2 + \tfrac{\gamma_k}{2} \|y - q_0\|_2^2, \  \forall y \in \bbr^m, \label{opt_cond_dual1_ada}\\
&  \langle -\tilde h_k, r_{k} - z \rangle + \tfrac{\tau_k}{2} \|r_{k} - r_{k-1}\|_2^2 + \tfrac{\gamma_k}{2} \|r_k - r_0\|_2^2 \nn \\
&\quad \le \tfrac{\tau_k}{2} \|z - r_{k-1}\|_2^2 - \tfrac{\tau_k+\gamma_k}{2} \|z - r_{k}\|_2^2 + \tfrac{\gamma_k}{2} \|y - r_0\|_2^2, \ \forall z \in \bbr^d_+.  \label{opt_cond_dual2_ada}
\end{align}
\end{lemma}

We now establish an important recursion about the CoexDurCG method,
which can be viewed as a counterpart of Proposition~\ref{Prop_CoexCG} for the CoexCG method.

\begin{proposition}\label{Prop_CoexDurCG}
For any $k>1$, we have
\begin{align*}
Q(w_k,w) &\leq  (1-\alpha_k) Q(w_{k-1},w)+ \tfrac{(L_f+z^T L_h)\alpha_k^2 D_X^2}{2} +\tfrac{\alpha_k\lambda_k^2 (9 \bar M_h^2+\|A\|^2) D_X^2}{2 \tau_k}\\
&\quad + \alpha_k [ \langle A(p_k-p_{k-1}), y- q_k\rangle - \lambda_k\langle A (p_{k-1}-p_{k-2}), y -q_{k-1}\rangle]\\
& \quad + \alpha_k[\langle l_h(x_{k-1},p_k)-l_h(x_{k-2},p_{k-1}), z- r_k\rangle - \lambda_k \langle l_h(x_{k-2},p_{k-1})-l_h(x_{k-3},p_{k-2}), z- r_{k-1}\rangle] \\
&\quad  + \tfrac{\alpha_k\tau_k}{2}( \|y-q_{k-1}\|_2^2+\|z-r_{k-1}\|_2^2) - \tfrac{\alpha_k(\tau_k+\gamma_k)}{2} (\|y-q_k\|_2^2 +\|z-r_k\|_2^2)\\
&\quad +\tfrac{\alpha_k\gamma_k}{2}[\|y-q_{0}\|_2^2 + \|z-r_{0}\|_2^2 ], \ \forall w \in X \times \bbr^m \times \bbr^d_+,
\end{align*}
where $D_X$ is defined in \eqnok{eq:def_D_X}.
\end{proposition}

\begin{proof}
Multiplying both sides of \eqnok{opt_cond_dual1_ada} and \eqnok{opt_cond_dual2_ada} by $\alpha_k$ and
summing them up with the inequality in \eqnok{eq:recursion_on_gap_function}, we
have 
\begin{align}
Q(w_k,w) & \leq  (1-\alpha_k) Q(w_{k-1},w)+ \tfrac{(L_f+z^TL_h)\alpha_k^2 D_X^2}{2} \nn \\
&\quad +\alpha_k \langle g(p_k)- \tilde g_k), y- q_k\rangle +\alpha_k \langle l_h(x_{k-1},p_k)- \tilde h_k, z - r_k\rangle\nn \\
&\quad + \tfrac{\alpha_k\tau_k}{2}[\|y-q_{k-1}\|_2^2-\|q_k-q_{k-1}\|_2^2] - \tfrac{\alpha_k(\tau_k+\gamma_k)}{2} \|y-q_k\|_2^2 \nn \\
 &\quad + \tfrac{\alpha_k\tau_k}{2}[\|z-r_{k-1}\|_2^2 -\|r_k-r_{k-1}\|_2^2] - \tfrac{\alpha_k(\tau_k+\gamma_k)}{2} \|z-r_k\|_2^2 \nn\\
 &\quad  +\tfrac{\alpha_k\gamma_k}{2}[\|y-q_{0}\|^2 -\|q_k-q_{0}\|^2] +\tfrac{\alpha_k\gamma_k}{2}[\|z-r_{0}\|^2 -\|z_k-r_{0}\|^2] \nn\\
&  \leq  (1-\alpha_k) Q(w_{k-1},w)+ \tfrac{(L_f+z^TL_h)\alpha_k^2 D_X^2}{2} \nn \\
&\quad +\alpha_k \langle g(p_k)- \tilde g_k), y- q_k\rangle +\alpha_k \langle l_h(x_{k-1},p_k)- \tilde h_k, z - r_k\rangle\nn \\
&\quad + \tfrac{\alpha_k\tau_k}{2}[\|y-q_{k-1}\|_2^2-\|q_k-q_{k-1}\|_2^2] - \tfrac{\alpha_k(\tau_k+\gamma_k)}{2} \|y-q_k\|_2^2 \nn \\
 &\quad + \tfrac{\alpha_k\tau_k}{2}[\|z-r_{k-1}\|_2^2 -\|r_k-r_{k-1}\|_2^2] - \tfrac{\alpha_k(\tau_k+\gamma_k)}{2} \|z-r_k\|_2^2 \nn\\
 &\quad  +\tfrac{\alpha_k\gamma_k}{2}[\|y-q_{0}\|_2^2 + \|z-r_{0}\|_2^2 ], \ \forall w \in X \times \bbr^m \times \bbr^d_+. \label{eq:Coex_temp1_ada}
\end{align}
The result then follows by plugging relations \eqnok{eq:Coex_temp2} and \eqnok{eq:Coex_temp3} into \eqnok{eq:Coex_temp1_ada}.
\end{proof}

\vgap

We are now ready to establish the main convergence properties of the
CoexDurCG method.

\begin{theorem} \label{the_CoexDurCG}
Let $\Gamma_k$ be defined in \eqref{Def_Gamma} and assume that
the algorithmic parameters $\alpha_k, \tau_k$ and $\lambda_k$ in the CoexDurCG method satisfy
\begin{equation}\label{cond2}
\alpha_1 = 1, \ \tfrac{\lambda_k\alpha_k}{\Gamma_k} = \tfrac{\alpha_{k-1}}{\Gamma_{k-1}} \text{ and } \tfrac{\alpha_k\tau_k}{\Gamma_k} \leq \tfrac{\alpha_{k-1}(\tau_{k-1}+\gamma_{k-1})}{\Gamma_{k-1}}
\ \forall k \ge 2.
\end{equation}
Then we have
\begin{equation}\label{thm1_gap_result_ada}
\begin{aligned}
Q(w_N,w) &\leq  \Gamma_N\tsum_{k=1}^N\left[\tfrac{(L_f+z^T L_h)\alpha_k^2D_X^2}{2\Gamma_k} +\tfrac{\alpha_k\lambda_k^2(9 \bar M_h^2+\|A\|^2)D_X^2}{2\tau_k\Gamma_k}\right] 
+\tfrac{\alpha_N(9\bar M_h^2+\|A\|^2) D_X^2}{2(\tau_N+\gamma_N)} \\
&\quad + \Gamma_N \left(\tfrac{\tau_1}{2}+\tsum_{k=1}^N \tfrac{\alpha_k\gamma_k}{2\Gamma_k}\right) (\|y-q_0\|_2^2 +\|z-r_0\|_2^2),\ \forall w \in X \times \bbr^m \times \bbr^d_+,
\end{aligned}
\end{equation}
where $D_X$ is defined in \eqnok{eq:def_D_X}. As a consequence, we have
\begin{align}
f(x_N) - f(x^*) &\le  \Gamma_N\tsum_{k=1}^N\left[\tfrac{L_f \alpha_k^2D_X^2}{2\Gamma_k} +\tfrac{\alpha_k\lambda_k^2(9 \bar M_h^2+\|A\|^2)D_X^2}{2\tau_k\Gamma_k}\right]
+ \tfrac{\alpha_N(9\bar M_h^2+\|A\|^2) D_X^2}{2(\tau_N+\gamma_N)}  \nn \\
&\quad +\Gamma_N \left(\tfrac{\tau_1}{2}+\tsum_{k=1}^N \tfrac{\alpha_k\gamma_k}{2\Gamma_k}\right) (\|q_0\|_2^2 + \|r_0\|_2^2), \label{thm1_obj_result_ada}
\end{align}
and
\begin{align}
\|g(x_N)\|_2 + \|[h(x_N)]_+\|_2 &\le   \Gamma_N\tsum_{k=1}^N\left[\tfrac{[L_f+(\|z^*\|_2+1) \bar L_h]\alpha_k^2D_X^2}
{2\Gamma_k} +\tfrac{\alpha_k\lambda_k^2(9 \bar M_h^2+\|A\|^2)D_X^2}{2\tau_k\Gamma_k}\right]
+\tfrac{\alpha_N(9\bar M_h^2+\|A\|^2) D_X^2}{2(\tau_N+\gamma_N)} \nn\\
&\quad +\Gamma_N \left(\tfrac{\tau_1}{2}+\tsum_{k=1}^N \tfrac{\alpha_k\gamma_k}{2\Gamma_k}\right) [(\|y^*\|_2+1)^2 + \|q_0\|_2^2 + (\|z^*\|_2+1)^2 + \|r_0\|_2^2],
\label{thm1_feas_result_ada}
\end{align}
where $(x^*, y^*, z^*)$ denotes a triple of optimal solutions for problem~\eqnok{minmax}.
\end{theorem}

\begin{proof}
It follows from Lemma~\ref{Lemma 2} and Proposition~\ref{Prop_CoexDurCG}  that
\begin{align*}
\tfrac{Q(w_N,w)}{\Gamma_N} \leq&  (1-\alpha_1) Q(w_0,w)+ \tsum_{k=1}^N[\tfrac{(L_f+z^T L_h)\alpha_k^2 D_X^2 }{2\Gamma_k} +\tfrac{\alpha_k\lambda_k^2(9 \bar M_h^2 + \|A\|^2)D_X^2}{2\tau_k\Gamma_k}]\\
& + \tsum_{k=1}^N \tfrac{\alpha_k}{\Gamma_k} [ \langle A(p_k-p_{k-1}), y- q_k\rangle - \lambda_k \langle A (p_{k-1}-p_{k-2}), y - q_{k-1}\rangle]\\
& + \tsum_{k=1}^N\tfrac{\alpha_k}{\Gamma_k}[\langle l_h(x_{k-1},p_k)-l_h(x_{k-2},p_{k-1}), z - r_k\rangle  \\
& \quad \quad  - \lambda_k \langle l_h(x_{k-2},p_{k-1})-l_h(x_{k-3},p_{k-2}), z - r_{k-1}\rangle ]\\
& + \tsum_{k=1}^N\left[\tfrac{\alpha_k\tau_k}{2\Gamma_k}(\|y-q_{k-1}\|_2^2+\|z-r_{k-1}\|_2^2) -\tfrac{\alpha_k(\tau_k+\gamma_k)}{2\Gamma_k}(\|y-q_k\|_2^2+\|z-r_k\|_2^2) \right] \\
& + \tsum_{k=1}^N \tfrac{\alpha_k\gamma_k}{2\Gamma_k}[\|y-q_{0}\|_2^2 + \|z-r_{0}\|_2^2 ], 
\end{align*}
which, in view of \eqref{cond2}, then implies that
\begin{align*}
Q(w_N,w) &\leq  \Gamma_N \tsum_{k=1}^N[\tfrac{(L_f+z^T L_h)\alpha_k^2 D_X^2 }{2\Gamma_k} +\tfrac{\alpha_k\lambda_k^2(9 \bar M_h^2 + \|A\|^2)D_X^2}{2\tau_k\Gamma_k}]\\
&\quad +\alpha_N \langle A(p_N-p_{N-1}), y-q_N\rangle -\tfrac{\alpha_N(\tau_N+\gamma_N)}{2}\|y-q_N\|_2^2 \\
&\quad + \alpha_N\langle l_h(x_{N-1},p_N)-l_h(x_{N-2},p_{N-1}),z-r_N\rangle -\tfrac{\alpha_N(\tau_N+\gamma_N)}{2}\|z-r_N\|_2^2 \\
&\quad  + \tfrac{\alpha_1\tau_1\Gamma_N}{2}[\|y-q_{0}\|_2^2 + \|z-r_{0}\|_2^2 ]\\
& \quad + \Gamma_N \tsum_{k=1}^N \tfrac{\alpha_k\gamma_k}{2\Gamma_k}[\|y-q_{0}\|^2 + \|z-r_{0}\|^2 ]\\
& \leq   \Gamma_N \tsum_{k=1}^N[\tfrac{(L_f+z^T L_h)\alpha_k^2 D_X^2 }{2\Gamma_k} +\tfrac{\alpha_k\lambda_k^2(9 \bar M_h^2 + \|A\|^2)D_X^2}{2\tau_k\Gamma_k}]\\
&\quad +\tfrac{\alpha_N}{2(\tau_N+\gamma_N)} \|A\|^2\|p_N-p_{N-1}\|_2^2 +\tfrac{9 \bar M_h^2 \alpha_N D_X^2}{2(\tau_N+\gamma_N)}\\
& \quad  + \tfrac{\alpha_1\tau_1\Gamma_N}{2}[\|y-q_{0}\|_2^2+ \|z-r_{0}\|_2^2]\\
&\quad  + \Gamma_N \tsum_{k=1}^N \tfrac{\alpha_k\gamma_k}{2\Gamma_k}[\|y-q_{0}\|_2^2 + \|z-r_{0}\|_2^2 ],
\end{align*}
where the last relation follows from Young's inequality and 
a result similar to \eqnok{eq:Coex_temp4}. The result in \eqnok{thm1_gap_result} then immediately follows from the above inequality.
We can show \eqnok{thm1_obj_result_ada} and \eqnok{thm1_feas_result_ada} similarly to
\eqnok{thm1_obj_result} and \eqnok{thm1_feas_result}, and hence the details are skipped.
\end{proof}

\vgap

Corollary~\ref{Cor_CoexDurCG} below shows how to specify the algorithmic parameters,
including the regularization weight $\gamma_k$, for the CoexDurCG method.
In particular, the selection of $\tau_k$ was inspired by the one used in \eqnok{para_Coex_main},
and $\gamma_k$ was chosen so that the last relation in \eqnok{cond2} is satisfied. 

\begin{corollary}\label{Cor_CoexDurCG}
If the algorithmic parameters $\alpha_k$, $\lambda_k$, $\tau_k$ and $\gamma_k$
of the CoexDurCG method are set to
\begin{equation}\label{para_CoexDur_main}
\alpha_k = \tfrac{2}{k+1}, \lambda_k = \tfrac{k-1}{k},\ \tau_k = \beta \sqrt{k}, \ \mbox{and} \  \gamma_k = \tfrac{\beta}{k}[(k+1)\sqrt{k+1} - k\sqrt{k}],
\end{equation}
with $\beta = D_X \sqrt{9 \bar M_h^2 +\|A\|^2}$ for $k \ge 1$, then
we have, $ \forall w \in X \times \bbr^m \times \bbr^d_+$,
\begin{equation}\label{cor 4}
\begin{aligned}
Q(z_k,z) \leq &  \tfrac{2 (L_f+z^T L_h) D_X^2}{N+1} +  \tfrac{D_X \sqrt{9 \bar M_h^2 +\|A\|^2}}{\sqrt{N}} \left[3 (\|y-q_0\|_2^2 +\|z-r_0\|_2^2)+1\right].
\end{aligned}
\end{equation}
In addition, we have
\begin{align}
f(x_N) - f(x^*) &\le  \tfrac{2 L_f D_X^2}{N+1} +  \tfrac{D_X \sqrt{9 \bar M_h^2 +\|A\|^2}}{\sqrt{N}} \left[3 (\|q_0\|_2^2 +\|r_0\|_2^2)+1\right] \label{cor_obj_result_ada}
\end{align}
and
\begin{align}
&\|g(x_N)\|_2 + \|[h(x_N)]_+\|_2 \le \tfrac{2 (L_f+(\|z^*\|_2+1) \bar L_h) D_X^2}{N+1} \nn \\
&\quad \quad +   \tfrac{D_X \sqrt{9 \bar M_h^2 +\|A\|^2}}{\sqrt{N}} \left[3 [(\|y^*\|_2+1)^2 + (\|z^*\|_2+1)^2  + \|q_0\|_2^2+ \|r_0\|_2^2] + 1\right],\label{cor_feas_result_ada}
%&\quad \tfrac{D_X \sqrt{3 \bar M_h +\|A\|^2}}{\sqrt{N}} \left[3 (\|q_0\|_2^2 +\|r_0\|_2^2)+1\right]
 %\Gamma_N\tsum_{k=1}^N\left[\tfrac{[L_f+(\|z^*\|_2+1) \bar L_h]\alpha_k^2D_X^2}{2\Gamma_k} +\tfrac{\alpha_k\lambda_k^2(3 \bar M_h^2+\|A\|^2)D_X^2}{2\tau_k\Gamma_k}\right]
%+\tfrac{\alpha_N(3\bar M_h^2+\|A\|^2) D_X^2}{2(\tau_N+\gamma_N)} \nn\\
%&\quad +\Gamma_N \left(\tau_1+\tsum_{k=1}^N \tfrac{\alpha_k\gamma_k}{2\Gamma_k}\right) [(\|y^*\|_2+1)^2 + \|q_0\|_2^2 + (\|z^*\|_2+1)^2 + \|r_0\|_2^2],
\end{align}
where $(x^*, y^*, z^*)$ denotes a triple of optimal solutions for problem~\eqnok{minmax}.
\end{corollary}

\begin{proof}
From the definition of $\alpha_k$ in \eqnok{para_CoexDur_main}, we have $\Gamma_k = 2/[k(k+1)]$ and $\alpha_k / \Gamma_k = k$.
Hence the first two conditions in \eqnok{cond2} hold. In addition, it follows from these identities and \eqnok{para_CoexDur_main}
that $\tfrac{\alpha_k\tau_k}{\Gamma_k} = \beta k\sqrt{k}$ and
\begin{align*}
\tfrac{\alpha_{k-1}(\tau_{k-1}+\gamma_{k-1})}{\Gamma_{k-1}} = (k-1)\left [\beta \sqrt{k-1} + \tfrac{\beta}{k-1}[k\sqrt{k} - (k-1)\sqrt{k-1}]\right] = \beta k \sqrt{k},
\end{align*}
and hence that the last relation in \eqnok{cond2} also holds.
Observe that by \eqnok{para_CoexDur_main},
\begin{align}
\tsum_{k=1}^N \tfrac{\alpha_k^2}{\Gamma_k} &= 2\tsum_{k=1}^N \tfrac{k}{k+1} \le 2 N, \label{eq:coexduCG_gamma_1}\\
 \tsum_{k=1}^N\tfrac{\alpha_k\gamma_k}{\Gamma_k} &= \beta\tsum_{k=1}^N[(k+1)\sqrt{k+1}-k\sqrt{k}] = \beta[(N+1)\sqrt{N+1}-1],\label{eq:coexduCG_gamma_2}\\
 \tsum_{k=1}^N\tfrac{\alpha_k\lambda_k^2}{\tau_k\Gamma_k} &= \tsum_{k=1}^N\tfrac{(k-1)^2}{\beta k\sqrt{k}}\leq \tfrac{1}{\beta} \tsum_{k=1}^N \sqrt{k-1}
 \le \tfrac{1}{\beta} \int_0^N \sqrt{t} dt = \tfrac{2}{3\beta}N^{3/2}. \label{eq:coexduCG_gamma_3}
\end{align}
Using these relations in \eqnok{thm1_gap_result_ada}, we have
\begin{align*}
Q(w_N,w) &\leq \tfrac{2 (L_f+z^T L_h) D_X^2}{N+1} + \tfrac{2\sqrt{N} (9 \bar M_h^2+\|A\|^2)D_X^2}{3 \beta (N+1)} 
+ \tfrac{N (9\bar M_h^2+\|A\|^2) D_X^2}{\beta (N+1)^2 \sqrt{N+1}}
+ \tfrac{2 \beta\sqrt{N+1}}{N} (\|y-q_0\|_2^2 +\|z-r_0\|_2^2)\\
&= \tfrac{2 (L_f+z^T L_h) D_X^2}{N+1}
+\tfrac{D_X \sqrt{9 \bar M_h^2 +\|A\|^2}}{\sqrt{N}}
 \left[ \tfrac{2 }{3} + \tfrac{N}{(N+1)^2} + \tfrac{2 \sqrt{N+1}}{\sqrt{N}} (\|y-q_0\|_2^2 +\|z-r_0\|_2^2)\right] \\
 &\le \tfrac{2 (L_f+z^T L_h) D_X^2}{N+1}
+  \tfrac{D_X \sqrt{9 \bar M_h^2 +\|A\|^2}}{\sqrt{N}} \left[3 (\|y-q_0\|_2^2 +\|z-r_0\|_2^2)+1\right].
%&+\Gamma_N\tsum_{k=1}^N\left[\tfrac{(L_f+z^T L_h)\alpha_k^2D_X^2}{2\Gamma_k} +\tfrac{\alpha_k\lambda_k^2(3 \bar M_h^2+\|A\|^2)D_X^2}{2\tau_k\Gamma_k}\right] 
%&+\tfrac{\alpha_N(3\bar M_h^2+\|A\|^2) D_X^2}{2(\tau_N+\gamma_N)} \\
%&\quad + \left(\tau_1\Gamma_N+\Gamma_N \tsum_{k=1}^N \tfrac{\alpha_k\gamma_k}{2\Gamma_k}\right) (\|y-q_0\|_2^2 +\|z-r_0\|_2^2)
\end{align*}
The bounds in \eqnok{cor_obj_result_ada} and \eqnok{cor_feas_result_ada} can be shown similarly and the details are skipped.
%Therefore, from \eqref{thm2}, we have
%\begin{align*}
%Q(z_N,z) \leq & \tfrac{2N(9\|M_h\|^2+\|A\|^2)D_X^2}{\beta(N+1)^2\sqrt{N+1}}+\tfrac{\|d-r_0\|^2+\|y-q_0\|^2}{N(N+1)}+\tfrac{2}{N(N+1)}\tsum_{k=1}^N \tfrac{k}{k+1}(L_f+d^TL_h)D_X^2 \\
%& + \tfrac{2}{N(N+1)}\tfrac{1}{2\beta}N^{3/2}(6\|M_h\|^2+\|A\|^2)D_X^2+ \tfrac{\beta[(N+1)\sqrt{N+1}-1]}{N(N+1)}[\|d-r_0\|^2+\|y-q_0\|^2] \\
%\leq & \tfrac{2(9\|M_h\|^2+\|A\|^2)D_X^2}{\beta(N+1)\sqrt{N+1}}+\tfrac{\|d-r_0\|^2+\|y-q_0\|^2}{N(N+1)}+\tfrac{2}{N+1}(L_f+d^TL_h)D_X^2 \\
%& + \tfrac{(6\|M_h\|^2+\|A\|^2)}{\sqrt{N}\beta}D_X^2+ \tfrac{\beta}{\sqrt{N+1}}[\|d-r_0\|^2+\|y-q_0\|^2] 
%\end{align*}
%Take $\beta = \sqrt{9\|M_h\|^2+\|A\|^2}D_X$, we have \eqref{cor 4}.
\end{proof}

In view of the results obtained in Corollary~\ref{Cor_CoexDurCG},
the rate of convergence of CoexDurCG matches that of CoexCG. Moreover, 
the cost of each iteration of the CoexDurCG is the same as that of CoexCG.
%In CoexDurCG we do not need to fix the number of iterations a priori

\subsection{Structured Nonsmooth Functions} \label{sec_ada_nonsmooth}
In this subsection, we consider problem \eqnok{general} with structured nonsmooth functions $f$ and $h_i$  given in \eqnok{Nonsmooth}.
One possible way to solve this nonsmooth problem is to apply the CoexDurCG method for the smooth approximation
problem~\eqnok{general2}. However, this approach still requires us to fix the number of iterations $N$ when choosing
smoothing parameters $\eta_i$, $i = 0, \ldots, d$. 
%In this subsection,
%we show that adaptive smoothing parameters can be used within the framework of the CoexDurCG method.

Our goal in this subsection is to generalize the CoexDurCG method to solve this structured nonsmooth problem directly. Rather than applying this algorithm
to problem~\eqnok{general2}, we modify the smoothing parameters $\eta_i$, $i = 0, \ldots, d$, 
at each iteration. More specifically,  we assume that
\beq \label{eq:decreasing_eta}
\eta_i^1 \ge \eta_i^2 \ge \ldots \ge \eta_i^k, \ \forall i = 0, \ldots, d,
\eeq
and define a sequence of smoothing functions $f_{\eta_0^k}(x)$ and $h_{i, \eta_i^k} (x)$, $i = 1, \ldots, d$,
according to \eqnok{eq:def_smoothapp_f} and \eqnok{eq:def_smoothapp_h}, respectively.
For simplicity, we denote
\[
f^k(x) \equiv f_{\eta_0^k}(x), \ \ h_i^k(x) \equiv h_{i, \eta_i^k} (x) \
\mbox{and} \
h^k(x) \equiv (h^k_1(x); \ldots; h^k_d(x)).
\]
Also let us define the Lipschitz constants 
\begin{align*}
L_f^k \equiv \tfrac{\|B\|^2}{\mu_0+\eta_0^k}, \ L_h^k \equiv (\tfrac{\|C_1\|^2}{\mu_1 + \eta_1^k}; \ldots; \tfrac{\|C_d\|^2}{\mu_d + \eta_d^k}), \ \mbox{and} \ \bar L_h^k \equiv \|L_h^k\|_2.
\end{align*}
It can be seen from \eqnok{eq:decreasing_eta} that
\beq \label{eq:app_relation_f}
f^{k-1}(x) \le f^k(x) \le f^{k-1}(x) + (\eta_0^{k-1} - \eta_0^k) D_U^2, \ \forall x \in X.
\eeq
Indeed, it suffices to show the second relation in \eqnok{eq:app_relation_f}.
By definition, we have
\begin{align*}
f^k(x) &= \max_{q \in Q} \{\langle Bx, q\rangle - \hat f(q) - \eta_0^k U(q)\} 
=  \max_{q \in Q} \{\langle Bx, q\rangle - \hat f(q) - \eta_0^{k-1} U(q) +(\eta_0^{k-1} - \eta_0^k) U(q)\}\\
&\le \max_{q \in Q} \{\langle Bx, q\rangle - \hat f(q) - \eta_0^{k-1} U(q) +(\eta_0^{k-1} - \eta_0^k) D_U^2\} 
= f^{k-1}(x) +(\eta_0^{k-1} - \eta_0^k) D_U^2,
\end{align*}
where the inequality follows from the definition of $D_U$ in \eqnok{eq:def_DU} and the assumption $\eta_0^{k-1} \ge \eta_0^k$ in \eqnok{eq:decreasing_eta}.
Similarly, we have
\beq \label{eq:app_relation_h}
h_i^{k-1}(x) \le h_i^{k}(x) \le h_i^{k-1}(x) + (\eta_i^{k-1} - \eta_i^k) D_{V_i}^2, \ \forall x X, \ i =1, \ldots,d.
\eeq
Note that in our algorithmic scheme, we can set $\eta_i^k = 0$, $i = 0, 1,\ldots, d$, if the corresponding
objective or constraint functions are smooth (i.e., $\mu_i = 0$).

We now describe the more general CoexDurCG method for solving structured nonsmooth problems.
\begin{algorithm}[H]
\caption{CoexDurCG for Structured Nonsmooth Problems}\label{algAdaCoexCG_nonsmooth}
\begin{algorithmic} %[1]
\State The algorithm is the same as 
Algorithm~\ref{algAdaCoexCG} except that the extrapolation step \eqnok{eq:step_ex2} is replaced by
\beq \label{eq:step_ex2_ada}
\tilde h_k = l_{h^{k-1}}(x_{k-2},p_{k-1})+\lambda_k [l_{h^{k-1}}(x_{k-2},p_{k-1})-l_{h^{k-2}}(x_{k-3},p_{k-2})],
\eeq
and the linear optimization step is replaced by
\beq \label{eq:step_lo_ada}
p_k = \argmin_{x\in X} \{ l_{f^{k}}(x_{k-1}, x) + \langle g(x), q_k \rangle +  \langle l_{h^k}(x_{k-1},x),  r_k\rangle \}.
\eeq
\end{algorithmic}
\end{algorithm}
In Algorithm~\ref{algAdaCoexCG_nonsmooth} we do not explicitly use the smooth approximation problem~\eqnok{general2}.
Instead, we incorporate in \eqnok{eq:step_ex2_ada} and \eqnok{eq:step_lo_ada}
the adaptive linear approximation functions $l_{h^k}$ and $l_{f^k}$ for the objective and constraints, respectively.
The convergence analysis of this algorithm relies on the adaptive primal-dual gap function:
\beq \label{eq:defQ_k}
Q^k(\bar w,w) \equiv Q_{\eta^k}(\bar w,w) 
:= f^k(\bar x) -f^k(x) +\langle g(\bar x), y\rangle-\langle g(x), \bar y \rangle + \langle h^k(\bar x), z\rangle - \langle h^k(x),\bar z\rangle,
\eeq 
as demonstrated in the following result.
%We first establish an important recursion about $Q^k$.

\begin{proposition}\label{Prop_CoexDurCG_nonsmooth}
For any $k>1$, we have
\begin{align*}
Q^k(w_k,w) &\leq  (1-\alpha_k) Q^{k-1}(w_{k-1},w)+ \tfrac{(L_f^k+z^T L_h^k)\alpha_k^2 D_X^2}{2} \\
&\quad + (1 -\alpha_k)[ (\eta_0^{k-1} - \eta_0^k) D_U^2 + \tsum_{i=1}^d (\eta_i^{k-1} - \eta_i^k) z_i D_{V_i}^2] \\
&\quad + \tfrac{\alpha_k\lambda_k^2 (12 \bar M_{C,V}^2+\|A\|^2) D_X^2}{2 \tau_k}+ \tfrac{3 \lambda_k^2}{\tau_k} \tsum_{i=1}^d (\eta_i^{k-2} -\eta_i^{k-1})^2 D_{V_i}^4 \\ 
&\quad + \alpha_k [ \langle A(p_k-p_{k-1}), y- q_k\rangle - \lambda_k\langle A (p_{k-1}-p_{k-2}), y -q_{k-1}\rangle]\\
& \quad + \alpha_k[\langle l_{h^k}(x_{k-1},p_k)-l_{h^{k-1}}(x_{k-2},p_{k-1}), z- r_k\rangle \\
& \quad \quad \quad - \lambda_k \langle l_{h^{k-1}}(x_{k-2},p_{k-1})-l_{h^{k-2}}(x_{k-3},p_{k-2}), z- r_{k-1}\rangle] \\
&\quad  + \tfrac{\alpha_k\tau_k}{2}( \|y-q_{k-1}\|_2^2+\|z-r_{k-1}\|_2^2) - \tfrac{\alpha_k(\tau_k+\gamma_k)}{2} (\|y-q_k\|_2^2 +\|z-r_k\|_2^2)\\
&\quad +\tfrac{\alpha_k\gamma_k}{2}[\|y-q_{0}\|_2^2 + \|z-r_{0}\|_2^2 ], \ \forall w \in X \times \bbr^m \times \bbr^d_+,
\end{align*}
where $D_X$ is defined in \eqnok{eq:def_D_X}.
\end{proposition}

\begin{proof}
Similar to \eqnok{eq:Coex_temp1_ada}, we can show that
\begin{align}
Q^k(w_k,w) 
&  \leq  (1-\alpha_k) Q^k(w_{k-1},w)+ \tfrac{(L_f^k+z^T L_h^k)\alpha_k^2 D_X^2}{2} \nn \\
&\quad +\alpha_k \langle g(p_k)- \tilde g_k), y- q_k\rangle +\alpha_k \langle l_{h^k}(x_{k-1},p_k)- \tilde h_k, z - r_k\rangle\nn \\
&\quad + \tfrac{\alpha_k\tau_k}{2}[\|y-q_{k-1}\|_2^2-\|q_k-q_{k-1}\|_2^2] - \tfrac{\alpha_k(\tau_k+\gamma_k)}{2} \|y-q_k\|_2^2 \nn \\
 &\quad + \tfrac{\alpha_k\tau_k}{2}[\|z-r_{k-1}\|_2^2 -\|r_k-r_{k-1}\|_2^2] - \tfrac{\alpha_k(\tau_k+\gamma_k)}{2} \|z-r_k\|_2^2 \nn\\
 &\quad  +\tfrac{\alpha_k\gamma_k}{2}[\|y-q_{0}\|_2^2 + \|z-r_{0}\|_2^2 ], \ \forall w \in X \times \bbr^m \times \bbr^d_+. \label{eq:Coex_temp1_ada_nonsmooth}
\end{align}
Moreover, by the definition of $\tilde h_k$ in \eqnok{eq:step_ex2_ada}, we have
\begin{align}
&\langle l_{h^k}(x_{k-1},p_k)- \tilde h_k, z - r_k\rangle  -\tfrac{\tau_k}{2} \|r_k-r_{k-1}\|_2^2 \nn \\
& = \langle l_{h^k}(x_{k-1},p_k)- l_{h^{k-1}}(x_{k-2},p_{k-1}), z - r_k\rangle -\lambda_k\langle l_{h^{k-1}}(x_{k-2},p_{k-1})- l_{h^{k-2}}(x_{k-3},p_{k-2}), z - r_{k-1}\rangle \nn \\
& \quad + \lambda_k\langle l_{h^{k-1}}(x_{k-2},p_{k-1})- l_{h^{k-2}}(x_{k-3},p_{k-2}), r_k - r_{k-1}\rangle - \tfrac{\tau_k}{2} \|r_k-r_{k-1}\|_2^2 \nn\\
& \leq \langle l_{h^{k}}(x_{k-1},p_k)- l_{h^{k-1}}(x_{k-2},p_{k-1}), z - r_k\rangle \nn \\
& \quad -\lambda_k\langle l_{h^{k-1}}(x_{k-2},p_{k-1})- l_{h^{k-2}}(x_{k-3},p_{k-2}), z - r_{k-1}\rangle \nn\\
&\quad +  \tfrac{6\lambda_k^2D_X^2 \bar M_{C,V}^2}{\tau_k}+ \tfrac{3 \lambda_k^2}{\tau_k} \tsum_{i=1}^d (\eta_i^{k-2} -\eta_i^{k-1})^2 D_{V_i}^4, \label{eq:Coex_temp3_ada}
\end{align}
where the last inequality follows from
\begin{align}
&\lambda_k\langle l_{h^{k-1}}(x_{k-2},p_{k-1})- l_{h^{k-2}}(x_{k-3},p_{k-2}), r_k - r_{k-1}\rangle - \tfrac{\tau_k}{2} \|r_k-r_{k-1}\|_2^2 \nn \\
&\le \tfrac{\lambda_k^2}{2\tau_k} \tsum_{i=1}^d [ l_{h_i^{k-1}}(x_{k-2},p_{k-1})- l_{h_i^{k-2}}(x_{k-3},p_{k-2})]^2 \nn \\
&= \tfrac{\lambda_k^2}{2\tau_k}\tsum_{i=1}^d [ h_i^{k-1}(x_{k-2})-h_i^{k-2}(x_{k-3}) + \langle \nabla h_i^{k-1}(x_{k-2}),p_{k-1}-x_{k-2}\rangle + \langle \nabla h_i^{k-2}(x_{k-3}),p_{k-2}-x_{k-3}\rangle]^2 \nn \\
&\le \tfrac{3 \lambda_k^2}{2\tau_k} \tsum_{i=1}^d \left[ ( h_i^{k-1}(x_{k-2})-h_i^{k-2}(x_{k-3}))^2 + 2M_{C_i, V_i}^2 D_X^2\rangle \right] \nn\\
&\le  \tfrac{3 \lambda_k^2}{2\tau_k} \tsum_{i=1}^d  \left[ 2 ( h_i^{k-2}(x_{k-2})-h_i^{k-2}(x_{k-3}))^2 + 2 (\eta_i^{k-2} -\eta_i^{k-1})^2 D_{V_i}^4 + 2M_{C_i, V_i}^2 D_X^2\rangle \right] \nn\\
&\le \tfrac{6\lambda_k^2D_X^2}{\tau_k}\tsum_{i=1}^dM_{C_i,V_i}^2 +  \tfrac{3 \lambda_k^2}{\tau_k} \tsum_{i=1}^d (\eta_i^{k-2} -\eta_i^{k-1})^2 D_{V_i}^4\nn \\
&=  \tfrac{6\lambda_k^2D_X^2 \bar M_{C,V}^2}{\tau_k} +  \tfrac{3 \lambda_k^2}{\tau_k} \tsum_{i=1}^d (\eta_i^{k-2} -\eta_i^{k-1})^2 D_{V_i}^4. \label{eq:Coex_du_temp4}
\end{align}
Here, the first inequality follows from Young's inequality, the second inequality follows from the cauchy-schwarz inequality,
the definition of $D_X$ in \eqnok{eq:def_D_X} and 
the bound of $\nabla h_i^k$ in \eqnok{eq:def_Lips_MC},
the third inequality follows by the relation between $h_i^{k-1}$ and $h_i^{k-2}$ in \eqnok{eq:app_relation_h}
and the simple fact that $(a+b)^2 \le 2a^2+2b^2$,
and the last inequality follows from the Lipschitz continuity of $h_i^{k-2}$ and the bound in \eqnok{eq:def_Lips_MC}. 
In addition, it follows from \eqnok{eq:app_relation_f} and \eqnok{eq:app_relation_h} that for any $w \in X \times \bbr^m \times \bbr^d_+$,
\beq \label{eq:relation_Q_ks}
Q^k(w_{k-1},w) \le Q^{k-1}(w_{k-1},w) + (\eta_0^{k-1} - \eta_0^k) D_U^2 + \tsum_{i=1}^d (\eta_i^{k-1} - \eta_i^k) z_i D_{V_i}^2.
\eeq
The result follows by combining \eqnok{eq:Coex_temp1_ada_nonsmooth}, \eqnok{eq:Coex_temp3_ada}, \eqnok{eq:relation_Q_ks}
and the bound in \eqnok{eq:Coex_temp2}.
\end{proof}

\vgap

\begin{theorem} \label{the_CoexDurCG_nonsmooth}
Let $\Gamma_k$ be defined in \eqref{Def_Gamma} and assume that
the algorithmic parameters $\alpha_k, \tau_k$ and $\lambda_k$ in the CoexDurCG method in Algorithm~\ref{algAdaCoexCG_nonsmooth} satisfy \eqnok{cond2}.
%\begin{equation}\label{cond3}
%\alpha_1 = 1, \ \tfrac{\lambda_k\alpha_k}{\Gamma_k} = \tfrac{\alpha_{k-1}}{\Gamma_{k-1}} \text{ and } \tfrac{\alpha_k\tau_k}{\Gamma_k} \leq \tfrac{\alpha_{k-1}(\tau_{k-1}+\gamma_{k-1})}{\Gamma_{k-1}}
%\ \forall k \ge 2.
%\end{equation}
Then we have, $\forall w \in X \times \bbr^m \times \bbr^d_+$,
\begin{equation}\label{thm2_gap_result_ada}
\begin{aligned}
Q(w_N,w) &\leq  \Gamma_N\tsum_{k=1}^N\left[\tfrac{(L_f^k+z^T L_h^k)\alpha_k^2D_X^2}{2\Gamma_k} +\tfrac{\alpha_k\lambda_k^2(12 \bar M_{C,V}^2+\|A\|^2)D_X^2}{2\tau_k\Gamma_k}+\tfrac{3\lambda_k^2}{\tau_k\Gamma_k}\tsum_{i=1}^d (\eta_i^{k-2}-\eta_i^{k-1})^2D_{V_i}^4\right] \\
&\quad+\Gamma_N\tsum_{k=1}^N\tfrac{\alpha_k}{\Gamma_k}(\eta_0^kD_U^2+\tsum_{i=1}^d\eta_i^kz_iD_{V_i}^2)+\tfrac{\alpha_N(12\bar M_{C,V}^2+\|A\|^2) D_X^2}{2(\tau_N+\gamma_N)}+\tfrac{6\alpha_N\tsum_{i=1}^d(\eta_i^{N-1}-\eta_i^N)^2D_{V_i}^4}{2(\tau_N+\gamma_N)}  \\
&\quad + \Gamma_N \left(\tfrac{\tau_1}{2}+\tsum_{k=1}^N \tfrac{\alpha_k\gamma_k}{2\Gamma_k}\right) (\|y-q_0\|_2^2 +\|z-r_0\|_2^2)+ \eta_0^ND_U^2+ \|z\|(\tsum_{i=1}^d(\eta_i^ND_{V_i}^2)^2)^{1/2},
\end{aligned}
\end{equation}
where $D_X$ is defined in \eqnok{eq:def_D_X}. As a consequence, we have
\begin{align}
f(x_N) - f(x^*) &\le \Gamma_N\tsum_{k=1}^N\left[\tfrac{L_f^k\alpha_k^2D_X^2}{2\Gamma_k} +\tfrac{\alpha_k\lambda_k^2(12 \bar M_{C,V}^2+\|A\|^2)D_X^2}{2\tau_k\Gamma_k}+\tfrac{3\lambda_k^2}{\tau_k\Gamma_k}\tsum_{i=1}^d (\eta_i^{k-2}-\eta_i^{k-1})^2D_{V_i}^4\right] \nn \\
&\quad+\Gamma_N\tsum_{k=1}^N\tfrac{\alpha_k}{\Gamma_k}\eta_0^kD_U^2+\tfrac{\alpha_N(12\bar M_{C,V}^2+\|A\|^2) D_X^2}{2(\tau_N+\gamma_N)}+\tfrac{6\alpha_N\tsum_{i=1}^d(\eta_i^{N-1}-\eta_i^N)^2D_{V_i}^4}{2(\tau_N+\gamma_N)} \nn \\
&\quad + \Gamma_N \left(\tfrac{\tau_1}{2}+\tsum_{k=1}^N \tfrac{\alpha_k\gamma_k}{2\Gamma_k}\right) (\|q_0\|_2^2 +\|r_0\|_2^2)+\eta_0^ND_U^2 \label{thm2_obj_result_ada}
\end{align}
and
\begin{align}
&\|g(x_N)\|_2 + \|[h(x_N)]_+\|_2 \nn \\
&\le   \Gamma_N\tsum_{k=1}^N\left[\tfrac{[L_f^k+(\|z^*\|_2+1) \bar L_h^k]\alpha_k^2D_X^2}
{2\Gamma_k} +\tfrac{\alpha_k\lambda_k^2(12 \bar M_{C,V}^2+\|A\|^2)D_X^2}{2\tau_k\Gamma_k}+\tfrac{3\lambda_k^2}{\tau_k\Gamma_k}\tsum_{i=1}^d (\eta_i^{k-2}-\eta_i^{k-1})^2D_{V_i}^4\right] \nn\\
&\quad+\Gamma_N\tsum_{k=1}^N\tfrac{\alpha_k}{\Gamma_k}(\eta_0^kD_U^2+\tsum_{i=1}^d\eta_i^k\hat z_iD_{V_i}^2)+\tfrac{\alpha_N(12\bar M_{C,V}^2+\|A\|^2) D_X^2}{2(\tau_N+\gamma_N)}+\tfrac{6\alpha_N\tsum_{i=1}^d(\eta_i^{N-1}-\eta_i^N)^2D_{V_i}^4}{2(\tau_N+\gamma_N)}  \nn\\
&\quad +\Gamma_N \left(\tau_1+\tsum_{k=1}^N \tfrac{\alpha_k\gamma_k}{2\Gamma_k}\right) [(\|y^*\|_2+1)^2 + \|q_0\|_2^2 + (\|z^*\|_2+1)^2 + \|r_0\|_2^2] \nn \\
&\quad +\eta_0^ND_U^2+ (\|z^*\|_2+1)(\tsum_{i=1}^d(\eta_i^ND_{V_i}^2)^2)^{1/2} ,
\label{thm2_feas_result_ada}
\end{align}
where $(x^*, y^*, z^*)$ denotes a triple of optimal solutions for problem~\eqnok{minmax}.
\end{theorem}

\begin{proof}
It follows from Lemma~\ref{Lemma 2}, Proposition~\ref{Prop_CoexDurCG_nonsmooth}  and  \eqref{cond2} that
%\begin{align*}
%\tfrac{Q^N(w_N,w)}{\Gamma_N} \leq&  (1-\alpha_1) Q^0(w_0,w)+ \tsum_{k=1}^N[\tfrac{(L_f^k+z^T L_h^k)\alpha_k^2D_X^2}{2\Gamma_k} +\tfrac{\alpha_k\lambda_k^2(12 \bar M_{C,V}^2+\|A\|^2)D_X^2}{2\tau_k\Gamma_k}+\tfrac{3\lambda_k^2}{\tau_k\Gamma_k}\tsum_{i=1}^d (\eta_i^{k-2}-\eta_i^{k-1})^2D_{V_i}^4]\\
%& + \tsum_{k=1}^N \tfrac{1-\alpha_k}{\Gamma_k}[(\eta_0^{k-1}-\eta_0^k)D_U^2+\tsum_{i=1}^d(\eta_i^{k-1}-\eta_i^k)z_iD_{V_i}^2] \\
%& + \tsum_{k=1}^N \tfrac{\alpha_k}{\Gamma_k} [ \langle A(p_k-p_{k-1}), y- q_k\rangle - \lambda_k \langle A (p_{k-1}-p_{k-2}), y - q_{k-1}\rangle]\\
%& + \tsum_{k=1}^N\tfrac{\alpha_k}{\Gamma_k}[\langle l_{h^{k}}(x_{k-1},p_k)-l_{h^{k-1}}(x_{k-2},p_{k-1}), z - r_k\rangle  \\
%& \quad \quad  - \lambda_k \langle l_{h^{k-1}}(x_{k-2},p_{k-1})-l_{h^{k-2}}(x_{k-3},p_{k-2}), z - r_{k-1}\rangle ]\\
%& + \tsum_{k=1}^N\left[\tfrac{\alpha_k\tau_k}{2\Gamma_k}(\|y-q_{k-1}\|_2^2+\|z-r_{k-1}\|_2^2) -\tfrac{\alpha_k(\tau_k+\gamma_k)}{2\Gamma_k}(\|y-q_k\|_2^2+\|z-r_k\|_2^2) \right] \\
%& + \tsum_{k=1}^N \tfrac{\alpha_k\gamma_k}{2\Gamma_k}[\|y-q_{0}\|_2^2 + \|z-r_{0}\|_2^2 ], 
%\end{align*}
%which, in view of \eqref{cond3}, then implies that
\begin{align*}
Q^N(w_N,w) &\leq  \Gamma_N \tsum_{k=1}^N[\tfrac{(L_f^k+z^T L_h^k)\alpha_k^2D_X^2}{2\Gamma_k} +\tfrac{\alpha_k\lambda_k^2(12 \bar M_{C,V}^2+\|A\|^2)D_X^2}{2\tau_k\Gamma_k}+\tfrac{3\lambda_k^2}{\tau_k\Gamma_k}\tsum_{i=1}^d (\eta_i^{k-2}-\eta_i^{k-1})^2D_{V_i}^4]\\
&\quad +\Gamma_N\tsum_{k=1}^N\tfrac{\alpha_k}{\Gamma_k}(\eta_0^kD_U^2+\tsum_{i=1}^d\eta_i^kz_iD_{V_i}^2)\\
&\quad +\alpha_N \langle A(p_N-p_{N-1}), y-q_N\rangle -\tfrac{\alpha_N(\tau_N+\gamma_N)}{2}\|y-q_N\|_2^2 \\
&\quad + \alpha_N\langle l_{h^{N}}(x_{N-1},p_N)-l_{h^{N-1}}(x_{N-2},p_{N-1}),z-r_N\rangle -\tfrac{\alpha_N(\tau_N+\gamma_N)}{2}\|z-r_N\|_2^2 \\
&\quad  + \tfrac{\alpha_1\tau_1\Gamma_N}{2}[\|y-q_{0}\|_2^2 + \|z-r_{0}\|_2^2 ] + \Gamma_N \tsum_{k=1}^N \tfrac{\alpha_k\gamma_k}{2\Gamma_k}[\|y-q_{0}\|^2 + \|z-r_{0}\|^2 ]\\
& \leq   \Gamma_N \tsum_{k=1}^N[\tfrac{(L_f^k+z^T L_h^k)\alpha_k^2D_X^2}{2\Gamma_k} +\tfrac{\alpha_k\lambda_k^2(12 \bar M_{C,V}^2+\|A\|^2)D_X^2}{2\tau_k\Gamma_k}+\tfrac{3\lambda_k^2}{\tau_k\Gamma_k}\tsum_{i=1}^d (\eta_i^{k-2}-\eta_i^{k-1})^2D_{V_i}^4]\\
&\quad +\Gamma_N\tsum_{k=1}^N\tfrac{\alpha_k}{\Gamma_k}(\eta_0^kD_U^2+\tsum_{i=1}^d\eta_i^kz_iD_{V_i}^2) \\
&\quad +\tfrac{\alpha_N}{2(\tau_N+\gamma_N)} \|A\|^2\|p_N-p_{N-1}\|_2^2 +\tfrac{12 \bar M_{C,V}^2 \alpha_N D_X^2}{2(\tau_N+\gamma_N)}+\tfrac{6\alpha_N\tsum_{i=1}^d(\eta_i^{N-1}-\eta_i^N)^2D_{V_i}^4}{2(\tau_N+\gamma_N)} \\
& \quad  + \tfrac{\alpha_1\tau_1\Gamma_N}{2}[\|y-q_{0}\|_2^2+ \|z-r_{0}\|_2^2]\\
&\quad  + \Gamma_N \tsum_{k=1}^N \tfrac{\alpha_k\gamma_k}{2\Gamma_k}[\|y-q_{0}\|_2^2 + \|z-r_{0}\|_2^2 ],
\end{align*}
where the last relation follows from Young's inequality and 
a result similar to \eqnok{eq:Coex_du_temp4}. 
The result in \eqnok{thm2_gap_result_ada} then immediately follows from the above inequality
and the observation that
$
Q(w^N,w) \leq Q^N(w^N,w) + \eta_0^ND_U^2+ \|z\|(\tsum_{i=1}^d(\eta_i^ND_{V_i}^2)^2)^{1/2}
$
due to  \eqnok{eq:bnd_Q_eta_close}.
We can show \eqnok{thm2_obj_result_ada} and \eqnok{thm2_feas_result_ada} similarly to
\eqnok{eq:nonsmooth_obj_result} and \eqnok{eq:nonsmooth_feas_result}, and hence the details are skipped.
\end{proof}

Corollary~\ref{Cor_CoexDurCG_nonsmooth} below shows how to specify the smoothing parameter $\{\eta^k_i\}$ in \eqref{eq:decreasing_eta} and other parameters
for the CoexDurCG method in Algorithm~\ref{algAdaCoexCG_nonsmooth}. We focus on the most challenging case when the objective function $f$ and all the constraint functions are
nonsmooth (i.e., $\mu_i = 0$, $i =1, \ldots, n$). Slightly improved rate of convergence can be obtained
by setting $\eta_i^k = 0$ for those component functions with $\mu_i > 0$. 
%In particular, the selection of $\tau_k$ was inspired by the one used in \eqnok{para_CoexDur_main},
%and $\gamma_k$ was chosen so that the last relation in \eqnok{cond3} is satisfied. 

\begin{corollary}\label{Cor_CoexDurCG_nonsmooth}
Suppose that the parameters $\alpha_k$, $\lambda_k$, $\tau_k$ and $\gamma_k$
in Algorithm~\ref{algAdaCoexCG_nonsmooth} are set to \eqnok{para_CoexDur_main} with $\beta = D_X \sqrt{12\bar M_{C,V}^2 +\|A\|^2}$ for $k \ge 1$.
If the smoothing parameters $\eta_i^k$ are set to
\begin{equation}\label{eta_def_ada}
\eta_0^k = \tfrac{\|B\|D_X}{\sqrt{k}D_U}, \ \eta_i^k = \tfrac{\|C_i\|D_X}{\sqrt{k}D_{V_i}}, \ \forall i =1,\ldots, d,
\end{equation}
%the algorithmic parameters $\alpha_k$, $\lambda_k$, $\tau_k$ and $\gamma_k$
%of the CoexDurCG method are set to
%\begin{equation}\label{para_CoexDur_main_ada}
%\alpha_k = \tfrac{2}{k+1}, \lambda_k = \tfrac{k-1}{k},\ \tau_k = \beta \sqrt{k}, \ \mbox{and} \  \gamma_k = \tfrac{\beta}{k}[(k+1)\sqrt{k+1} - k\sqrt{k}],
%\end{equation}
%with $\beta = D_X \sqrt{12\bar M_{C,V}^2 +\|A\|^2}$ for $k \ge 1$, 
then
we have, $ \forall w \in X \times \bbr^m \times \bbr^d_+$,
\begin{equation}\label{cor 5}
\begin{aligned}
Q(w_k,w) \leq & \tfrac{8 (\|B\|D_U+\tsum_{i=1}^dz_i\|C_i\|D_{V_i}) D_X}{3\sqrt{N}} +\tfrac{\sqrt{12\bar M_{C,V}^2 +\|A\|^2}D_X}{\sqrt{N}} [2(\|y-q_0\|^2+\|z-r_0\|^2)+2]\\
& \quad + \tfrac{12\tsum_{i=1}^d\|C_i\|^2D_XD_{V_i}^2}{\sqrt{12\bar M_{C,V}^2 +\|A\|^2}(N+1)\sqrt{N}} + \tfrac{D_X}{\sqrt{N}}(\|B\|D_U+\|z\|\sqrt{\tsum_{i=1}^d\|C_i\|^2D_{V_i}^2}).
\end{aligned}
\end{equation}
In addition, we have
\begin{align}
f(x_N) - f(x^*) &\le\tfrac{11 \|B\|D_U D_X}{3\sqrt{N}} +\tfrac{\sqrt{12\bar M_{C,V}^2 +\|A\|^2}D_X}{\sqrt{N}} [2(\|q_0\|^2+\|r_0\|^2)+2] \nn \\
& \quad + \tfrac{12\tsum_{i=1}^d\|C_i\|^2D_XD_{V_i}^2}{\sqrt{12\bar M_{C,V}^2 +\|A\|^2}(N+1)\sqrt{N}} \label{cor_obj_result_nonsmooth}
\end{align}
and
\begin{align}
&\|g(x_N)\|_2 + \|[h(x_N)]_+\|_2 \le \tfrac{7 (\|B\|D_U+(\|z^*\|+1)\sqrt{\tsum_{i=1}^d\|C_i\|^2D_{V_i}^2}) D_X}{3\sqrt{N}}
 + \tfrac{12\tsum_{i=1}^d\|C_i\|^2D_XD_{V_i}^2}{\sqrt{12\bar M_{C,V}^2 +\|A\|^2}(N+1)\sqrt{N}} \nn \\
& +\tfrac{2\sqrt{12\bar M_{C,V}^2 +\|A\|^2}D_X}{\sqrt{N}}[4[(\|y^*\|_2+1)^2 + (\|z^*\|_2+1)^2  + \|q_0\|_2^2+ \|r_0\|_2^2] +2],\label{cor_feas_result_nonsmooth}
\end{align}
where $(x^*, y^*, z^*)$ denotes a triple of optimal solutions for problem~\eqnok{minmax}.
\end{corollary}

\begin{proof}
From the definition of $\alpha_k$ in \eqnok{para_CoexDur_main}, we have $\Gamma_k = 2/[k(k+1)]$ and $\alpha_k / \Gamma_k = k$.
Similarly to Corollary~\ref{Cor_CoexDurCG}, we can check that condition~\eqnok{cond2}, and the bounds in
\eqnok{eq:coexduCG_gamma_1}-\eqnok{eq:coexduCG_gamma_3} hold. In addition,
it follows from the definition of $\eta_i^k$ in \eqref{eta_def_ada} that
\begin{align*}
(\eta_i^{k-2}-\eta_i^{k-1})^2 = \tfrac{\|C_i\|^2D_X^2}{D_{V_i}^2}(\tfrac{1}{k-1}+\tfrac{1}{k-2}-\tfrac{2}{\sqrt{k-1}\sqrt{k-2}}) \leq  \tfrac{\|C_i\|^2D_X^2}{(k-1)(k-2)D_{V_i}^2},\\
L_f^k = \tfrac{\|B\|D_U\sqrt{k}}{D_X}, \ \mbox{and} \ L_{h,i}^k = \tfrac{\|C_i\|D_{V_i}\sqrt{k}}{D_X},\ \forall i =1,\ldots,d.
\end{align*}
%From the definition of $\alpha_k$ in \eqnok{para_CoexDur_main_ada}, we have $\Gamma_k = 2/[k(k+1)]$ and $\alpha_k / \Gamma_k = k$.
%Hence the first two conditions in \eqnok{cond3} hold. In addition, it follows from these identities and \eqnok{para_CoexDur_main_ada}
%that $\tfrac{\alpha_k\tau_k}{\Gamma_k} = \beta k\sqrt{k}$ and
%\begin{align*}
%\tfrac{\alpha_{k-1}(\tau_{k-1}+\gamma_{k-1})}{\Gamma_{k-1}} = (k-1)\left [\beta \sqrt{k-1} + \tfrac{\beta}{k-1}[k\sqrt{k} - (k-1)\sqrt{k-1}]\right] = \beta k \sqrt{k}
%\end{align*}
%and hence the last relation in \eqnok{cond3} also holds.
%Observe that by \eqnok{para_CoexDur_main_ada},
%\begin{align*}
%\tsum_{k=1}^N \tfrac{\alpha_k^2}{\Gamma_k} &= 2\tsum_{k=1}^N \tfrac{k}{k+1} \le 2 N,\\
% \tsum_{k=1}^N\tfrac{\alpha_k\gamma_k}{\Gamma_k} &= \beta\tsum_{k=1}^N[(k+1)\sqrt{k+1}-k\sqrt{k}] = \beta[(N+1)\sqrt{N+1}-1],\\
% \tsum_{k=1}^N\tfrac{\alpha_k\lambda_k^2}{\tau_k\Gamma_k} &= \tsum_{k=1}^N\tfrac{(k-1)^2}{\beta k\sqrt{k}}\leq \tfrac{1}{\beta} \tsum_{k=1}^N \sqrt{k-1}
% \le \tfrac{1}{\beta} \int_0^N \sqrt{t} dt = \tfrac{2}{3\beta}N^{3/2}.
%\end{align*}
Using these relations in \eqnok{thm2_gap_result_ada}, we have
\begin{align*}
Q(w_N,w) &\leq \tfrac{2 (\|B\|D_U+\tsum_{i=1}^dz_i\|C_i\|D_{V_i}) D_X}{N(N+1)}\tsum_{k=1}^N \tfrac{k\sqrt{k}}{k+1} + 
\tfrac{2\sqrt{12\bar M_{C,V}^2 +\|A\|^2}D_X}{3\sqrt{N}} \\
& \quad + \tfrac{3}{\beta N(N+1)}\tsum_{k=1}^N (\sqrt{k}(k+1)\tsum_{i=1}^d \tfrac{\|C_i\|^2D_X^2D_{V_i}^2}{(k-1)(k-2)}) \\
& \quad + \tfrac{2 (\|B\|D_U+\tsum_{i=1}^dz_i\|C_i\|D_{V_i}) D_X}{N(N+1)}\tsum_{k=1}^N\sqrt{k} + \tfrac{\sqrt{12\bar M_{C,V}^2 +\|A\|^2}D_X}{(N+1)\sqrt{N+1}} \\
& \quad +\tfrac{3\tsum_{i=1}^d\|C_i\|^2D_X^2D_{V_i}^2}{(N+1)\sqrt{N+1}(N-1)(N-2)} + \tfrac{\beta}{N(N+1)}[(N+1)\sqrt{N+1}][\|y-q_0\|^2+\|z-r_0\|^2] \\
& \quad + \eta_0^ND_U^2+ \|z\|(\tsum_{i=1}^d(\eta_i^ND_{V_i}^2)^2)^{1/2},
%& \leq \tfrac{8 (\|B\|D_U+\tsum_{i=1}^dz_i\|C_i\|D_{V_i}) D_X}{3\sqrt{N}} +\tfrac{\sqrt{12\bar M_{C,V}^2 +\|A\|^2}D_X}{\sqrt{N}} [2(\|y-q_0\|^2+\|z-r_0\|^2)+2]\\
%& \quad + \tfrac{12\tsum_{i=1}^d\|C_i\|^2D_XD_{V_i}^2}{\sqrt{12\bar M_{C,V}^2 +\|A\|^2}(N+1)\sqrt{N}}+ \eta_0^ND_U^2+ \|z\|(\tsum_{i=1}^d(\eta_i^ND_{V_i}^2)^2)^{1/2}.
\end{align*}
which implies \eqref{cor 5} after simplification. \eqnok{cor_obj_result_nonsmooth} and \eqnok{cor_feas_result_nonsmooth} can be shown similarly and the details are skipped.
%Therefore, from \eqref{thm2}, we have
%\begin{align*}
%Q(z_N,z) \leq & \tfrac{2N(9\|M_h\|^2+\|A\|^2)D_X^2}{\beta(N+1)^2\sqrt{N+1}}+\tfrac{\|d-r_0\|^2+\|y-q_0\|^2}{N(N+1)}+\tfrac{2}{N(N+1)}\tsum_{k=1}^N \tfrac{k}{k+1}(L_f+d^TL_h)D_X^2 \\
%& + \tfrac{2}{N(N+1)}\tfrac{1}{2\beta}N^{3/2}(6\|M_h\|^2+\|A\|^2)D_X^2+ \tfrac{\beta[(N+1)\sqrt{N+1}-1]}{N(N+1)}[\|d-r_0\|^2+\|y-q_0\|^2] \\
%\leq & \tfrac{2(9\|M_h\|^2+\|A\|^2)D_X^2}{\beta(N+1)\sqrt{N+1}}+\tfrac{\|d-r_0\|^2+\|y-q_0\|^2}{N(N+1)}+\tfrac{2}{N+1}(L_f+d^TL_h)D_X^2 \\
%& + \tfrac{(6\|M_h\|^2+\|A\|^2)}{\sqrt{N}\beta}D_X^2+ \tfrac{\beta}{\sqrt{N+1}}[\|d-r_0\|^2+\|y-q_0\|^2] 
%\end{align*}
%Take $\beta = \sqrt{9\|M_h\|^2+\|A\|^2}D_X$, we have \eqref{cor 4}.
\end{proof}

Comparing the results in Corollary~\ref{Cor_CoexDurCG_nonsmooth} with those in Corollary~\ref{cor_nonsmooth},
we can see that the rate of convergence of CoexDurCG is about the same as that of CoexCG for nonsmooth
optimization. However, it is more convenient to implement CoexDurCG since it does not
require us  to fix the number of iterations a priori. 

\section{Numerical Experiments} \label{sec_num}
In this section, we  apply the proposed algorithms to the
intensity modulated radiation therapy (IMRT) problem briefly
discussed in Section~1. %and report some preliminary numerical results.
\subsection{Problem Formulation}
In IMRT, the patient will be irradiated by a linear accelerator (linac) from several angles and in each angle the device uses different apertures. 
In traditional IMRT,  we select and fix 5-9 angles and then design and optimize the apertures and their corresponding intensity. 
Following \cite{romeijn2005column},  we would like to integrate the angle selection into direct aperture optimization in order to use a small number of  angles and apertures in the final treatment plan.

To model the IMRT treatment planning, we discretize each structure $s$ of the patient into small cubic volume elements called \textit{voxels}, $\mathcal{V}$.
There are a finite number of angles, denoted by ${\mathcal A}$,  around the patient.
A beam in each angle, $b_a$, is decomposed into a rectangular grid of \textit{beamlets}. A beamlet $(i, j)$ is effective if it is not blocked by either the left, $l_i$, and right, $r_i$, leaves.
An aperture is then defined as the collection of effective beamlets.
The relative motion of the leaves controls  the set of effective beamlets and thus the shape of the aperture.
The estimated dose received by voxel $v$ from beamlet $(i, j)$ at unit intensity is denoted by $D_{(i, j) v }$ in Gy.
The dose absorbed by a given voxel is the summation of the dose from each individual beamlet.

Let $P_a$ be the set of allowed apertures determined by the position of the left and right leaves in beam angle $a$. 
Suppose that the rectangular grid in each angle has $m$ rows and $n$ columns, and the leaves move along each row
independently. Then the number of possible apertures in each angle amounts to $(\tfrac{n (n-1)}{2})^m$.
We use $\textbf{x}^{a,t}$, comprised of binary decision variables  $x_{(i, j)}^{a,t} $, to describe the shape of aperture $t \in P_a$.
In particular, $x_{(i, j)}^{a,t }= 1$ if beamlet $(i, j)$ is effective, i.e., falling within the left and right leaves of row $i$,
otherwise  $x_{(i, j)}^{a,t} = 0$. In addition to selecting angles and apertures, we also
need to determine the influence rate $y^{a,t}$ for aperture $t\in P_a$, which will be used to determine 
the dose intensity and the amount of radiation time from aperture $t$.
%A basic IMRT treatment planning problem can be modeled as follows:
%\begin{subequations} \label{eq:GP}
%\begin{align}
%%\textbf{GP}: \quad
%\min \quad
%& \textit{f}(\textbf{z}):= \sum_{v \in \mathcal{V}} {\underline{w}_v} \, [\underline{d}_{v} - z_{v}]_{+}^2 + {\overline{w}_v} \, [z_{v} - \overline{d}_{v}]_{+}^2 \label{eq:Model_GP2_obj}\\
%{\rm s.t.} & \ \ \ z_{v}= \tsum\limits_{a \in \mathcal{A}} \tsum\limits_{k \in P_a} \tsum\limits_{i=1}^{m} \tsum\limits_{j=1}^{n} R D_{(i, j) v} \, x_{ij}^{a,t} y^{a,t} & \forall v \in \mathcal{V}, \label{eq:CZinGP2}\\
%& \ \ \ \tsum\nolimits_{a \in \mathcal{A}} \tsum\nolimits_{k \in P_a} y^{a,t} \leq 1.   &  \label{eq:sumYinGP2}
%\end{align}
%\end{subequations}
%where
The dose absorbed by voxel $v$ is computed by $z_{v}= \tsum_{a \in \mathcal{A}} \tsum_{t \in P_a} \tsum_{i=1}^{m} \tsum_{j=1}^{n} R D_{(i, j) v} \, x_{ij}^{a,t} y^{a,t} $, based on the dose-influence matrix $D$, the aperture shape $\textbf{x}^k$, and the aperture influence rate $y^k$.
We measure the treatment quality by
$\textit{f}(\textbf{z}):= \sum_{v \in \mathcal{V}} {\underline{w}_v} \, [\underline{T}_{v} - z_{v}]_{+}^2 + {\overline{w}_v} \, [z_{v} - \overline{T}_{v}]_{+}^2$  
via voxel-based  quadratic penalty, where $[\cdot]_{+}$ denotes $\max\{0, \cdot \}$, 
and $\underline{T}_v$ and $\overline{T}_v$ are pre-specified lower and upper dose thresholds for voxel $v$. 

We also need to consider a few important function constraints. Firstly, 
in order to obtain a sparse solution with a small number of angles, we add the following group sparsity constraint 
%\begin{equation}
 $\tsum_{a\in \mathcal{A}} \max_{t \in P_a} y^{a,t} \leq \Phi$
%\end{equation}
for some properly chosen $\Phi > 0$.
Intuitively, this constraint will encourage the selection of apertures in those angles $P_a$ that have already contained some nonzero 
elements of $y^{a,t}$, $t \in P_a$. Secondly, we need to 
meet a few critical clinical criteria to avoid underdose (resp., overdose) for tumor (resp., healthy) structures. 
These criteria are usually specified as
value at risk (VaR) constraints. For example, in 
the prostate benchmark dataset,  the clinical criterion of ``PTV56:V56$\geq 95\%$" means that the percentage of voxels 
in structure PTV56 that receive at least 56 Gy dose 
should be at least $95\%$. Similarly, the criterion of  ``PTV68: V74.8$\leq 10\%$" implies 
that the percentage of voxels in structure PTV68 that receive more than 74.8 Gy dose 
should be at most $10\%$. One possible way to satisfy these criteria
is to tune the weights ($(\underline w_v, \overline w_v)$) in $\textit{f}(\textbf{z})$. However, it would be time consuming
to tune
these weights to satisfy all the prescribed clinical criteria. Therefore, we suggest to
incorporate a few critical criteria as problem constraints explicitly.

%However, the objective function \eqref{eq:Model_GP2_obj} is hard to capture the critical requirements in the clinical criteria, e.g., for our Prostate dataset, we separate the clinical criteria into two classes, constraints to avoid underdose and constraints to avoid overdose.

%In clinical criteria, we specify 
%Take the first constraint V56$\geq 95\%$ as an example, that requires , which makes the objective function \eqref{eq:Model_GP2_obj} not so efficient in this purpose. Thus, we can either spend much effort on tuning the weights in \eqref{eq:Model_GP2_obj} or add some other constraints (e.g., chance constraints) to force these criteria be satisfied.
%
%\subsection{CVaR Constraints}
Instead of using VaR, we will use its convex approximation,
commonly referred to as Conditional Value at Risk (CVaR) in the constraints~\cite{roc00}. Recall the following definitions of VaR and CVaR
%\begin{subequations}
\begin{align*}
\text{Upper tail: } & {\rm VaR}_\alpha(X) = \inf_\tau \{\tau: P(X\leq \tau)\geq \alpha\},  {\rm CVaR}_\alpha (X) = \inf_\tau \tau +\tfrac{1}{1-\alpha}\bbe[X - \tau]_+.\\
\text{Lower tail: } & {\rm VaR}_\alpha(X) = \sup_\tau \{\tau: P(X\geq \tau)\geq \alpha\}, {\rm CVaR}_\alpha (X) = \sup_\tau \tau - \tfrac{1}{1-\alpha}\bbe[\tau - X]_+.
\end{align*}
%\end{subequations}
The upper (resp., lower) tail CVaR will be used to enforce the underdose (resp., overdose)
clinical criteria. 
For example, letting $S_1$ and $S_2$ denote structures PTV68 and PTV 56, and
$N_1$ and $N_2$ be the number of voxels in these structures, we can approximately formulate
 the criterion of ``PTV68: V74.8$\leq 10\%$" as
$
 \inf_{\tau} \tau_1+\tfrac{1}{(1-0.9)N_{1}}\tsum_{v\in S_1} [z_v - \tau_1]_+ \leq b
$
for some $b \ge 74.8$. Separately, the criterion of ``PTV56:V56$\geq 95\%$" will be approximated by
% $ {\rm CVaR}_\alpha (X) \ge b$ and $ {\rm CVaR}_\alpha (X) \ge b$
%The corresponding $CVaR_\alpha(X)$ is defined as the following,
%\begin{subequations}
%\begin{align}
%\text{Upper tail: } & CVaR_\alpha (X) = \inf_\tau \tau +\tfrac{1}{1-\alpha}\bbe[X - \tau]_+. \\
%\text{Lower tail: } & CVaR_\alpha (X) = \sup_\tau \tau - \tfrac{1}{1-\alpha}\bbe[\tau - X]_+.
%\end{align}
%\end{subequations}
%Take criteria of PTV68 as an example, we can form the criteria $V68\geq 95\%$ as
$
\sup_{\tau} \tau - \tfrac{1}{(1-0.95)N_{2}}\tsum_{v\in S_2} [\tau - z_v]_+ \geq b, 
$
or equivalently
$
\inf_{\tau} -\tau + \tfrac{1}{(1-0.95)N_{2}}\tsum_{v\in S_2} [\tau - z_v]_+ \leq -b
$
for some $b \le 56$.
Putting the above discussions together and
denoting $\hat D_{v}^{a,t} := \tsum_{i=1}^{m} \tsum_{j=1}^{n}  D_{(i, j) v}\, x_{ij}^{a,t} $,
we obtain the following problem formulation.
\begin{subequations}
\begin{align}
%\textbf{GP$\_$CVaR}: & \nonumber \\
\min\quad \textit{f}(\textbf{z}) := &\tfrac{1}{N_v} \tsum_{v \in \mathcal{V}} {\underline{w}_v} \, [\underline{T}_{v} - z_{v}]_{+}^2 + {\overline{w}_v} \, [z_{v} - \overline{T}_{v}]_{+}^2 \label{obj}\\
\text{ s.t. }\quad & z_{v}= \tsum\limits_{a \in \mathcal{A}} \tsum\limits_{t \in P_a}R \hat D_{v}^{a,t} y^{a,t},  \\
& -\tau_i + \tfrac{1}{p_iN_i}\tsum_{v\in S_i}[\tau_i- z_v]_+ \leq -b_i, \forall i \in {\rm UD}, \label{cvar1}\\
& \tau_i + \tfrac{1}{p_iN_i}\tsum_{v\in S_i}[z_v - \tau_i]_+ \leq b_i, \forall i\in {\rm OD}, \label{cvar2} \\
& \tsum\nolimits_{a \in \mathcal{A}}\max_{t \in P_a} y^{a,t} \leq \Phi, \label{Group_s}\\
& \tsum\nolimits_{a \in \mathcal{A}} \tsum\nolimits_{t \in P_a} y^{a,t} \leq 1,  \label{simplex} \\
&  y^{a,t} \geq 0, \label{simplex0}\\
& \tau_i \in [\underline \tau_i, \bar \tau_i], \forall i\in {\rm UD}\ \&\ {\rm OD}, \label{tau_bound}
\end{align}
\end{subequations}
where OD and UD denote the set of overdose and underdose clinical criteria, respectively. Clearly, the objective function $f$ is convex and smooth.
Constraints in \eqnok{cvar1}, \eqnok{cvar2} and \eqnok{Group_s} are structured nonsmooth function constraints corresponding to the
function constraints $h$ in \eqnok{general},
while \eqnok{simplex}-\eqnok{simplex0} and  \eqnok{tau_bound},
respectively, define a simplex constraint on $y$ and
a box constraint on $\tau_i$, with their Catesian product 
corresponding to the convex set $X$ in \eqnok{general}. The bounds $\underline\tau$ and $\bar\tau$
in constraints \eqref{tau_bound} can 
be obtained from the corresponding clinical criteria.
For example, the criterion of ``PTV68:V68$\geq 95\%$" implies that value at risk $\geq 68$. By the definition of CVaR,
the optimal $\tau$ equals to the value at risk, hence we set $\underline \tau = 68$.
In a similar way, we set $\bar \tau = 74.8$ in view of the criterion of ``PTV68: V74.8$\leq 10\%$". %Thus, we have a convex optimization problem with some nonsmooth constraints. Using the smoothing approximation \eqref{eq:def_smoothapp_f} and \eqref{eq:def_smoothapp_h}, we can apply the CoexCG and CoexDurCG algorithms.

We can apply
the CoexCG and CoexDurCG methods described in Subsections~\ref{sec_nonsmooth_basic} and ~\ref{sec_ada_nonsmooth}, respectively,
to solve problem \eqref{obj}-\eqref{tau_bound}.
Since the number of the potential apertures (i.e., the dimension of $y^{a,t}$) increases exponentially w.r.t. $m$,
we cannot compute the full gradient of the objective and constraint functions w.r.t. $y^{a,t}$. Instead, we will perform gradient computation
and linear optimization simultaneously. Let us focus on the CoexCG method for illustration. Denote the constraints~\eqref{cvar1}-\eqref{Group_s} as $h_{i}$, $i \in OD \cup UD$,
and let the corresponding smooth approximation $h_{i,\eta_i}$ be defined by \eqref{eq:def_smoothapp_h} (using entropy distances for smoothing). 
For a given search point $x_{k-1} := (\{y_{k-1}^{a,t}\}, \{\tau_{i,k-1}\})$ and dual variable $\{r_{i,k-1}\}$, let us 
denote $\pi^f_{k-1} = \partial f(\textbf{x}_{k-1})/\partial \textbf{z}$ and $\pi^{h_i}_{k-1} = \partial h_{i,\eta_i}(\textbf{x}_{k-1})/\partial \textbf{z}$. 
Clearly, in view of~\eqref{eq:step_lo},
$y_{k-1}^{a,t}$ will be updated to a properly chosen extreme point of the  simplex constraint in \eqref{simplex}-\eqref{simplex0}.
In order to determine this extreme point, we need to find the aperture with the smallest 
 coefficient in the linear objective of \eqref{eq:step_lo} given by:
 $$\psi^{a,t}:= \pi_{k-1}^f\tfrac{\partial z}{\partial y^{a,t}} + \tsum_i r_{i,k-1} \pi_{k-1}^{h_i}\tfrac{\partial z}{\partial y^{a,t}} 
 = R \tsum_{i=1}^{m} \tsum_{j=1}^{n} (\tsum_v D_{(i,j)v}(\pi^f_{v,k-1}+\tsum_i r_{i,k-1}\pi^{h_i}_{v,k-1}))x_{ij},\ x_{ij}\in \{0,1\}.$$
%corresponding to coefficients of the linear objective in \eqref{eq:ste_lo}.
%Thus, from step~\ref{eq:step_lo}, we are looking for 
%\begin{align*} 
%y_k = & \argmin_{y\in Y} \{ l_f(x_{k-1}, y) +  \langle l_h(x_{k-1},y),  r_k\rangle \} \\
%=& \argmin_{y\in Y}\{ \langle \pi_{k-1}^f \tfrac{\partial z}{\partial y^{a,t}} , y\rangle+ \tsum_i \langle \pi^{h_i}_{k-1}\tfrac{\partial z}{\partial y^{a,t}}, y\rangle \}
%\end{align*}
%which, in vew of simplex constraint \eqref{simplex} and \eqref{simplex0}, implies the most negative 
% $$\psi^{a,t}:= \pi_{k-1}^f\tfrac{\partial z}{\partial y^{a,t}} + \tsum_i r_{i,k} \pi_{k-1}^{h_i}\tfrac{\partial z}{\partial y^{a,t}} = R \tsum_{i=1}^{m} \tsum_{j=1}^{n} (\tsum_v D_{(i,j)v}(\pi^f_{v,k-1}+\tsum_i r_{i,k}\pi^{h_i}_{v,k-1}))x_{ij},\ x_{ij}\in \{0,1\}.$$
%while most existing first-order methods, including ConEx in \cite{BoobDengLan19-1}, require full gradient information of every coordinate, and hence
%are not applicable to this problem. On the other hand, each iteration of 
% the CoexCG and CoexDurCG method requires the solution of minimizing a linear function over the simplex constraint (c.f. \eqref{eq:step_lo}).  
This can be achieved by using the following constructive approach.
For any row $i$ of the rectangular grid in angle $a$, we find the column indices $c_1$ and $c_2$, respectively,
for the left and right leaves, that give the most negative value of $\tsum_{c_1<j<c_2} \tsum_vD_{(i,j)v}(\pi^f_{v,k-1}+\tsum_i r_{i,k}\pi^{h_i}_{v,k-1})$.
Repeating this process row by row, we construct the aperture with the smallest value of $\psi^{a,t}$ in angle $a$.
We construct one aperture similar to this for each angle, and then choose the one with the most negative value of $\psi^{a,t}$ among all the angles. 
%{\color{blue}
Therefore, to solve the linear optimization suproblem (i.e., to find the aperture with the smallest coefficient) only requires ${\cal O}\{ |{\mathcal A}| mn(n-1)\}$
arithmetic operations, even though the dimension of the problem (i.e., the total number of apertures) is given by $|{\mathcal A} |( n(n-1)/2)^m$.
%as described in more details in Section 3.3.1 of \cite{romeijn2005column}.

%\begin{subequations}
%\begin{align*}
%\textbf{GP$\_$CVaR2}: & \nonumber \\
%\min\quad  & F(y):=  \tsum_{v \in \mathcal{V}} {\underline{w}_v} \, [\underline{T}_{v} - z_{v}]_{+}^2 + {\overline{w}_v} \, [z_{v} - \overline{T}_{v}]_{+}^2 \\
%\text{ s.t. }\quad & z_{v}= \tsum\limits_{a \in \mathcal{A}} \tsum\limits_{k \in P_a}R \hat D_{v}^{a,t} y^{a,t} ,  \\
%&  -\tau_i + \tfrac{1}{p_iN_i}\tsum_{v\in S_i}[ \nu_{i,v}^1(\tau_i-z_v) - \eta_i\tsum_{j=0}^1 \nu_{i,v}^j \log( \nu_{i,v}^j )] \leq -b_i, \forall i=1,2 \\
%& \tau_i + \tfrac{1}{p_iN_i}\tsum_{v\in S_i}[ \nu_{i,v}^1(z_v - \tau_i)- \eta_i\tsum_{j=0}^1 \nu_{i,v}^j \log( \nu_{i,v}^j )] \leq b_i, \forall i\geq 3 \\
%& \tsum\nolimits_{a \in \mathcal{A}} \tsum\nolimits_{k \in P_a} y^{a,t} \leq 1    \\
%&  \tsum\nolimits_{a \in \mathcal{A}}\tsum_{k=1}^{ N_a}(u^{a,t} y^{a,t} -\eta_0 u^{a,t}\log(u^{a,t}))\leq \Phi\\
%&  y^{a,t} \geq 0,
%\end{align*}
%\end{subequations}
%where $u^{a,t}$, $\nu_{i,v}^j$ have the explicit formulation
%\begin{align*}
%u^{a,t} = \tfrac{\exp\{y^{a,t}/\eta_0\}}{\tsum_{j\in P_a} \exp\{y^{j,a}/\eta_0\}},\  & \nu_{i,v}^0 = 1,
%\ \nu_{i,v}^1 = \tfrac{\exp\{(\tau_i-z_v)/\eta_i\}}{1+ \exp\{(\tau_i-z_v)/\eta_i\}}, \forall i = 1,2 \\
%& \nu_{i,v}^0 = 1,\ \nu_{i,v}^1 = \tfrac{\exp\{(z_v-\tau_i)/\eta_i\}}{1+ \exp\{(z_v-\tau_i)/\eta_i\}}, \forall i \geq 3.
%\end{align*}

\subsection{Comparison of CoexCG and CoexDurCG on randomly generated instances}
Due to the privacy issue, publicly available IMRT datasets for real patients are very limited. 
To test the performance of our proposed algorithms we first randomly generate some problem instances as follows.
Let $V= [-l,l]^3 \subseteq \bbr^3$ be a cube with length $l$. Viewing $V$ as the human body, we then arbitrarily choose two (or more) cuboids as healthy organs, 
and randomly choose 2 cubes inside $V$ as the target tumor tissues. For a given accuracy $\delta>0$, we discretize all these structures 
into small cubes with length $\delta$ to define a voxel. Around the cube $V$, we generate a circle with radius $2l$ on the plane $\{x = 0\}$, and define
 every two degrees as one angle for radiation therapy. In each angle, we consider the aperture as a square in $[-l,l]^2$, and
 also discretize it with small squares with length $\delta$, resulting in a grid with size $\tfrac{2l}{\delta}\times \tfrac{2l}{\delta}$. After that, we randomly generate $N_a$ beamlets with coordinate $(x',y') \in [-l,l]^2$ for each angle $a$. As for the matrix $D$ (recording the dose received by voxel $v$ from each beamlet), we first check if the voxel is radiated by the beamlet since each beamlet is a line perpendicular 
 to the aperture plane. If so, the dose received by the voxel from this beamlet will be set to $2/d$, where $d$ is the distance between the voxel and the aperture plane; otherwise, the dose is $0$.
By choosing different accuracy $\delta$, we can create instances with different sizes
in terms of the number of voxels and potential apertures. 
Table~\ref{Rand_data} shows five different test instances generated with $l= 8$.
We set $\delta = 1$ and $0.25$ for the first three instances (Ins. 1, Ins. 2 and Ins. 3),
and the last two instances (Ins. 4 and Ins. 5), respectively.
Note that we consider 2 underdose and 1 overdose constraints and their corresponding r.h.s. $b$ and $p$
are shown in the last column of Table~\ref{Rand_data}. We set the $\underline T_v = \bar T_v = 56$ for tumor tissue and $\underline T_v = \bar T_v = 0$ for healthy organ in \eqref{obj}. In addition,
we set $\Phi = 0.2$ for the group sparsity constraint in \eqnok{Group_s}.
%Different discretize accuracy  $\delta$ and problem parameters $b$ and $p$ in constraints \eqref{cvar1} and \eqref{cvar2}. 
%Note that we consider ?? overdose and ?? under dose constraints
%From the choice of the parameters $b$ and $p$ we know, the constraints in instance 3 are stronger than the ones in instance 2, and the constraints in instance 2 are stronger than the ones in instance 1.

We implement in Matlab the CoexCG and CoexDurCG algorithms for structured nonsmooth
problems, and report the computational results 
in Table~\ref{Comp1}.  Here we use $x_N:=(y_N,\tau_N)$, $f(x_N)$ and $\|h(x_N)\|$, respectively, to denote the output solution, 
the objective value and constraint violations.
The CPU times are in seconds on a Macbook Pro with 2.6 GHz 6-Core Intel Core i7 processor.
As shown in Table~\ref{Comp1}, both CoexCG and CoexDurCG exhibit comparable performance 
in terms of objective value, constraint violation and CPU time for different
iteration limit $N$. However, unlike the CoexDurCG algorithm, we
need to rerun CoexCG for all the experiments whenever $N$ changes.

\begin{table}
\footnotesize
\centering
\caption{Data Instances with $\Phi = 0.2$}\label{Rand_data}
 \begin{tabular}{p{2cm}p{2cm}p{2.5cm}cr}
 \hline
 Index & $\#$ of voxels & $\#$ of apertures  & $b_i$ $\&$ $p_i$ \\
 \hline\hline
Ins. 1& 4096 & 460800  & [30,40,200]  $\&$ [0.05,0.05,0.05] \\
 \hline
 Ins. 2& 4096 & 460800  & [40,50,100]  $\&$ [0.01,0.01,0.05] \\
 \hline
 Ins. 3& 4096 & 460800  & [50,60,80]  $\&$ [0.01,0.01,0.01] \\
 \hline
Ins. 4 &262144 & 7372800 & [40,50,100]  $\&$ [0.01,0.01,0.05] \\
 \hline
 Ins. 5& 262144 & 7372800 & [50,60,80]  $\&$ [0.01,0.01,0.01] \\
 \hline\hline
\end{tabular}
\end{table}

\begin{table}
\footnotesize
\centering
\caption{Results for different Instances}\label{Comp1}
 \begin{tabular}{p{1.5cm}p{1cm}p{1.5cm}p{2cm}p{1.5cm}p{1.5cm}p{1.5cm}cr}
 \hline
 \multirow{2}{*}{ Index} &  \multirow{2}{*}{N }& \multicolumn{3}{c}{ CoexCG } &  \multicolumn{3}{c}{CoexDurCG }  \\
 && $f(x_N)$ &  $\|h(x_N)\|$& CPU(s) & $f(x_N)$& $\|h(x_N)\|$ & CPU(s)\\
 \hline\hline
 \multirow{3}{*}{Ins. 1}& 1 &   46.8723 & 1.7237e+03 &  &&& \\
&100 &0.0683  & 0.4234 &34 & 0.0616  & 0.3705 &33 \\
 &1000 &  0.0197 & 0.0319 &323 & 0.0210 & 0.0219 & 327\\
 \hline
  \multirow{3}{*}{Ins. 2}& 1 &   46.8723 & 1.7237e+03 && &&\\
&100 &0.0568 & 0.4424 & 33& 0.0583  & 0.5002& 34 \\
 &1000 &  0.0224 & 0.0426 & 327& 0.0232 & 0.0334 & 339 \\
 \hline
   \multirow{3}{*}{Ins. 3}& 1 &   46.8723 & 1.7237e+03 && &&\\
&100 &0.0625 & 13.7567 & 33 & 0.0604 & 7.3929 & 33 \\
 &1000 &  0.0227 & 0.0514& 332  &0.0226 & 0.0193 &332\\
 \hline
   \multirow{3}{*}{Ins. 4}& 1 &   47.7099 & 8.7850e+03 && &&\\
&100 &0.4643 & 163.3043&1645  &0.4643 & 163.3043 &1645 \\
 &1000 &   0.0398 & 12.1765 &17254& 0.0398 & 12.1765&17356 \\
 \hline
  \multirow{3}{*}{Ins. 5}& 1 &   47.7099 & 8.7850e+03 && &&\\
&100 &  0.4866 & 253.9389 & 1644  & 0.4581 &  206.9143 & 1637\\
 &1000 &  0.0406  &  39.2051 &17146  &0.0417  & 38.6486 &17607 \\
 \hline\hline
\end{tabular}
\end{table}

In order to test our CoexCG and CoexDurCG algorithms, we still want to compare them with some existing algorithm for constrained convex problems, such as ConEx algorithm \cite{boob2019stochastic}. Since the most existing algorithms will require the computation of full gradient and hence get in stuck when dealing with this very high dimensional problems, we generated a very small dimensional problem (dimension of decision variable is 80000) with similar problem formulation for comparison. From Table~\ref{Comp ConEx}, we see the ConEx algorithm still requires the computation of full gradient, and hence has a total around 630 second computational time of the $\hat D$ matrix for every component. And also we can see the ConEx algorithm finally converges to a almost feasible solution with a bit higher objective function value comparing to the solutions of CoexCG and CoexDurCG algorithms. We also implemented the ConEx algorithm to the Instances 1-5, but the algorithm gets in stuck in the first iteration and keeps running forever.
\begin{table}
\centering
\caption{Comparison with ConEx}\label{Comp ConEx}
\begin{tabular}{p{1cm}p{2cm}p{2cm}p{2cm}p{1.5cm}}
\hline
N & Alg & $f(x_N)$ & $\|h(x_N)\|$ & CPU(s) \\
\hline\hline
\multirow{3}{*}{2} & ConEx &0.482 & 32.156& 632.902 \\
& CoexCG &0.495 & 64.738&0.143\\
& CoexDurCG &0.495 &64.738&0.156\\
\hline
\multirow{3}{*}{10} & ConEx &0.311 &6.137 & 633.948 \\
& CoexCG &0.033 &9.381 & 0.692\\
& CoexDurCG &0.074 &8.913&0.654\\
\hline
\multirow{3}{*}{100} & ConEx &0.279 &0.193 & 642.456 \\
& CoexCG &0.010 &6.165 &6.501\\
& CoexDurCG & 0.015&6.384&6.535\\
\hline
\multirow{3}{*}{1000} & ConEx & 0.301&7.392e-04 & 725.477 \\
& CoexCG &0.010 & 6.626&63.958\\
& CoexDurCG &0.022 &6.361&66.661\\
\hline\hline
\end{tabular}
\end{table}

\subsection{Results for real dataset}
In this subsection, we apply CoexDurCG to the real dataset for a patient with prostate cancer ({\tt https://github.com/cerr/CERR/wiki}), and evaluate the generated solution from the clinical point of view. 
Dose volume histogram (DVH), a histogram relating radiation dose to tissue volume in radiation therapy planning, is commonly used as a plan evaluation tool 
to compare doses received by different structures under different plans~\cite{drzymala1991dose, mayles2007handbook}.
%DVHs, a histogram relating radiation dose to tissue volume in radiation therapy planning, are most commonly used as a plan evaluation tool and to compare doses from different plans or to structures \cite{drzymala1991dose, mayles2007handbook}. 
In this prostate dataset, there are totally 10 DVH criteria as follows, PTV56: V56$\geq 95\%$; PTV68: V68$\geq 95\%$, V74.8$\leq 10\%$; Rectum:  V30$\leq 80\%$,  V50$\leq 50\%$, V65$\leq 25\%$; Bladder:  V40$\leq 70\%$, V65$\leq 30\%$; Left femoral head: V50$\leq 1\%$; Right femoral head: V50$\leq 1\%$. 
For this dataset, we have $3,047,040$ voxels, $180$ angles and over $2\times 10^{30}$ potential apertures in each angle. 

Since a smaller number of angles results in shorter treatment duration, we study the quality of the treatment plan generated
when enforcing the group sparsity requirement with different $\Phi$ in \eqref{Group_s}.
In order to balance the scale of the constraint violation, we normalized all the constraints \eqref{cvar1}-\eqref{Group_s} by dividing both sides of the inequalities by the right hand side $b_i$ or $\Phi$. 
The total number of apertures in a typical treatment plan for this dataset would not be greater than $100$. 
Thus, we set the iteration limit to 100 since the CoexDurCG algorithm generates at most one new aperture in each iteration.

Table~\ref{group_spar} shows the number of apertures/angles, objective value and constraints violation for different solutions given different values of $\Phi$. 
Figure~\ref{DVH_comp} plots the DVH performance of the generated treatment plans by presenting how the percentage of voxels in each organ changes over different iterations.
If $\Phi = 1$, the constraint \eqref{Group_s} is redundant and
we obtain a solution with the smallest function value and zero constraints violation, but with the largest number of angles as shown in Table~\ref{group_spar}. 
In addition, the plots in the first column (i.e., parts (a), (d), (g), (j) and (m)) of Figure~\ref{DVH_comp} show that the generated plan satisfies all the DVH criteria. 
Comparing the first two rows in Table~\ref{group_spar}, we see that the solutions remain the same when $\Phi\geq 0.1$. By keeping decreasing $\Phi$, we can obtain solutions with fewer angles. 
Plots in the second column of Figure~\ref{DVH_comp}  shows that most DVH criteria are still satisfied even if the number of angles in the solution reduces from 39 to 8. Moreover, the number of angles can
be decreased to 3 if we are willing to sacrifice certain DVH criteria as we can see from the plots in the third column of Figure~\ref{DVH_comp}. 
\begin{table}
\footnotesize
\centering
\caption{Group Sparsity}\label{group_spar}
 \begin{tabular}{p{2.5cm}p{2.5cm}p{2.5cm}p{2.5cm}cr}
 \hline
  $\Phi$  & $\#$ of apertures  & $\#$ of angles & Obj. Val. & Con. Vio. \\
 \hline\hline
 1& 96 & 39  & 0.0902  &  0   \\
 \hline
 0.1& 96 & 39  & 0.0902  &  0   \\
 \hline
 0.005  & 96 & 8  &  0.1027 & 0.098 \\
 \hline
 0.0005  & 97 & 3  &  0.1357 & 0.0589 \\
 \hline\hline
\end{tabular}
\end{table}

\begin{figure}
\centering
\subfigure[PTV56 when $\Phi = 1$]{
   % \label{fig:subfig:1a} %% label for first subfigure
    \includegraphics[width=2in]{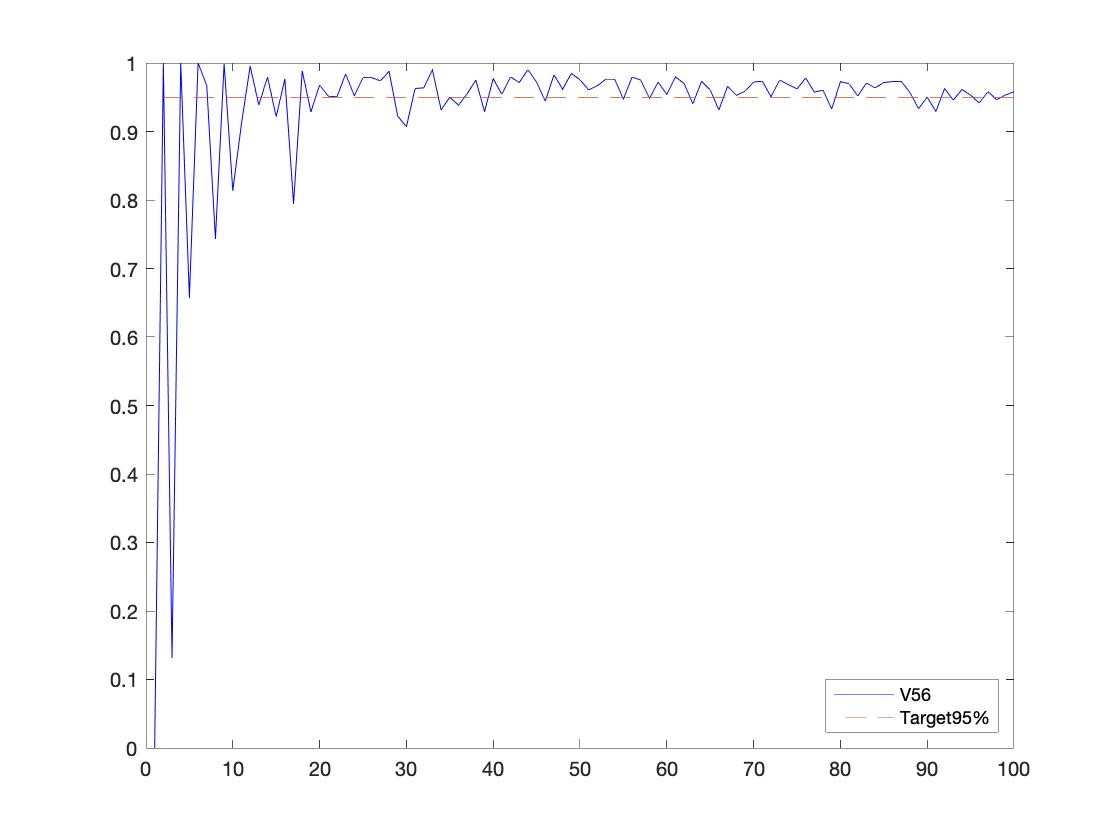}}
  \subfigure[PTV56 when $\Phi = 0.005$]{
%    \label{fig:subfig:1b} %% label for second subfigure
    \includegraphics[width=2in]{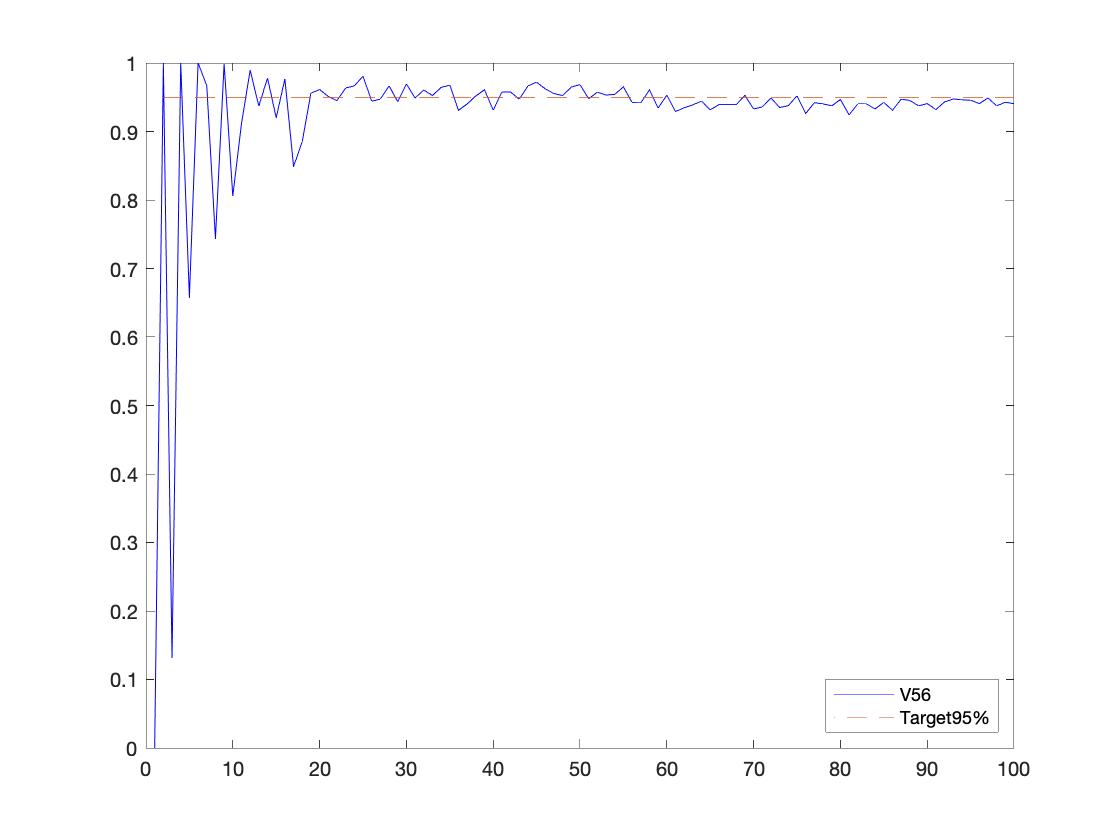}}
  \subfigure[PTV56 when $\Phi = 0.0005$]{
%    \label{fig:subfig:1b} %% label for second subfigure
    \includegraphics[width=2in]{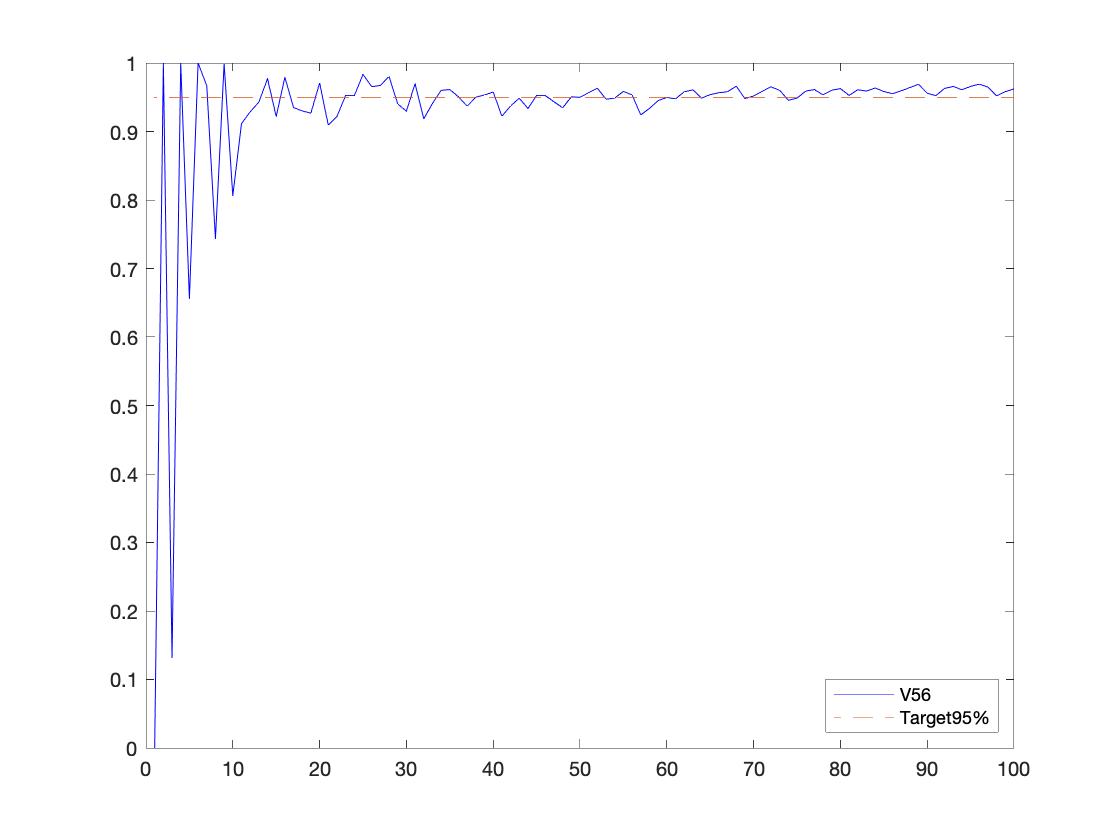}}
\subfigure[PTV68 when $\Phi = 1$]{
%    \label{fig:subfig:1a} %% label for first subfigure
    \includegraphics[width=2in]{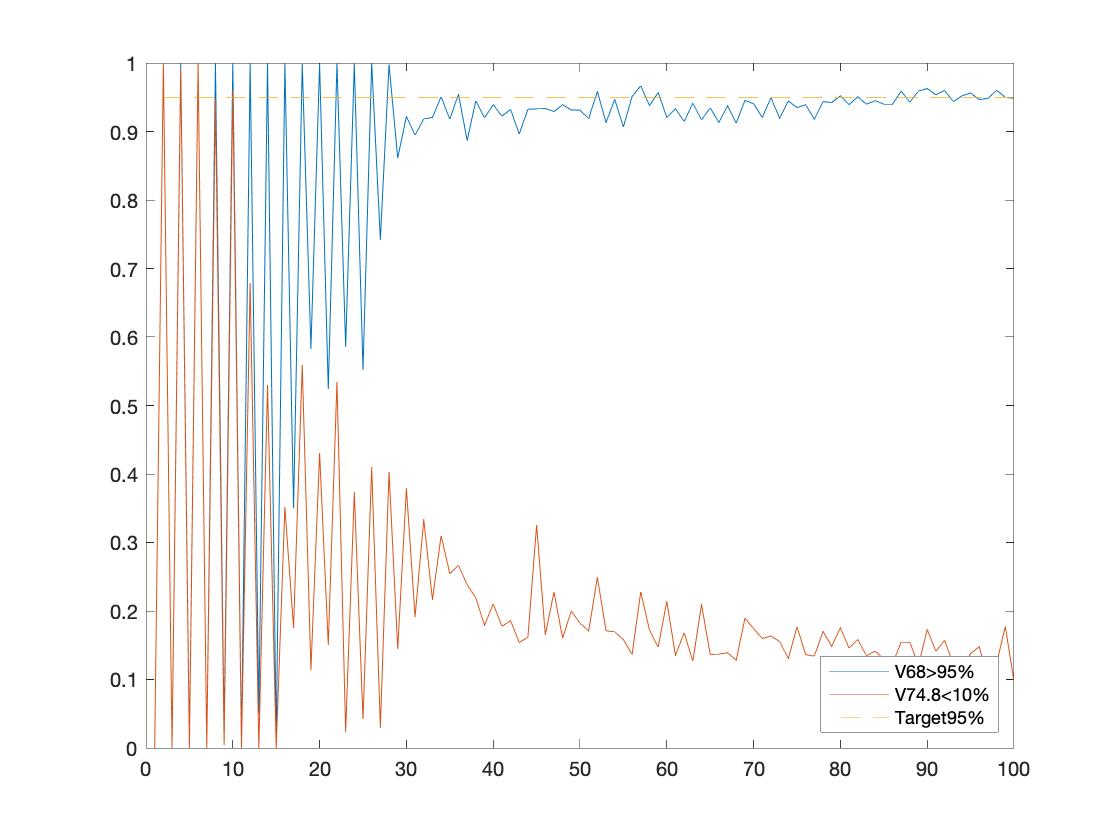}}
  \subfigure[PTV68 when $\Phi = 0.005$]{
%    \label{fig:subfig:1b} %% label for second subfigure
    \includegraphics[width=2in]{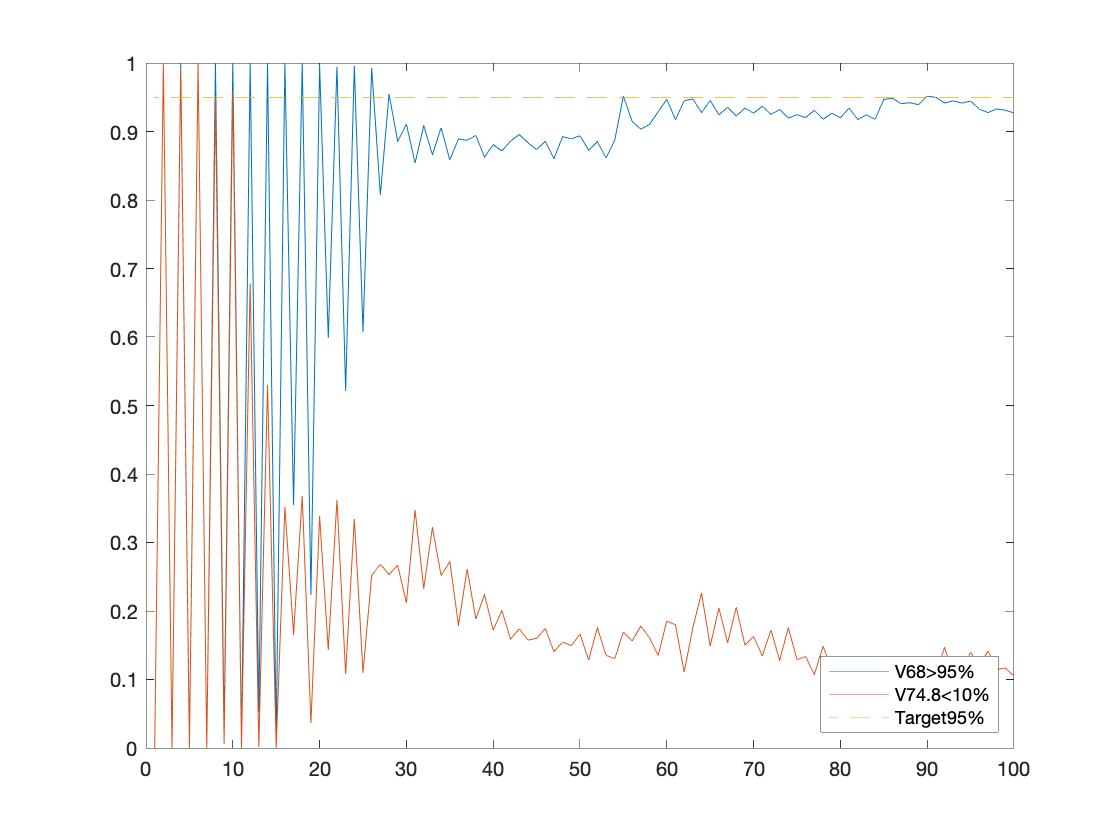}}
  \subfigure[PTV68 when $\Phi = 0.0005$]{
%    \label{fig:subfig:1b} %% label for second subfigure
    \includegraphics[width=2in]{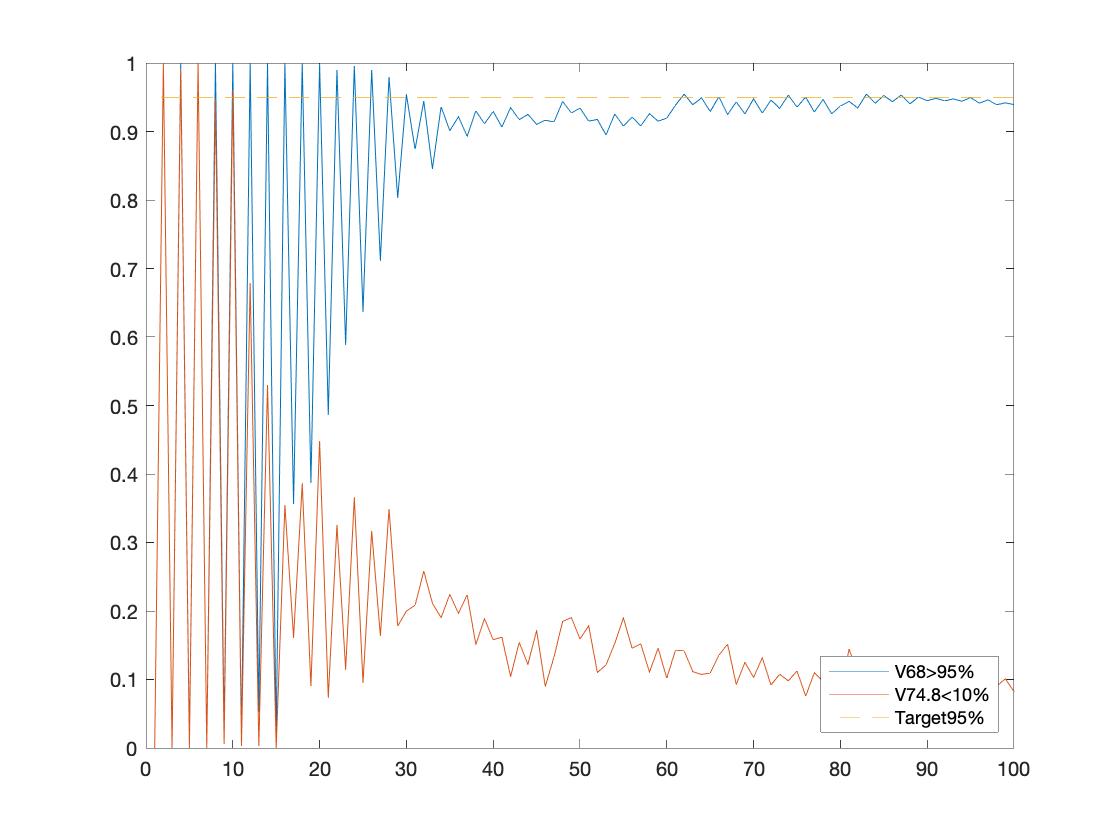}}
\subfigure[Rectum when $\Phi = 1$]{
%    \label{fig:subfig:1a} %% label for first subfigure
    \includegraphics[width=2in]{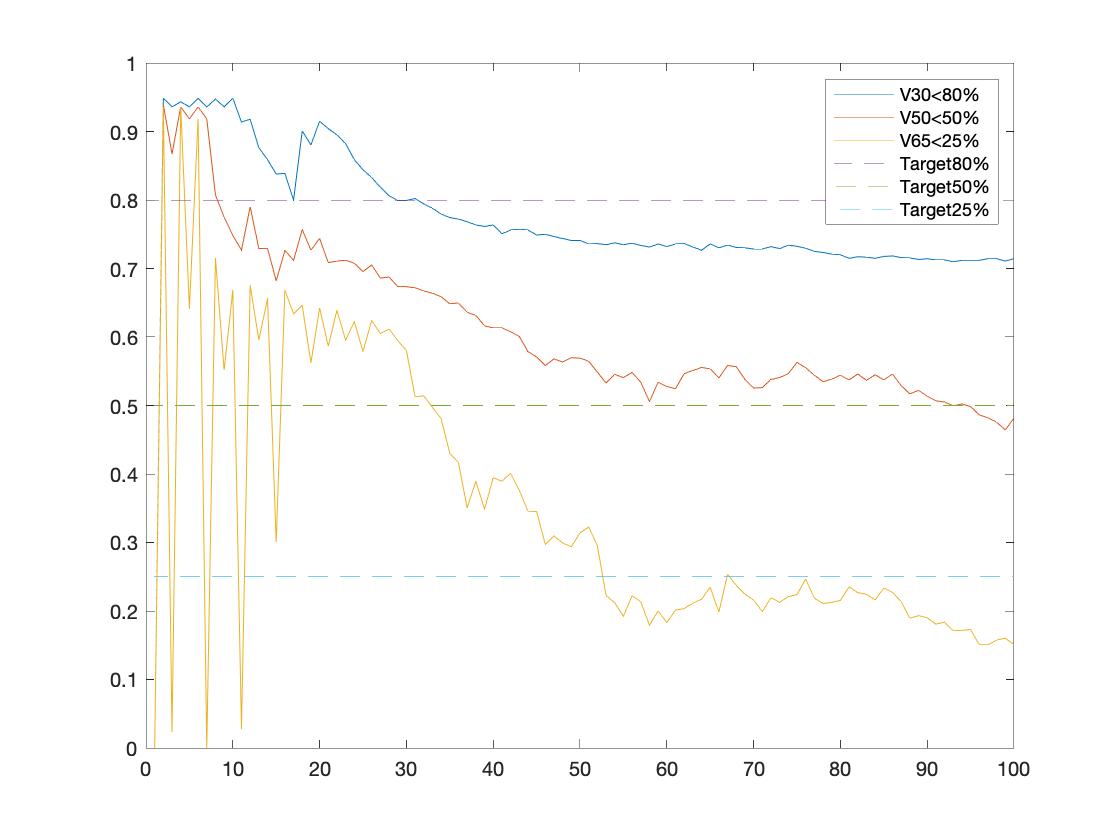}}
  \subfigure[Rectum when $\Phi = 0.005$]{
%    \label{fig:subfig:1b} %% label for second subfigure
    \includegraphics[width=2in]{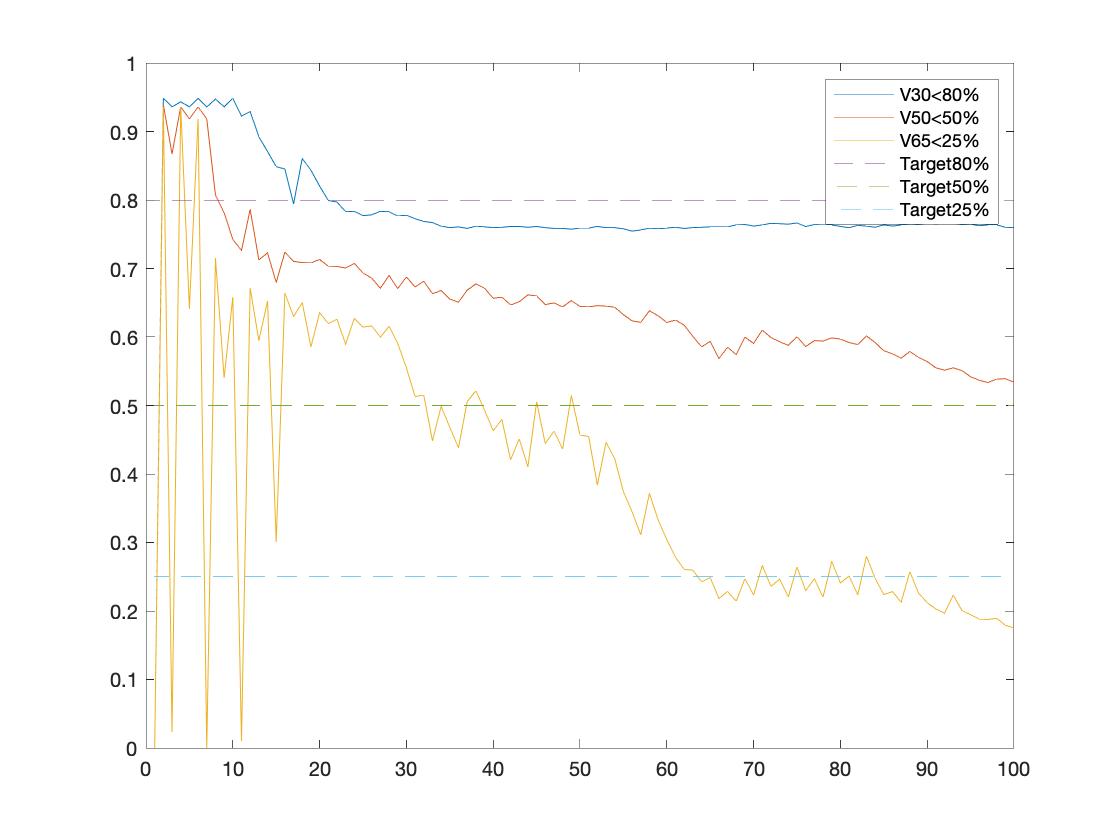}}
  \subfigure[Rectum when $\Phi = 0.0005$]{
%    \label{fig:subfig:1b} %% label for second subfigure
    \includegraphics[width=2in]{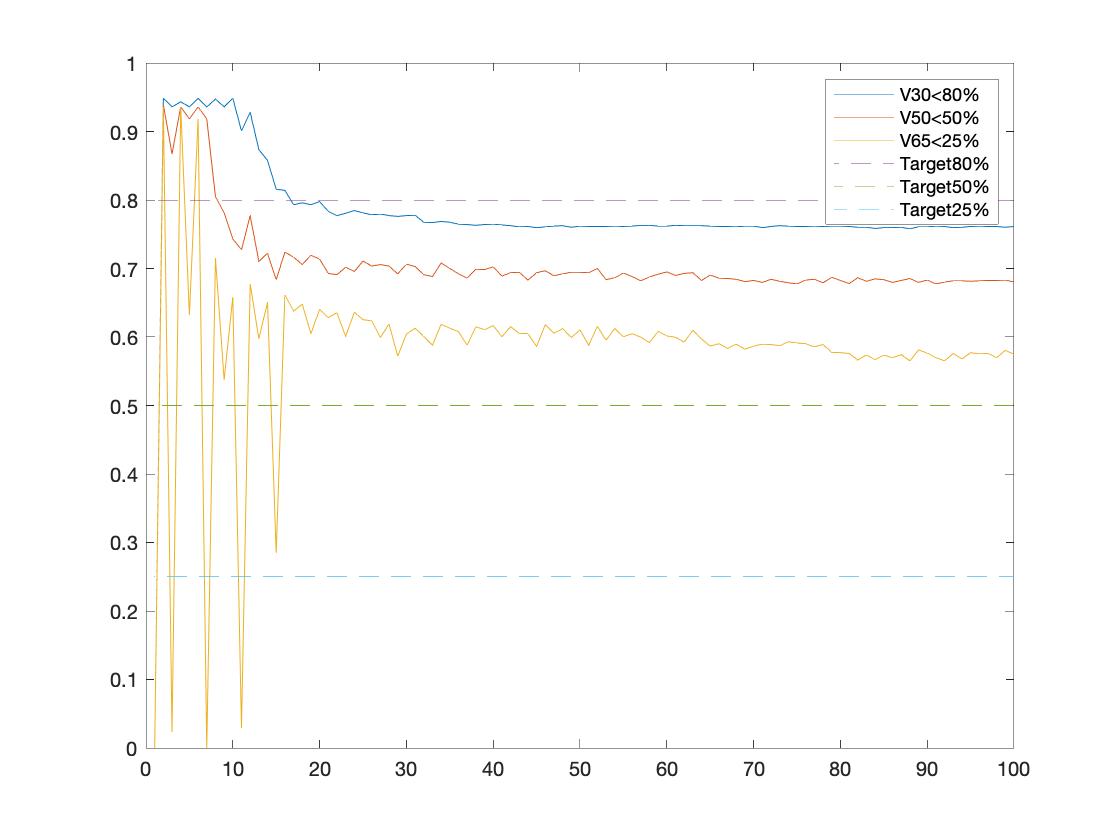}}
    \subfigure[Bladder when $\Phi = 1$]{
%    \label{fig:subfig:1a} %% label for first subfigure
    \includegraphics[width=2in]{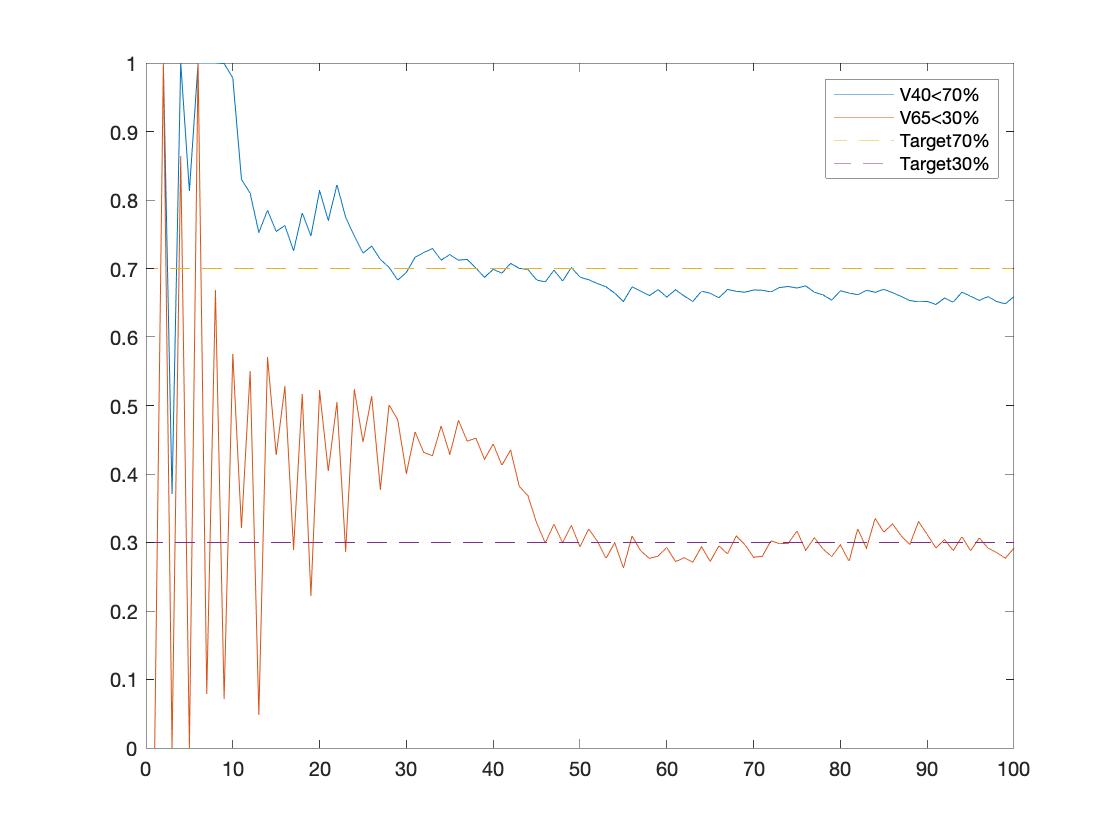}}
  \subfigure[Bladder when $\Phi = 0.005$]{
%    \label{fig:subfig:1b} %% label for second subfigure
    \includegraphics[width=2in]{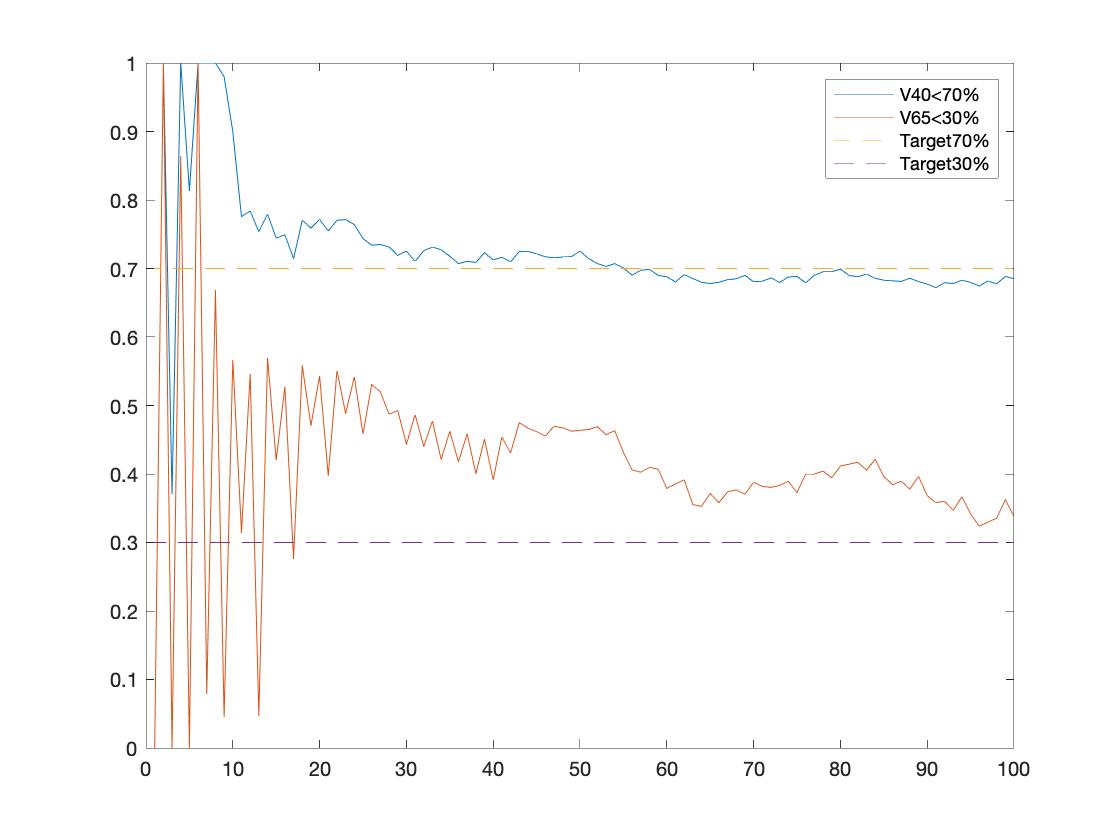}}
  \subfigure[Bladder when $\Phi = 0.0005$]{
%    \label{fig:subfig:1b} %% label for second subfigure
    \includegraphics[width=2in]{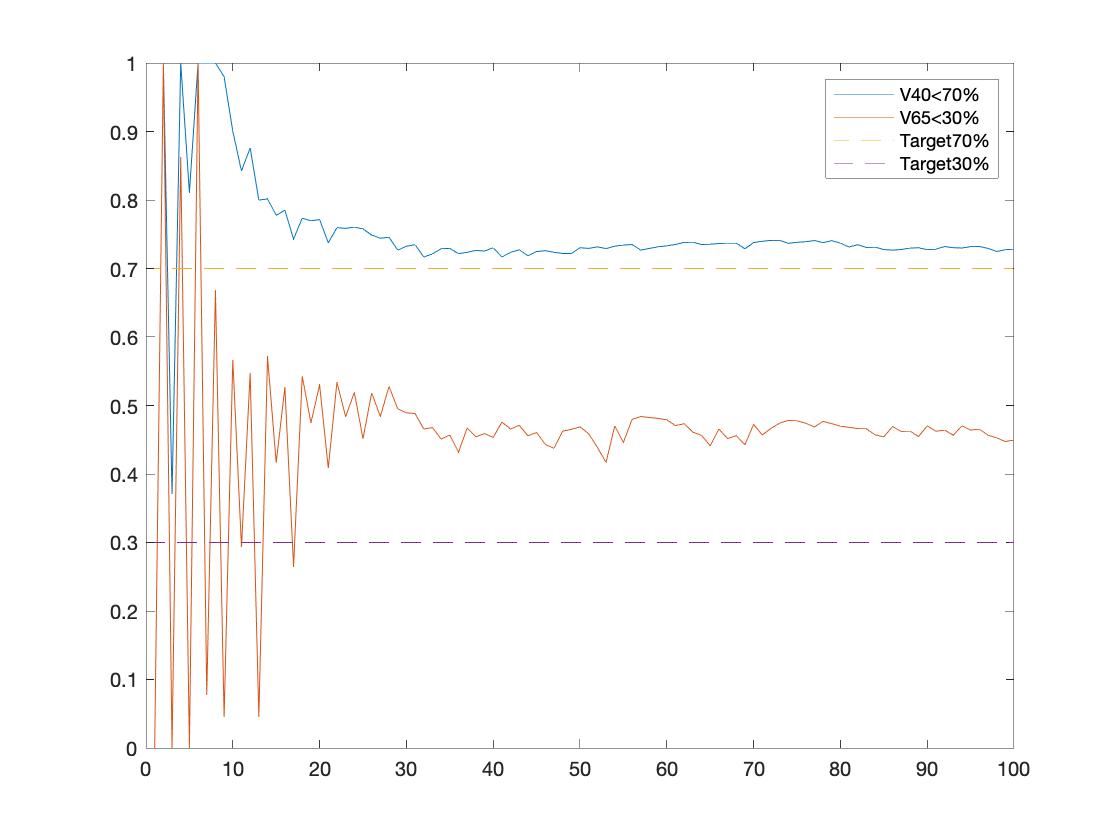}}
    \subfigure[Lt. $\&$ Rt. when $\Phi = 1$]{
%    \label{fig:subfig:1a} %% label for first subfigure
    \includegraphics[width=2in]{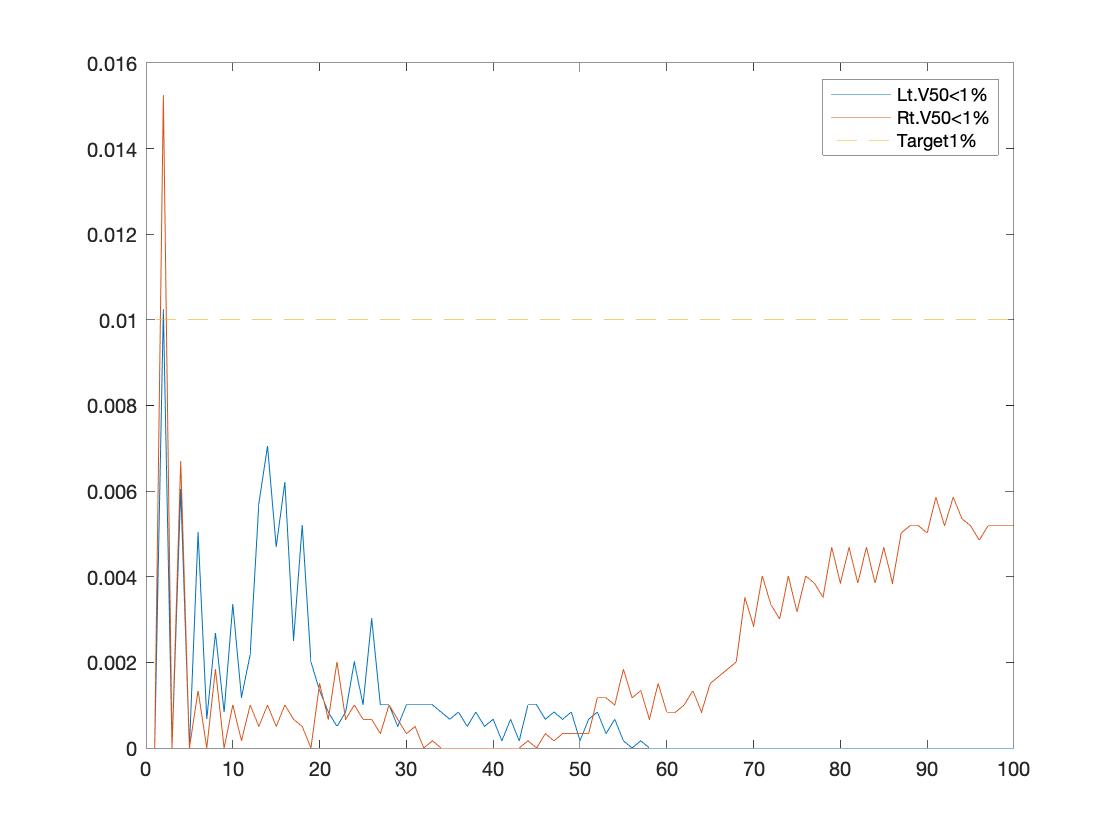}}
  \subfigure[Lt. $\&$ Rt. when $\Phi = 0.005$]{
%    \label{fig:subfig:1b} %% label for second subfigure
    \includegraphics[width=2in]{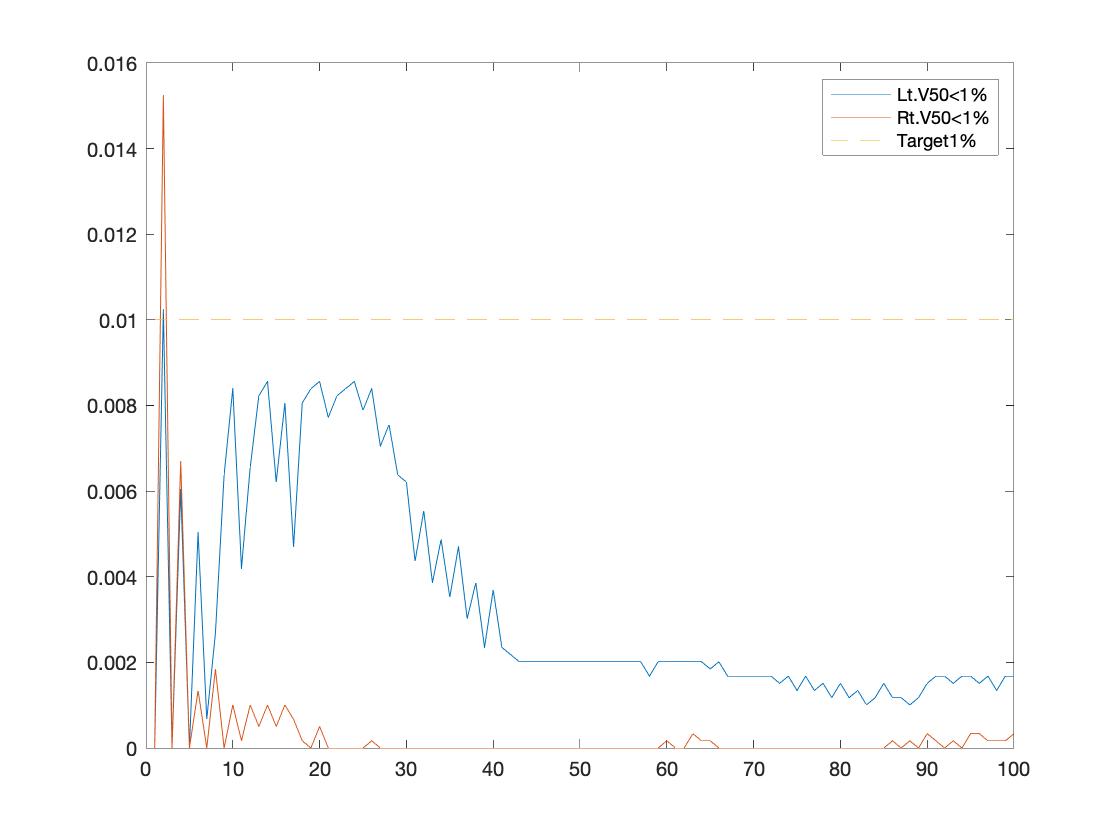}}
  \subfigure[Lt. $\&$ Rt. when $\Phi = 0.0005$]{
%    \label{fig:subfig:1b} %% label for second subfigure
    \includegraphics[width=2in]{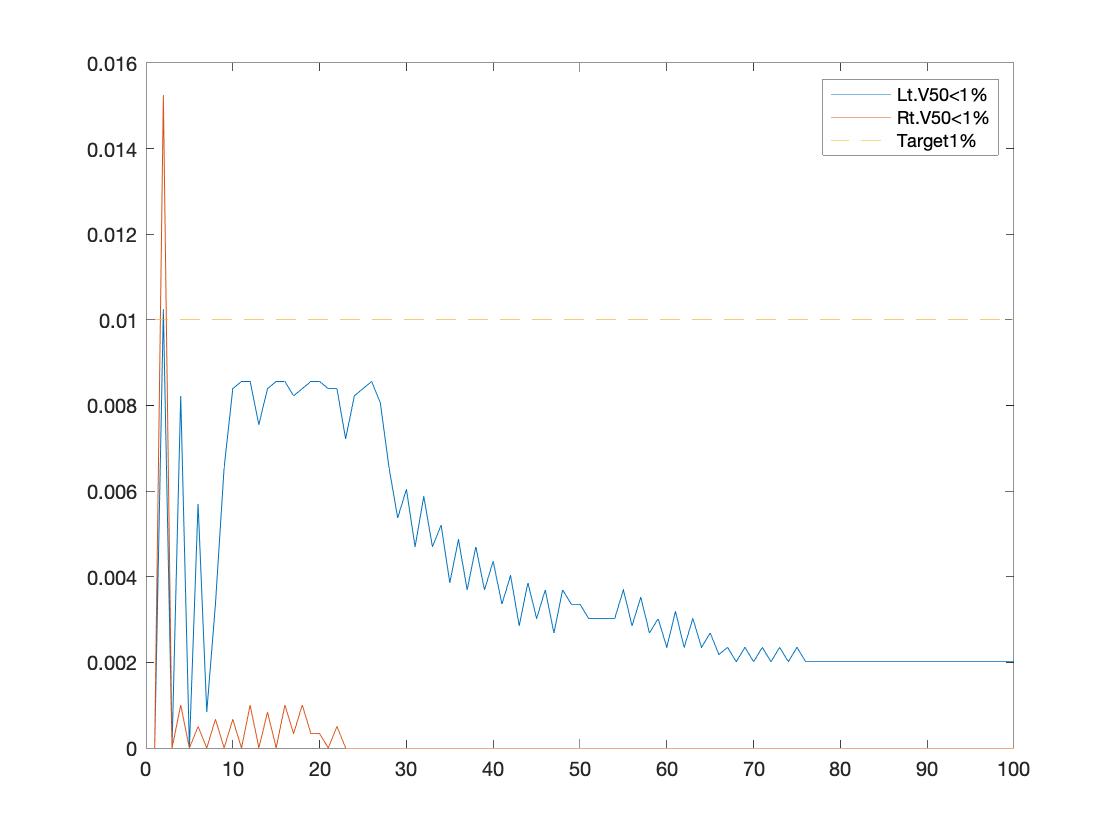}}
\caption{Percentage of voxels in different organs}\label{DVH_comp}
\end{figure}

\section{Concluding Remarks} \label{sec_remark}
In this paper, we propose new constraint-extrapolated conditional gradient (CoexCG) methods for solving general 
convex optimization problems with function constraints. These methods
require only linear optimization rather than projection over the convex set $X$. We 
establish the ${\cal O}(1/\epsilon^2)$ iteration complexity for CoexCG
and show that the same complexity still holds even if the objective or constraint functions
are nonsmooth with certain structures. We further present novel dual regularized algorithms that
do not require us to fix the number of iterations a priori and show that they can attain
complexity bounds similar to CoexCG. Effectiveness of these methods
are demonstrated for solving a challenging function constrained convex optimization problems arising from IMRT
treatment planning. 

It seems to be possible to use some ideas from the conditional gradient sliding methods~\cite{lan2014conditional} 
to improve the number of gradient computation of $f$ and $h$, as well as the
operator evaluation of $g$. However, the conditional gradient sliding type methods would require
us to compute and store the full gradient information.
For the IMRT treatment planning problem, it is impossible to compute the full gradient since its dimension increases exponentially with the size of aperture.
Nevertheless, incorporating the idea of conditional gradient sliding for solving problems
with function constraints will be an interesting topic for future research. 

%\newpage
\bibliographystyle{abbrv}
\bibliography{glan-bib}
\end{document}